\documentclass[12pt, reqno]{amsart}

\title{Auslander algebras, flag combinatorics and quantum flag varieties}
\author{Bernt Tore Jensen}
\address{BTJ: Department of Mathematical Sciences, Norwegian University of Science and Technology, 
Gj\o vik, Teknologvn. 22, 2815 Gj\o vik, Norway}
\email{bernt.jensen@ntnu.no}
\author{Xiuping Su}
\address{XS: Department of Mathematical Sciences, University of Bath, BA2 7AY, U. K.}
\email{xs214@bath.ac.uk}

\usepackage[a4paper, hmargin=2.75cm, vmargin=2.25cm]{geometry}
\usepackage{amssymb,amsthm,amsfonts,amscd,array}
\usepackage{tikz}
\usetikzlibrary{positioning} 
\usetikzlibrary{calc,decorations.markings,arrows.meta,arrows}
  \tikzset{equals/.style={double=none, double distance=2pt}, 
     cdmap/.style={black, thick, -angle 60}} 
\usepackage{color}

\usepackage{amssymb,amsthm,array}

\usepackage{etex} 
\usepackage[all]{xy}
\usepackage{ifthen}

\numberwithin{equation}{section}
\usepackage{amssymb,amsthm,array}

\usepackage{lmodern}
\usepackage{verbatim}

\usepackage{lmodern}
\usepackage{verbatim}

\usepackage{amsthm}

\usepackage[linktocpage,pdfstartview=FitH,colorlinks=true,
linkcolor=blue,citecolor=magenta]{hyperref}

\usepackage{cleveref}

\theoremstyle{definition}
\newtheorem{theorem}{Theorem}[section]
\newtheorem{proposition}[theorem]{Proposition}
\newtheorem{lemma}[theorem]{Lemma}
\newtheorem{corollary}[theorem]{Corollary}
\newtheorem{remark}[theorem]{Remark}
\newtheorem{example}[theorem]{Example}

\theoremstyle{definition}
\newtheorem{definition}[theorem]{Definition}

\numberwithin{equation}{section}
\numberwithin{figure}{section}
\newtheorem{Theorem}{Theorem}

\newcommand{\mat}[1]{\operatorname{\mathsf{#1}}}
\newcommand{\cm}{\mathsf{m}} 
\newcommand{\cn}{\mathsf{n}}

\newcommand{\proj}{\operatorname{proj}}
\newcommand{\flag}{\operatorname{Fl}}

\newcommand{\ZZ}{\mathbb{Z}}

\newcommand{\CC}{\mathbb{C}}

\newcommand{\aus}{D}

\newcommand{\ol}[1]{\overline{#1}}

\newcommand{\qcl}[1]{\mathcal{A}_q{(#1)}}

\newcommand{\qminor}[1]{\Delta^q_{#1}}

\newcommand{\setminors}{\backslash}

\newcommand{\Dual}{\mathsf{D}}

\newcommand{\stHom}{\underline{\mathrm{Hom}}}
\newcommand{\Hom}{\mathrm{Hom}}

\newcommand{\gl}{\operatorname{GL}}

\newcommand{\End}{\operatorname{End}}
\newcommand{\sub}{\operatorname{Sub}}
\renewcommand{\mod}{\operatorname{mod}}
\newcommand{\fac}{\operatorname{Fac}}
\newcommand{\Gr}{\operatorname{Gr}}

\newcommand{\rad}{\operatorname{rad}}
\newcommand{\e}{\operatorname{e}}

\newcommand{\dvector}{{\underline{\operatorname{dim}}}}
\newcommand{\dv}[1]{{\bf #1}}
\newcommand{\minus}{\backslash}

\newcommand{\add}{\operatorname{add}}

\newcommand{\supp}{\operatorname{Supp}}
\newcommand{\Ext}{\operatorname{Ext}}

\newcommand{\ext}{\operatorname{ext}}
\newcommand{\cok}{\operatorname{Cok}}
\renewcommand{\ker}{\operatorname{Ker}}
\newcommand{\im}{\operatorname{Im}}
\renewcommand{\top}{\operatorname{top}}
\newcommand{\soc}{\operatorname{soc}}

\newcommand{\charT}{\mathbb{T}}

\newcommand{\lra}{\longrightarrow}

\newcommand{\ra}{\rightarrow}
\newcommand{\sdp}{\times\kern-.2em\vrule height1.1ex depth-.05ex}
\newcommand{\epi}{\lra \kern-.8em\ra}

\newcommand{\g}{\hat{g}}

\renewcommand{\l}{\mathfrak{l}}
\newcommand{\n}{\mathfrak{n}}
\newcommand{\f}{\hat{f}}
\newcommand{\M}{\hat{M}}
\newcommand{\N}{\hat{N}}

\renewcommand{\L}{\hat{L}}

\newcommand{\T}{\hat{T}}

\newcommand{\funF}{\mathsf{F}}

\newcommand{\WS}{\mathcal{S}}
\newcommand{\maxWS}{\mathcal{S}}
\newcommand{\maxNC}{\maxWS}
\newcommand{\Irr}{\operatorname{Irr}}

\newcommand{\grothE}{K(\Dfiltered)_{\mathcal{E}}}
\newcommand{\grothEJ}{K(\Dfiltered(\setJ))_{\mathcal{E}}}

\newcommand{\Dfiltered}{\mathcal{F}_{\Delta}}
\newcommand{\pad}{\operatorname{pad}}

\newcommand{\verteq}{\rotatebox{90}{$\,=$}}
\newcommand{\wtl}{\mathfrak{L}}

\newcommand{\X}{\hat{X}}

\newcommand{\projd}{\operatorname{proj.dim}}

\newcommand{\CM}{\operatorname{CM}}

\newcommand{\projfun}{\pi}

\newcommand{\matadj}{\mat{L}^{\mathrm{ad}}}

\newcommand{\row}{\mathrm{row}}

\newcommand{\rootlat}{\mathcal{R}}
\newcommand{\fln}{\hat{\rho}}
\newcommand{\charfl}{\hat{\Phi}}

\newcommand{\setJ}{\mathsf{J}}
\newcommand{\setK}{\mathsf{K}}
\newcommand{\setL}{\mathsf{L}}
\newcommand{\cat}{\mathcal{C}}

\newcommand{\qflag}{\CC_q[\flag]}
\newcommand{\pqflag}{\CC_q[\flag(\setJ)]}
\newcommand{\qnil}{A_q(\mathfrak{n})}

\newcommand{\qnilw}{A_q(\mathfrak{n}_\setK)}

\newcommand{\Cq}{\CC[q^{\frac{1}{2}}, q^{-\frac{1}{2}}]}
\newcommand{\bfi}{\mathbf{i}}
\newcommand{\bfj}{\mathbf{j}}

\newcommand{\intw}[2]{\left[#1, #2 \right]_{\mathrm{w}}}
\newcommand{\FIng}{\mathsf{I}}

\newcommand{\mapfl}{\delta}

\newcommand{\gr}{\operatorname{Gr}}
\newcommand{\typeA}{\mathbb{A}}
\newcommand{\inv}{\mathrm{inv}}

\newcommand{\qclusalg}{ \qcl{\mat{B}, \mat{L}}} 
\newcommand{\qclusalgT}{ \qcl{\mat{B}_T, \mat{L}_T}} 
\newcommand{\qclusalgbar}{ \qcl{\ol{\mat{B}}, \ol{\mat{L}}}}

\newcommand{\charC}{\kappa}
\newcommand{\Dnabla}{\mathcal{F}_{\nabla}}

\newcommand{\GP}{\operatorname{GP}}
\newcommand{\qnu}{\CC[\nu]_q}
\newcommand{\liftV}{U}

\normalbaselines
\begin{document}
\maketitle
\begin{abstract} 
Let $D$ be the Auslander algebra of $\CC[t]/(t^n)$. This is a quasi-hereditary algebra with 
the subcategory of good modules denoted by $\Dfiltered$. For any $\setJ\subseteq[1, n-1]$, 
we construct a subcategory $\Dfiltered(\setJ)$ of $\Dfiltered$ with an exact structure 
$\mathcal{E}$. We show that, under $\mathcal{E}$, $\Dfiltered(\setJ)$ is  Frobenius 
stably 2-Calabi--Yau and admits a cluster structure.
Furthermore, $\Dfiltered(\setJ)$  provides an (additive) categorification of the cluster structure on the coordinate 
ring $\CC[\flag(\setJ)]$ of the (partial) flag variety $\flag(\setJ)$.

We apply $\Dfiltered(\setJ)$ to study  flag combinatorics and the quantum cluster 
structure on $\flag(\setJ)$. Intriguingly,  weak and strong separation, 
defined by Leclerc--Zelevinsky, admits homological characterisations: weak separation is 
characterised by the vanishing of the extension groups $\ext^1(-, -)$ 
under $\mathcal{E}$, while strong separation is detected by the vanishing of the extension groups $\Ext^1(-,-)$ 
inherited from $\mod D$. 
Furthermore, we give a categorical interpretation of the quasi-commutation rules of 
quantum minors. In addition, we 
show that the combinatorial operations of flips and geometric exchanges 
correspond to certain mutations of cluster tilting objects in $\Dfiltered(\setJ)$, and that
the resulting relations are quantum Pl\"{u}cker and incidence relations. 

Let $\setK=[n-1]\backslash \setJ$. By carefully choosing a reduced word of 
$w=w_0w_0^{\setK}$,  we construct modules giving rise to Geiss--Leclerc--Schr\"{o}er's 
dual PBW generators for $\qnilw$ such that 
they all have a simple top and therefore their duals  lift to rank one modules 
in $\Dfiltered(\setJ)$. 
In particular, the associated cluster variables are quantum minors.
Finally, we show  that $\CC_q[\flag(\setJ)]$ is a quantum 
cluster algebra over $\Cq$. 
Moreover, the difference between the quasi-commutation rules in $\CC_q[\flag(\setJ)]$
and $\qnilw$ is governed by a bilinear form defined on the Grothendieck group  $\grothEJ$. 
\end{abstract}

\setcounter{tocdepth}{1}
\tableofcontents

\section{Introduction}
Cluster algebras were introduced by Fomin--Zelevinsky in \cite{FZ1}. Since then much 
effort has been invested into uncovering  (quantum) cluster structures on 
coordinate rings.
Scott \cite{Scott} discovered a cluster structure on the coordinate ring $\CC[\gr(k, n)]$ 
of the Grassmannian $\gr(k, n)$, extending Fomin--Zelevinsky's result on $\CC[\gr(2, n)]$ \cite{FZ1}. 
Thereafter, many more cluster structures on (quantum) coordinate rings have been 
revealed via categorification. 
Using subcategories of the module categories of preprojective algebras, 
Geiss--Leclerc--Schr\"{o}er \cite{GLS08, GLS11} 
discovered a cluster structure on the coordinate rings of unipotent cells. 
Leclerc \cite{Lec} further extended results in \cite{GLS11} to study cluster 
structures on Richardson varieties.
Jensen--King--Su \cite{JKS1} constructed a Grassmannian cluster category $\CM C$ 
and used it to provide an additive categorification of the cluster algebra 
$\CC[\gr(k, n)]$. This construction, along with the associated categorification,  
was generalised by Iyama--Demonet \cite{DI} to obtain the cluster structure on flag 
varieties of finite simply-laced types. More recently, Jensen--Riordan--Su applied 
Grassmannian cluster subcategories of the form $\GP B$ to study cluster structures on 
positroid varieties \cite{JRS}. Their work offers a new proof of Galashin--Lam's Theorem \cite{GL}, 
which confirms Muller--Speyer's conjectures \cite{MS}. Prior to \cite{JRS}, for the connected case, Pressland 
gave a categorification of the cluster structure arising from Galashin--Lam's result.

In the quantum cluster setting, Geiss--Leclerc--Schr\"{o}er \cite{GLS13} applied subcategories 
$\mathcal{C}_w$ of module categories of preprojective algebras 
introduced in \cite{GLS11} to study the quantum coordinate ring of the 
unipotent cell $N_w$ associated to a Weyl group element $w$. They proved  that 
\[
\mathcal{A}_q(\mathcal{C}_w)\otimes_{\Cq} \CC(q^{\frac{1}{2}}) \cong 
\mathcal{A}_q(n(w))\otimes_{\Cq} \CC(q^{\frac{1}{2}}),
\]
where $\mathcal{A}_q(n(w))$ is the integral form of the quantum deformation of $\CC[N(w)]$ 
and $\mathcal{A}_q(\mathcal{C}_w)$ is a quantum cluster algebra constructed from 
$\mathcal{C}_w$. This result was improved by Kang--Kashiwara--Kim--Oh \cite[Cor. 11.2.8]{KKKO} 
to an isomorphism over $\ZZ[q, q^{-1}]$ (see also Goodearl--Yakimov's work \cite[Thm. B]{GY}).

Working over $\CC(q^{\frac{1}{2}})$, Grabowski--Launois \cite{GL} showed that the 
quantum Grassmannian coordinate ring $\CC_q[\gr(k, n)]$ is a quantum cluster algebra, 
via homogenisation. Jensen--King--Su \cite{JKS2}, exploiting the advantage of categorification,  
constructed quantum seeds from cluster tilting objects. For cluster tilting objects of rank one modules, 
they showed that the quantum cluster algebra  is isomorphic to  $\CC_q[\gr(k, n)]$ over 
$\CC(q^{\frac{1}{2}})$, thereby providing an alternative proof of the result in \cite{GL}.

 In this paper,  we construct and apply subcategories $\Dfiltered(\setJ)$ of 
 good modules of Auslander 
algebras to study the (quantum) cluster structures on flag varieties $\flag(\setJ)$ of 
type $\mathbb{A}$, 
where $\setJ\subseteq [n-1]=\{1, \dots, n-1\}$. 
In the classical case, a cluster character is explicitly constructed. In the quantum case, 
we emphasise that 
the relation between their quantum coordinate rings $\CC_q[\flag(\setJ)]$ and 
$\qnilw$ is not straightforward (see for instance \cite{GL, JKS2}), 
unlike their classical counterparts where $\CC[N_\setK]$ is a quotient 
of  $\CC[\flag(\setJ)]$ by setting some minors to be 1, where $\setK$ is the complement 
of $\setJ$.
In particular, the quasi-commutation rules of the corresponding minors are different. 
We show that $\CC_q[\flag(\setJ)]$ is a cluster algebra over $\Cq$ and 
formulate a precise relation between the quantum cluster structures on $N_\setK$ and $\flag(\setJ)$ via categorification. 

\subsection{Auslander algebras and the Dlab--Ringel Theorem} 
Throughout, \emph{all modules are assumed to be finite dimensional, unless otherwise stated}. 
To introduce our category $\Dfiltered(\setJ)$, 
we briefly recall relevant definitions and results \cite{BHRR, DR6} concerning
 the Auslander algebra $D=D_n$ of the truncated polynomial ring
$\CC[t]/(t^{n})$. The algebra $D$ is  quasi-hereditary and so it admits a 
highest weight category $\Dfiltered$ of good $D$-modules, which are filtered 
by the \emph{standard $D$-modules} $\Delta_i$ 
(also known as Verma modules), one for each vertex $i$.
Br\"{u}stle--Hille--Ringel--R\"{o}hrle (BHRR) proved that 
an indecomposable module in $\Dfiltered$ is rigid if and only if it is isomorphic to a 
submodule of the largest indecomposable projective module $D_0$ at vertex 0 
and such a submodule is uniquely determined by its 
 $\Delta$-support (see \S \ref{sec:2.2}). Therefore, rigid indecomposable modules in $\Dfiltered$
 are classified by  non-empty subsets of $[n]$ \cite{BHRR}. These 
modules are referred to as  {\em rank one modules} and denoted by $M_I$, where $I$
is the $\Delta$-support of $M_I$.

We highlight several surprising applications of the Auslander algebra and their generalisations (e.g. see \cite{HV}). In Lie theory, particularly in the study of nilpotent orbits, 
BHRR 
used rigid good $D$-modules to construct Richardson elements in the nilpotent radical 
of parabolic subalgebras in type $\mathbb{A}$ (see also \cite{JS1}). In joint work with Yu,
we extended BHRR's construction to establish the existence of Richardson elements for 
seaweed Lie algebras of type $\mathbb{A}$ \cite{JSY}, and subsequently for all Dynkin types
\cite{JS2}. In a different direction, Baur--Erdmann--Parker used 
the Auslander algebra of $\CC[t]/(t^n)\rtimes C_2$ to construct Richardson elements
for parabolic subgroups in $O_N$~\cite{BEP}.
In this paper, we present another remarkable application of the Auslander algebra $D$ 
by exploring its connections with flag combinatorics and the (quantum) cluster structures on flag varieties.

Let $\charT=\oplus_{i=1}^{n}\charT_i$ be the characteristic tilting $D$-module  \cite{DR6}.
Define $\eta X$ to be the trace of $\charT$ in $X$.
Let $\Pi$ be the preprojective algebra of type $\mathbb{A}_{n-1}$ and define
\[
\projfun: \Dfiltered \to \mod\Pi, ~  X\mapsto X/\eta X.
\]
Dlab--Ringel \cite[Thm. 3]{DR6} showed that $\projfun$ induces an equivalence between
the additive categories, 
$\Dfiltered/\charT $ and $\mod \Pi$.

In the remainder of this section, 
we outline the main results of this paper, following the order of the exposition.

\subsection{Frobenius stably 2-Calabi--Yau (2-CY) subcategories $\Dfiltered(\setJ)$ of $\Dfiltered$} 
Let \[\charT_\setJ=\oplus_{j\in \setJ}\charT_j.\]
Define $\Dfiltered(\setJ)$ to be the full subcategory of $\Dfiltered$, whose objects $X$ have 
their trace $\eta_\charT X$ in $\add \charT_\setJ$.
We show that the restriction of $\projfun$ to $\Dfiltered(\setJ)$ induces an equivalence 
\[\Dfiltered(\setJ)/\add \charT_{\setJ} \rightarrow \sub Q_{\setJ},\]
which is usually not exact. 
We define an exact structure $\mathcal{E}$ on $\Dfiltered(\setJ)$ such that 
an exact sequence is contained in  $\mathcal{E}$  if and only if it remains exact
under $\eta_{\charT}$. Denote by $\ext^1(M, N)$ the extension 
group of $M, N\in \Dfiltered(\setJ)$ with respect to  $\mathcal{E}$.

\begin{Theorem} [Theorems~\ref{Thm:fullexact}, \ref{Thm:Frob2CY},  \ref{thm:liftARseq}, 
\ref{thm:clusterstructure}] 
The following are true.
\begin{itemize}
\item[(1)] The functor
$
\projfun\colon (\Dfiltered(\setJ),\mathcal{E}) \to \sub Q_{\setJ}, M\mapsto M/\eta_\charT M 
$
is dense, full, exact and  induces an isomorphism on the extension groups
\[
\ext^1(M,N)\cong \Ext^1(\projfun M, \projfun N).
\]
\item[(2)] The category 
$(\Dfiltered(\setJ),\mathcal{E})$ is a Frobenius stably 
$2$-CY category.
\item[(3)] The AR-sequences in $\Dfiltered(\setJ)$ are lifts of the AR-sequences in $\sub Q_{\setJ}$.
\item[(4)] The category $(\Dfiltered(\setJ),\mathcal{E})$ admits a cluster structure. 
\end{itemize}
\end{Theorem}

\subsection{Categorification of the cluster structure on flag varieties}
We identify the Grothendieck group $K(\mod \Pi)$  with the root lattice $\rootlat$ of $G=\gl_n$ 
such that the classes of the simple modules 
$S_i$ are simple roots; the Grothendieck group $K(\Dfiltered)$ with the weight lattice $\wtl$ of $G$
 such that the classes $[\charT_i]$ 
are the fundamental weights. Denote by $\wtl_\setJ$ the sublattice generated by 
the fundamental weights $\omega_j$ with $j\in \setJ$, and by $\grothEJ$ the 
Grothendieck group of $(\Dfiltered(\setJ), \mathcal{E})$. Then $\wtl_\setJ=K(\add \charT_\setJ)$.
We have the following isomorphism (\Cref{Lem:isom1}),
$$
\sigma\colon \grothEJ \rightarrow K(\sub Q_{\setJ})\oplus K(\add \charT_{\setJ}), ~
[M]\mapsto [\projfun M] + [\eta M]
$$
Denote the coordinate ring of the flag variety by $\CC[\flag(\setJ)]$. 
The coordinate ring has the following decomposition into graded components  \eqref{eq:gradingE}, 
\[
\CC[\flag(\setJ)]=\bigoplus_{(d, \lambda)} L(\lambda)_{d+\lambda}
=\bigoplus_{\omega}\CC[\flag(\setJ)]_{\omega}, 
\]
where $L(\lambda)$ is the irreducible representation with 
 highest weight $\lambda$, $(d, \lambda)\in \rootlat\oplus \wtl_\setJ$,
$w\in \grothEJ$ and $\sigma(w)=(d, \lambda)$.

We construct a $\grothEJ$-graded embedding (\ref{eq:rhohat}), 
\begin{equation*}
\fln\colon \CC[{\flag(\setJ)}]\rightarrow \CC[N_{\setK}][K(\add \charT_{\setJ})]. 
\end{equation*}
This embedding enables us to construct 
a cluster character on $\Dfiltered(\setJ)$, 
$$
\charfl\colon \Dfiltered(\setJ) \rightarrow \CC[\flag(\setJ)],
$$ 
from the Geiss--Leclerc--Schr\"{o}er's 
cluster character $\psi\colon \sub Q_\setJ\to \CC[N_\setK]$ (see \cite{GLS08}).
We denote the flag minor indexed by $I\subseteq[n]$ by $\Delta_I$.

\begin{Theorem}[\Cref{Thm:classiccat}] 
The map $\charfl$ is a cluster 
character and gives an additive categorification of the cluster algebra structure
on $\CC[\flag(\setJ)]$. Moreover, 
\begin{itemize}
\item[(1)] $\charfl(M)$ has degree $[M]\in \grothEJ$,
\item[(2)] $\charfl(M)$ for all $M\in \Dfiltered(\setJ)$ span $\CC[\flag(\setJ)]$,
\item[(3)] $\charfl(M_I) = \Delta_I$.
\end{itemize}
\end{Theorem}
\noindent This approach of categorification via $\fln$ is different from the homogenisation approach in 
\cite{GLS08, JKS1} and can be quantised to show that $\pqflag$ 
 admits the quantum cluster structure.
Using \cite[Thm. 4]{GLS12}, we show that   for any $M\in \Dfiltered(\setJ)$, 
the cluster character $\charfl(M)$ can be expressed as a Laurent polynomial 
in the initial cluster variables in a similar style to the cluster character defined by 
Fu--Keller \cite{FuKeller}. 
Note that  the global dimension of $(\End T)^{\text{op}}$ 
is in general infinite for a cluster tilting object $T\in \Dfiltered(\setJ)$. 
Our approach allows us to avoid dealing 
with the technicality of having infinite global dimension. 

\subsection{Flag combinatorics and cluster tilting objects}  \label{sec:1.4}
 Leclerc--Zelevinksy \cite{LZ} showed that two quantum minors $\qminor{I}$, $\qminor{J}$
 quasi-commute if and only if $I$ and $J$ are weakly separated, and 
 gave an explicit formula  $c(I, J)$ for the quasi-commutation rule. 
 Moreover, they showed that one of the products $\qminor{I}\qminor{J}$, 
 $\qminor{J}\qminor{I}$ is invariant under the bar-involution if and only if 
 $I$ and $J$ are strongly separated. 
In the case of Grassmannians,  together with King, we gave a homological interpretation of weak separation
 using vanishing of extension groups \cite{JKS1}.
Intriguingly, for flag varieties,  $\Dfiltered(\setJ)$ provides 
 insights about both types of separation. 
 
\begin{Theorem}[Theorems \ref{thm:extWS}, \ref{Thm:quasicomm}, \ref{thm:strongseparation}, 
\Cref{cor:barinv}]
Let $I,J\subseteq [n]$. 
\begin{itemize}
\item[(1)] $I$ and $J$ are weakly separated if and only if $\ext^1(M_I, M_J)=0$. In this case, 
\begin{eqnarray*}
c(I, J) &=&\dim\Hom(M_I, M_J)-\dim\Hom(M_J, M_I) \\
&=&\dim\Ext^1(M_I, M_J)-\dim\Ext^1(M_J, M_I). 
\end{eqnarray*}
\item[(2)] $I$ and $J$ are strongly separated if and only
if either $\Ext^1(M_J, M_I)=0$ or $\Ext^1(M_I, M_J)=0$. Moreover, 
$\qminor{J}\qminor{I}$ is invariant under the bar-involution if and only if 
$\Ext^1(M_I, M_J)=0$.
\end{itemize}
\end{Theorem}

Consequently, the exact (and 2-CY) structure 
 defined by $\ext^1$ interprets (quantum) cluster combinatorics, while the exact structure 
 inherited from $\mod D$  
detects when the product of two quantum minors is invariant under the 
 bar involution. 

A set $I\subseteq [n]$ is called a {\em $k$-set} if $|I|=k$, and a {\em $\setJ$-set} if $|k|\in \setJ$.
Denote by 
$\maxNC$ a maximal collection of weakly separated $\setJ$-sets $I$. Let 
\[
T_{\maxNC}=\oplus_{I\in \maxNC} M_I.
\]

When $\setJ$ is not an interval, there exists an indecomposable projective-injective object 
$P\in \Dfiltered(\setJ)$ that is not of the form $M_I$ and such an object must be a summand of any cluster tilting object, since
$\Dfiltered(\setJ)$ is Frobenius. We show that  $T_\maxNC$ is a cluster tilting object if and
only if $\setJ$ is an interval (\Cref{Thm:cardinality}).

When $\setJ$ is an interval, we extend a maximal collection $\maxNC_{k, \square}$ 
of weakly separated $k$-sets, which defines the rectangle cluster
in $\CC[\Gr(k, n)]$ \cite{RW},
to a collection $\maxNC_{\setJ, k, \square}$, where $k\in \setJ$. This is achieved by adding interval 
$t$-sets with $t>k$ and cointerval $s$-sets with $s<k$, where a cointerval is defined as 
the complement of an interval subset of $[n]$.
We show that $ \maxNC_{\setJ, k, \square}$ is a maximal collection of weakly separated $\setJ$-sets  
and $T_{\maxNC_{\setJ, k, \square}}$ is a cluster tilting object (\Cref{Thm:cardinality}) and 
is referred to as an \emph{extended cluster tilting object}. 
A module $M\in \Dfiltered(\setJ)$ is said to be \emph{reachable} if 
it appears as a summand of a cluster tilting object obtained by mutations from some 
$T_{\maxNC_{\setJ, k, \square}}$.

\begin{Theorem} [\Cref{Thm:reachability}]\label{thm4}
Let $\setJ\subseteq [n-1]$ be an interval.
\begin{itemize}
\item[(1)] The  cluster tilting objects  $T_{\maxNC_{\setJ,k, \square}}$ 
and $T_{\maxNC_{\setJ,k',  \square}}$, for any $k, k'\in \setJ$,  are mutation equivalent 
through mutations by flips. 

\item[(2)] Any module $M_I$ with $I$ a $\setJ$-set can be extended to a reachable  
cluster tilting object of the form $T_{\maxWS}$ that can be obtained from some 
$T_{\maxNC_{\setJ,k, \square}}$ through mutations by geometric exchanges. 
Consequently, $M_I$ is reachable.
\end{itemize}
\end{Theorem}

\subsection{Quantum cluster structure on quantum flag varieties}
Associated to any cluster tilting object $T\in \Dfiltered(\setJ)$, we construct 
a pair of matrices $(\mat B_{T}, \mat L_{T})$. Following 
\cite[Prop. 10.1 and 10.2]{GLS13} (see also \cite[Thm. 6.3]{JKS2} and \cite[Thm. 6.22]{GP}), we 
know that 
the pair $(\mat{B}_T, \mat{L}_T)$ is compatible and 
$\mu_k(\mat{B}_T, \mat{L}_T)=(\mat{B}_{\mu_k T}, \mat{L}_{\mu_k T})$
(\Cref{prop:main}).
We thus obtain a quantum cluster algebra $\qcl{\mat{B}_T, \mat{L}_T}$, which is contained 
in the quantum torus $\CC_q[K(\add T)]$
of the Grothendieck group of $\add T$.
Let $\qcl{{\mat{B}}_{\pi T}, {\mat{L}}_{\pi T}}$
be the quantum cluster algebra constructed in \cite{GLS13}. We obtain 
a $\grothEJ$-graded injection,
\[
\qnu\colon \qcl{{\mat{B}}_T, {\mat{L}}_T} \to \qcl{\mat B_{\projfun T}, \mat L_{\projfun T}}[K(\add \charT_{\setJ})]. 
\]
Moreover,  $\qnu$ becomes an isomorphism after inverting the cluster variable 
associated to $\charT_j$, for all $j\in \setJ$  (\Cref{thm:qmain1}). 
In addition, following \Cref{lem:qrules}, 
we deduce that the difference of the 
quasi-commutation rules of the corresponding cluster variables in 
$\qcl{{\mat{B}}_T, {\mat{L}}_T}$ and $\qcl{\mat B_{\projfun T}, \mat L_{\projfun T}}$
are determined by a bilinear form $\matadj$ defined on $\grothEJ$, which is
independent of the choice of $T$.

We then investigate the relations between $\qcl{{\mat{B}}_T, {\mat{L}}_T}$ and $\pqflag$.
First, consider the case  $\setJ=[n-1]$. Then the combinatorial operations of geometric exchanges (also called 
square moves) and flips give rise to specific mutations of cluster tilting objects
(Propositions \ref{prop:FlipMutation} and  \ref{prop:GeomExMutation}). 
Exploiting this correspondence, we derive explicitly 
the quantum mutation relations associated to these transformations, and show
that they coincide with quanutm Pl\"{u}cker and incidence relations (\Cref{prop:flipqrel}). 
This allows us to identify $\qminor{I}$ with the cluster variable 
$\Upsilon(M_I)$ for any $I\in \maxNC_{\setJ, k, \square}$. Furthermore, 
by \Cref{thm4}, which notably also relies on the combinatorial operations, 
any quantum minor $\qminor{I}$ is a cluster variable. 
Therefore, we have an embedding
\[
\pqflag\to \qcl{{\mat{B}}_T, {\mat{L}}_T}.
\]
As the two algebras are equal at $q=1$ and both quantum algebras are flat 
deformations of their classical counterparts, the embedding becomes an 
equality after extending the ground ring to $\CC(q)$ (\Cref{prop:overfield}).
Then, using the map $\qnu$ and the isomorphism between 
$\qcl{\mat B_{\projfun T}, \mat L_{\projfun T}}$ and $A_q(\n)$ 
(\cite[Thm. 12.3]{GLS13}, \cite[Cor. 11.2.8]{KKKO}, \cite[Thm. B]{GY}),
we obtain an isomorphism of $\Cq$-algebras (\Cref{thm:pfl1}), 
\[\chi\colon \pqflag \to \qcl{\mat B_T, \mat L_T}, ~\qminor{I}\mapsto \Upsilon([M_I]), 
\forall I\subseteq [n] \text{ and } I\not=[n].\]

Next, we generalise the isomorphism $\chi$ to an arbitrary flag variety $\flag(\setJ)$. 
A key step is to prove the reachability of quantum minors in 
$\pqflag$, or equivalently, the reachability of rank one modules in $\Dfiltered(\setJ)$. 
 We don't have an analogue of \Cref{thm4} for arbitrary $\setJ$.
To achieve the reachability, we make use of Geiss--Leclerc--Schr\"{o}er's description of the Gabriel quiver  
$Q_{V_{\bfi}}$ \cite[Thm 2.10]{GLS11},
where $V_{\bfi}$ is a cluster tilting object in $\sub Q_\setJ$ constructed from 
a reduced word of $\bfi$ associated to $\setJ$, to deduce that, for $\setJ_1\subseteq\setJ_2\subseteq[n-1]$,  
$\Dfiltered(\setJ_1)$ is a cluster subcategory of $\Dfiltered(\setJ_2)$. This, together with 
\Cref{thm4} for $\setJ=\setJ_1=[l, l]$ for some $l$, implies that
any rank one module in $\Dfiltered(\setJ)$ is reachable. 
We obtain the following generalisation.

\begin{Theorem} [\Cref{thm:qmain2}] 
Let $\setJ\subseteq [n-1]$. Then 
\begin{equation*}
\chi_{\setJ}\colon \pqflag \to \qcl{\Dfiltered(\setJ)}, ~\qminor{I}\mapsto \Upsilon^T([M_I])
\end{equation*}
is a $\Cq$-algebra isomorphism. 
\end{Theorem}

\subsection{Notation} 
\begin{itemize}

\item $D$ is the Auslander algebra of $\CC[t]/(t^n)$ and $D_i=De_i$ is the indecomposable projective 
$D$-module associated to vertex $i$.

\item The standard $D$-modules are $\Delta_1,\dots, \Delta_n$.

\item  
$\Pi$ is the preprojective algebra of type $\mathbb{A}_{n-1}$, 
$\Pi_j$ and $R_j$ are the indecomposable projective and injective  $\Pi$-modules associated to vertex $j$, respectively.

\item We normally omit the subscripts for $\Hom$, $\End$ and $\Ext^1$, unless the relevant module structure is not clear from the context or we wish to emphasise it.

\item $Q_j$ is the indecomposable injective object in $\sub Q_{\setJ}$ associated to $j$, and $P_j$ 
    is the minimal lift  of $Q_{n-j}$ to $\Dfiltered(\setJ)$.

\item For any algebra defined by a quiver with relations, let  
$S_i$ be  the simple module associated to vertex $i$.

\item $\add M$ is the category whose objects are summands of  finite direct sums of $M$.

\item  $\top M$, $\soc M$ and $\rad M$ are the top, socle  and radical of  $M$, respectively.

\item $K(\mathcal{C})$ is the Grothendieck group of an exact category $\mathcal{C}$.

\item $\qcl{\mat B, \mat L}$ is the quantum cluster algebra defined by the compatible pair 
$(\mat B, \mat L)$.
\item $\qcl{\mathcal{C}}$ is the quantum cluster algebra defined by a reachable cluster tilting object 
in  $\mathcal{C}$.
\item $[m]=[1, m]=\{1, \dots, m\}$ for any integer $m>0$.
\item For any set $L$, $Li=L\cup\{i\}$.
\end{itemize}

\section{The Auslander algebra of $\CC[t]/(t^n)$ and the Dlab--Ringel equivalence}\label{sec2}

We recall basic properties of the Auslander algebra $D$ of $\CC[t]/(t^n)$, 
the Dlab--Ringel equivalence from the category of good $D$-modules to the module category of a preprojective algebra \cite{DR6}, and the category $\sub Q_\setJ$
constructed by Geiss--Leclerc--Schr\"{o}er~\cite{GLS08}.

\subsection{The Auslander algebra of $\CC[t]/(t^{n})$}\label{sec:ausalg}
Fix an integer $n>1$. 
We can present the Auslander algebra $\aus$ of $\CC[t]/(t^{n})$ 
as the path algebra of the quiver 

\begin{equation}\label{eq:doublequivA}
\begin{tikzpicture}[scale=1.6, arr/.style={-angle 60}]
\draw (0, 0) node (a1) {0};
\draw (1, 0) node (a2) {1};
\path[arr] (a1) edge [bend right] node[black, below]{$b_1$}  (a2); 
\path[arr] (a2) edge [bend right]  node[black, above]{$a_1$} (a1); 

\draw (2, 0) node (a3) {2};
\path[arr] (a2) edge [bend right] node[black, below]{$b_2$} (a3); 
\path[arr] (a3) edge [bend right] node[black, above]{$a_2$} (a2); 

\draw (3, 0) node (a4) {$3$};
\path[arr] (a3) edge [bend right] node[black, below]{$b_3$} (a4); 
\path[arr] (a4) edge [bend right] node[black, above]{$a_3$} (a3); 

\draw (3.4, 0) node (x) {$\dots$};
\draw (4, 0) node (c1) {$n-2$};
\draw (5.5, 0) node (d1) {$n-1$};

\draw (4.3, 0.03) node (c) {};
\draw (5.2, 0.05) node (d) {};

\draw (4.3, -0.03) node (c2) {};
\draw (5.2, -0.05) node (d2) {};

\path[arr] (c2) edge [bend right] node[black, below]{$b_{n-1}$}  (d2); 
\path[arr] (d) edge [bend right]  node[black, above]{$a_{n-1}$} (c); 
\end{tikzpicture}
\end{equation}
with relations  $b_{n-1}a_{n-1}$,
$a_ib_i-b_{i-1}a_{i-1}$ for $2\leq i\leq n-1$. Denote the trivial path at vertex $i$ 
by $e_i$, and let   
\begin{equation}\label{eq:t}
t=a_1b_1.
\end{equation}

Denote by $\Delta$ the subalgebra of $\aus$ generated by the arrows $a_i$. Then 
$\Delta$ is the path algebra of the linear quiver of type $\mathbb{A}_n$, and there 
are algebra homomorphisms
\[
\Delta \longrightarrow \aus\longrightarrow \Delta
\] 
such that composition is the identity. This allows us to view $\aus$-modules as
$\Delta$-modules and vice versa.  

Denote by $S_i$ and $\aus_i$  the simple and projective $\aus$-modules associated to 
vertex $i$, respectively. The indecomposable projective module $\aus_{0}$  is 
also injective with  
\[
\End \aus_{0}=e_{0}\aus e_{0}=\CC[t]/(t^n).
\] 
We denote the indecomposable projective $\Delta$-module of length $i$ by 
$\Delta_{i}$. Then 
\begin{center}
\emph{$\Delta_i$ is the indecomposable projective $\Delta$-module 
at vertex $i-1$}. 
\end{center}
Note the shift of indices.
The $n$ indecomposable projective $\Delta$-modules are 
$\Delta_1, \dots, \Delta_n$. 

\subsection{Quasi-hereditary structure of $\aus$}\label{sec:2.2}
We recall several facts about $\aus$ as a quasi-hereditary algebra. 
For background on quasi-hereditary algebras, 
we refer the reader to \cite{DR2}, and for further details on $D$, see
{\cite{BHRR, HV}}.  

The algebra $\aus$ is quasi-hereditary with \emph{standard modules} $\Delta_i$. 
Let \[\Dfiltered\subseteq \mod \aus\] 
be the full subcategory of {\it good modules},  
that is,  the $D$-modules that are filtered by the standard modules $\Delta_i$, and 
are therefore also called  {\it $\Delta$-filtered modules}. 
For any $M\in \mod D$, denote by $_{\Delta}M$ the restriction of $M$ to $\Delta$.

\begin{lemma}\cite{BHRR, HV} \label{Lem:chardelta}
Let $M\in \mod \aus$. The following are equivalent.
\begin{itemize}
\item[(1)] $M$ is a good module.
\item[(2)] The restriction $_{\Delta}M$ is a projective $\Delta$-module.
\item[(3)] The projective dimension of $M$ is at most one.
\item[(4)] There is an embedding $M\to \aus_0^r$ for some $r\geq 0$.
\item[(5)] $\soc M=S_0^r$ for some $r\geq 0$.
\end{itemize}
Consequently, any submodule of a good module is also a good module. 
\end{lemma}

We remark that there is a dual version $\Dnabla$ of $\Dfiltered$. Let $\nabla$ be 
the subalgebra of $D$ generated by the arrows $b_i$. Then $\Dnabla$ consists $D$-modules that 
are filtered by  
the  injective indecomposable $\nabla$-modules $\nabla_i$, called \emph{co-standard} $D$-modules.
Objects in $\Dnabla$ can 
be described by a dual version of \Cref{Lem:chardelta}.

Let $M\in \Dfiltered$. If $_{\Delta}M=\oplus_i \Delta_i^{d_i}$, then 
$\dvector_\Delta M = (d_i)_i$ is called the $\Delta$-{\em dimension vector} 
of $M$, and  we define  the $\Delta$-{\em support} of $M$ to be the set
$\{ i\colon  d_i\neq 0\}$.
The {\em rank} of $M$ is the smallest number $r$ such that there is 
an embedding $M\to  (D_0)^r$.

Denote by $[i,j]$ the interval $\{i,i+1,\cdots,j\}$, where $i\leq j$, and let $[i]=[1, i]$. 
When $i>j$, $[i, j]$ is the empty set. For any $I\subseteq [n]$, 
there exists a unique submodule of $D_0$ with $\Delta$-support $I$ and 
we denote this submodule by $M_I$ \cite{BHRR}. 
When no confusion arises, we simply write $M_{123}$ for $M_{\{1,2,3\}}$. 
 Note that $D_i=M_{[i+1, n]}$ and $\Delta_i=M_{\{i\}}$  
for all $i$.

Recall that a module is {\em rigid} if it has no self-extensions. 

\begin{lemma}\label{Lem:rankoneclassification} \cite{BHRR} 
A module in $\Dfiltered$ is rigid and indecomposable if and only if 
it is a nonzero submodule of $\aus_0$.
Consequently, indecomposable rigid good $D$-modules are parameterised by non-empty subsets of 
$[n]$.
\end{lemma}
\begin{example} \label{Ex:first}
Let $n=3$. The Auslander algebra $\aus$ is defined by the quiver
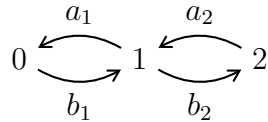
\[
\begin{tikzpicture}[scale=1.6, arr/.style={-angle 60}]

\draw (0, 0) node (a1) {0};
\draw (1, 0) node (a2) {1};
\path[arr] (a1) edge [bend right] node[black, below]{$b_1$}  (a2); 
\path[arr] (a2) edge [bend right]  node[black, above]{$a_1$} (a1); 

\draw (2, 0) node (a3) {2};
\path[arr] (a2) edge [bend right] node[black, below]{$b_2$} (a3); 
\path[ arr] (a3) edge [bend right] node[black, above]{$a_2$} (a2); 
\end{tikzpicture}
\]
with relations  $b_2a_2$ and $b_1a_1-a_2b_2$. 
There are seven indecomposable rigid $D$-modules in $\Dfiltered$.
First, we have the standard modules, 
\[
\begin{tikzpicture}[scale=1.6, arr/.style={black, -angle 60}]
\draw (0, 0) node (a) {$\Delta_3$:};
\draw (1, 0) node (a1) {0};
\draw (2, 0) node (a2) {1};
\draw (3, 0) node (a3) {2};

\path[arr] (a2) edge  node[black, above]{}  (a1); 
\path[arr] (a3) edge  node[black, above]{} (a2);

\draw (4, 0) node (b) {$\Delta_2$:};
\draw (5, 0) node (b1) {0};
\draw (6, 0) node (b2) {1};
\path[arr] (b2) edge  node[black, above]{}  (b1); 

\draw (7, 0) node (c) {$\Delta_1$:};
\draw (8, 0) node (c1) {0};
\end{tikzpicture}
\]
where $\Delta_1$ is a one dimensional simple $D$-module supported at vertex $0$ (and should 
not be confused with the zero module $\{0\}$). The other modules are 
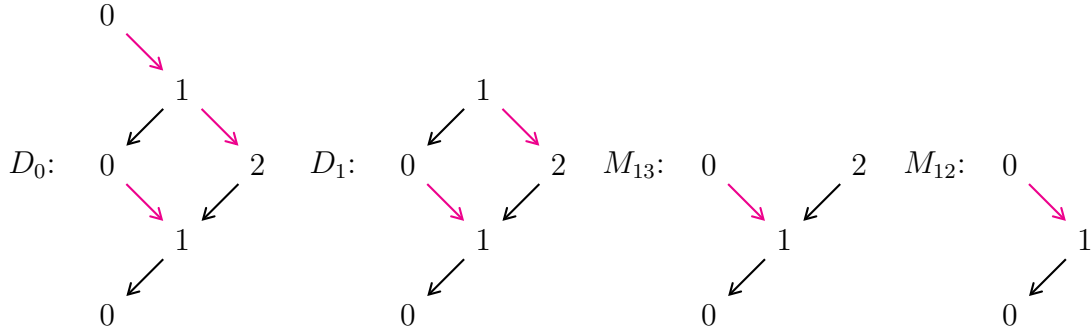
\[
\begin{tikzpicture}[scale=1, black arr/.style={black, -angle 60}, 
magenta arr/.style={magenta, -angle 60}, thick]
\draw (-0, 2) node (a) {$D_0$:};
\draw (1, 0) node (a1) {$0$};
\draw (2, 1) node (b1) {1};
\draw (1, 2) node (c1) {0};
\draw (3, 2) node (c2) {2};
\draw (2, 3) node (d1) {1};
\draw (1, 4) node (e1) {0};

\path[magenta arr] (e1) edge  node{} (d1); 
\path[magenta arr] (d1) edge  node{} (c2); 
\path[magenta arr] (c1) edge  node{} (b1); 
\path[black arr] (b1) edge  node{} (a1); 
\path[black arr] (c2) edge  node{} (b1); 
\path[black arr] (d1) edge  node{} (c1); 

\draw (4, 2) node (a) {$D_1$:};
\draw (5, 0) node (a11) {0};
\draw (6, 1) node (b11) {1};
\draw (5, 2) node (c11) {0};
\draw (7, 2) node (c12) {2};
\draw (6, 3) node (d11) {1};

\path[magenta arr] (d11) edge  node{} (c12); 
\path[magenta arr] (c11) edge  node{} (b11); 
\path[black arr] (b11) edge  node{} (a11); 
\path[black arr] (c12) edge  node{} (b11); 
\path[black arr] (d11) edge  node{} (c11); 

\draw (8, 2) node (a) {$M_{13}$:};
\draw (9, 0) node (a21) {0};
\draw (10, 1) node (b21) {1};
\draw (9, 2) node (c21) {0};
\draw (11, 2) node (c22) {2};

\path[black arr] (b21) edge  node{} (a21); 
\path[black arr] (c22) edge  node{} (b21); 
\path[magenta arr] (c21) edge  node{} (b21);

\draw (12, 2) node (a) {$M_{12}$:};
\draw (13, 0) node (a31) {0};
\draw (14, 1) node (b31) {1};
\draw (13, 2) node (c31) {0};

\path[magenta arr] (c31) edge  node{} (b31); 
\path[black arr] (b31) edge  node{} (a31); 
\end{tikzpicture}
\]
\end{example}

\subsection{The characteristic tilting module}
Let $\charT_i=M_{[i]}$  for any $i\in [n]$. Then 
\[
\charT_{n}=D_0, \; \charT_1=\Delta_1, \; \charT_i=D_0/D_{i} \text{ for any } 1<i<n.
\]
The {\it characteristic tilting $\aus$-module} (see \cite{DR6}) is 
\[
\charT=\bigoplus_{i=1}^{n} \charT_i.
\]
By \cite[Prop. 3.1]{DR6}, 
\begin{equation}\label{eq:chartilt}
\add \charT=\Dfiltered\cap \Dnabla.
\end{equation}

\begin{lemma} \label{Lem:charaddT}
Let $M\in \Dfiltered$.  Then $M\in \add \charT$ if and only 
if $\Ext^1(\charT, M)=0$.
\end{lemma}
\begin{proof} By \Cref{Lem:chardelta}, $\charT$ has projective dimension at most 1. 
By \cite[Prop. 3.2]{DR6}, a $D$-module $X\in \Dnabla$ if and only if $\Ext^i(\charT, X)=0$ for all 
$i>0$. Now the lemma follows from \eqref{eq:chartilt}.
\end{proof}

\begin{lemma} \label{Lem:tiltingchar}
Let $M\in \mod D$. Then $M\in \Dfiltered$  if and only if 
$\Ext^1(M,\charT)=0$.
\end{lemma}
\begin{proof}
By \cite[Prop. 3.2]{DR6}, $M\in \Dfiltered$ if and only if 
 $\Ext^i_\aus(M,\charT)=0$ for all $i>0$. 
In our case, by  \eqref{eq:chartilt} and the dual version of \Cref{Lem:chardelta},  $\charT\in \Dnabla$ and has 
injective dimension at most 1. So the lemma follows.
\end{proof}

For any $X\in \Dfiltered$,  define $\eta_X: \Dfiltered \to \Dfiltered$,  $M\mapsto \eta_{X} M $, the trace of $X$ in $M$, that is, 
\[
\eta_{X} M=\sum_{f: X\rightarrow M} \im f.
\]
Then $\eta_X M$ is a submodule of $M$. By \Cref{Lem:chardelta},  $\eta_X M\in \Dfiltered$, 
and so $\eta_X$ is well-defined. When $X=\charT$, we denote $\eta_X$ by $\eta$.

\begin{lemma}\label{lem:chartilt}
The category $\add \charT$ is closed under taking quotient modules that 
are in $\Dfiltered$.
\end{lemma}
\begin{proof} The injective $\nabla$-modules  are closed under taking quotient 
representations. Now the lemma follows from \eqref{eq:chartilt}, \Cref{Lem:chardelta} and its dual version.
\end{proof}

\begin{lemma} \label{Lem:trace}
Let $M\in \Dfiltered$. Then
\begin{itemize}
\item[(1)] $\eta M = \eta_{\charT_{n}} M$.
\item[(2)] $\eta M = \aus e_0M$.
\item[(3)] $\eta M\in \add \charT$.
\end{itemize}
In particular, 
 $\eta M_I=\charT_{|I|}$, $\eta \aus_i=\charT_{n-i}$, and 
when $M_I\not= \charT_{|I|}$,  $\soc (M_I/\eta M_I)=S_{|I|}$.
\end{lemma}
\begin{proof}
(1) As $\charT_n=\aus_0$ is the projective $\aus$-module at vertex $0$ and any 
$\charT_i$ has a single top at $0$, any map $f\colon \charT_i\ra M$ can be 
extended to $fg\colon \charT_n\to M$ with $\im f=\im fg$, where $g\colon\charT_n\ra \charT_i$ 
is a surjection.  Therefore, (1) holds. 

(2) Since $\Hom(\charT_n, M)=\Hom(D_0, M)=e_0M$,   
$\eta_{\charT_n}M$ is the submodule of $M$ generated by $e_0M$ as claimed.
 
(3) By definition, $\eta M$ is a quotient of some module in $\add \charT$. So 
$\eta M\in \add \charT$,
by \Cref{lem:chartilt}.

The remainder is straightforward.   
\end{proof}

Note that multiplication with the idempotent $e_0$ is an exact functor, which we denote by $\e_0$. We can interpret $\e_0$ as follows,
\[
\e_0=\Hom(\aus_0,-)\colon \mod \aus \rightarrow \mod \CC[t]/(t^n), M\mapsto e_0M.
\]
\begin{lemma} \label{Lem:charTequivalence}
The functor
\[
\e_0\colon \Dfiltered \rightarrow \mod \CC[t]/(t^n) 
\]
is faithful, and when restricted to $\add \charT$, it is an equivalence of additive categories.
\end{lemma}
\begin{proof} 
Let $f:M\ra N$ be a homomorphism in $\Dfiltered$ with $\e_0(f)=0$. 
Then $e_0\im (f)=0$, and so $\im (f) = 0$, since $\soc N \in \add S_0$. 
Hence $\e_0$ is faithful.

Note that $\e_0\charT_i=\CC[t]/(t^{i})$. So the restriction of $\e_0$ to $\add \charT$
is dense. 
By comparing the dimension of homomorphism spaces between 
indecomposables, we can conclude that $\e_0$ is full on $\add \charT$.
Therefore, the restriction of the functor $\e_0$ to $\add \charT$ is an equivalence.
\end{proof}
\begin{example} \label{Ex:second}
Continuing   \Cref{Ex:first}, we have 
\[
\charT_3=D_0, ~\charT_2=M_{[2]} \text{ and } \charT_1=S_0
\]
and 
\[
\e_0\charT_3 = \CC[t]/(t^3), ~\e_0\charT_2=\CC[t]/(t^2) \text{ and } 
\e_0\charT_1=\CC[t]/(t).
\]
\end{example}
\subsection{The Dlab--Ringel equivalence}
Let 
\[
\Pi = D/De_0D,
\] 
which is the preprojective algebra of type $\mathbb{A}_{n-1}$. Note that the 
top of the projective $\Pi$-module  $\Pi_i$ at vertex $i$ is $S_i$ and its socle is  
$S_{n-i}$. We have 
\begin{eqnarray*}
\mod \Pi &= & \{M\in \mod \aus\colon \Hom(\charT, M)=0 \}\\
                &=& \{M\in \mod \aus\colon e_0 M=0\}.
\end{eqnarray*}
Also, by \Cref{Lem:trace} (2), 
\[
M/\eta M=M/\aus e_0M\in \mod \Pi,
\]
for any $M\in \mod \aus$.

The following theorem is a reformulation of \cite[Thm. 3 and Lem. 7.1]{DR6} (see
also \cite[Prop. 5]{RZ}). 
\begin{theorem}\label{Thm:DlabRingel} The functor 
$$
\projfun\colon \mod \aus \rightarrow \mod \Pi, \;\; M \rightarrow M/\eta M
$$
induces an equivalence of additive categories
$$
\Dfiltered/\add  \mathbb{T} \rightarrow \mod \Pi. 
$$ 
\end{theorem}

The maps which factor through $\add \charT$ can be understood as follows. 

\begin{lemma} \label{Lem:factormaps}
A map $f\colon M\rightarrow N$ in $\Dfiltered$ factors through $\add \charT$ if and only 
if it factors through the inclusion $\eta N\subseteq N$.
\end{lemma}
\begin{proof}
If $f$ factors through $\add \charT$, then its image is contained in the trace $\eta N$,
and so it factors through the inclusion $\eta N\subseteq N$. The converse follows since $\eta N$ is in 
$\add \charT$ by \Cref{Lem:trace} (3).
\end{proof}
\subsection{The subcategory $\sub Q_{\setJ}$}
Let $\setJ \subseteq [n-1]$. We recall 
the Frobenius subcategory 
$\sub Q_{\setJ}$ of $\mod \Pi$ constructed by Geiss--Leclerc--Schr\"{o}er 
\cite{GLS08}. Let
\[
Q_{\setJ}=\oplus_{j\in \setJ}\Pi_{n-j} 
\]
and $\sub Q_{\setJ}$ be the full subcategory of $\mod\Pi$ consisting of objects 
that are submodules of modules in $\add Q_{\setJ}$. In other words, $\sub Q_{\setJ}$ 
are the $\Pi$-modules which have socles in $\add S_\setJ$, where 
\[
S_\setJ=\oplus_{j\in \setJ}S_j.
\]
Any $M\in \mod \Pi$ has a unique maximal quotient module in $\sub Q_\setJ$. When
$M=\Pi_{n-j}$, this maximal quotient is denoted by $Q_j$.

The lemma below follows from the construction of the modules $Q_j$.
\begin{lemma} \label{Lem:soclesub}
Let $\setJ=\{j_1 < j_2 < \cdots < j_r\}$. 
\begin{itemize}
\item[(1)] If $j< j_1$ then $\soc Q_j=S_{j_1}$
\item[(2)] If $j\in \setJ$, then $Q_j=\Pi_{n-j}$ and $\soc Q_j=S_j$.
\item[(3)] If $j_i < j < j_{i+1}$, then $\soc Q_j = S_{j_i}\oplus S_{j_{i+1}}$.
\item[(4)] If $j>j_r$, then $\soc Q_j = S_{j_r}$.
\end{itemize}
\end{lemma}
\begin{theorem} \cite{GLS08} \label{Prop:GLSFrob}
$\sub Q_{\setJ}$ is a full exact Frobenius stably 2-CY subcategory 
of $\mod \Pi$ with indecomposable projective-injective objects  $Q_1, \dots, Q_n$. 
\end{theorem}
\begin{example}
Let $n=5$ and $\setJ=\{1, 3\}$. The projective-injective $\Pi$-module 
$\Pi_2$ and the projective-injective object $Q_2$ in $\sub Q_{\setJ}$ 
are as follows.
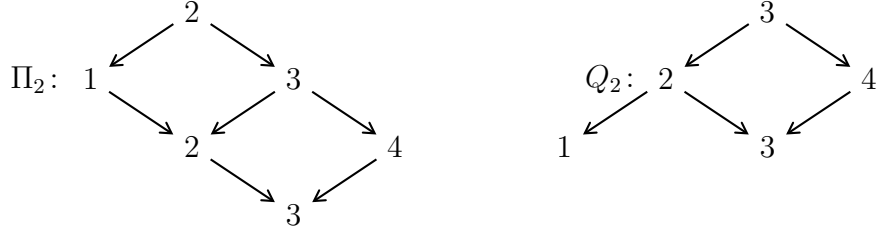
\[
\begin{tikzpicture}[scale=0.45,
 arr/.style={black, thick, -angle 60}]
\path (2,4) node (c1) {$2$};
\path (-1,2) node (c2) {$1$};
\path (5,2) node (c3) {$3$};
\path (2, 0) node (c4) {$2$};
\path (-2.5,2) node  {$\Pi_2\colon$};

\path (8, 0) node (c5) {$4$};
\path (5,-2) node (c6) {$3$};

\path[arr] (c1) edge node[auto]{} (c2);
\path[arr] (c1) edge node[auto]{}  (c3);
\path[arr] (c2) edge (c4);
\path[arr] (c3) edge (c4);

\path[arr] (c3) edge (c5);
\path[arr] (c4) edge (c6);
\path[arr] (c5) edge (c6);

\path (19,4) node (d1) {$3$};
\path (16,2) node (d2) {$2$};
\path (22,2) node (d3) {$4$};
\path (11.5,2) node  {$Q_2\colon$};

\path (13,0) node (d4) {$1$};
\path (19,0) node (d5) {$3$};

\path[arr] (d1) edge node[auto]{} (d2);
\path[arr] (d1) edge node[auto]{}  (d3);

\path[arr] (d2) edge node[auto]{} (d4);
\path[arr] (d2) edge node[auto]{} (d5);
\path[arr] (d3) edge node[auto]{} (d5);

\end{tikzpicture}
\]
\end{example}
\section{The subcategory $\Dfiltered(\setJ)$} \label{sec3}

For any  $\setJ\subseteq [n-1]$, we construct  a subcategory $\Dfiltered(\setJ)$ of 
$\Dfiltered$, which is mapped onto $\sub Q_\setJ$ via the Dlab--Ringel equivalence.

\subsection{Construction of $\Dfiltered(\setJ)$}
Let $\setJ$ be a subset of $[n-1]$ and define 
\[
\charT_{\setJ}=\oplus_{j\in \setJ} \charT_j.
\]
Let $\Dfiltered(\setJ)$ be the full subcategory of $\Dfiltered$ consisting
of $M$ such that  
\[
\eta M \in \add \charT_{\setJ}. 
\]
In particular, $\charT_{\setJ}$ belongs to $\Dfiltered(\setJ)$. Note that if 
$M\in \Dfiltered(\setJ)$, then  
\[
\eta M=\eta_{\charT_{\setJ}}M.  ~~ 
\] 
For rank one modules, we have the following consequence of \Cref{Lem:trace}.
\begin{lemma} \label{Lem:whichrankone} \qquad
\begin{itemize}
\item[(1)] Let $I\subseteq [n]$. Then $M_I\in \Dfiltered(\setJ)$ if and only if $|I|\in \setJ$.  

\item[(2)] 
The projective module $D_{n-j}\in \Dfiltered(\setJ)$ if and only if $j\in \setJ$, in which case 
 we have the short exact sequence (with $Q_j=\Pi_{n-j}$), 
\begin{equation} \label{Eq:projmod}
\begin{tikzpicture}[scale=0.6,
 arr/.style={black, -angle 60}, baseline=(bb.base)]
\pgfmathsetmacro{\eps}{0.2}
\coordinate (bb) at (0, 2.8);

\path (0,3) node(z1) {$0$}; 
\path (12.2,3) node(z2) {$0$.};
\path (3,3) node(b1) {$\charT_j$};
\path (6.1,3) node(b2) {$D_{n-j}$};
\path (9.2,3) node(b3) {$Q_j$};

\path[arr] (z1) edge (b1);
\path[arr] (b1) edge (b2);
\path[arr] (b2) edge (b3);
\path[arr] (b3) edge (z2);
\end{tikzpicture}
\end{equation}
\end{itemize}
\end{lemma}

\subsection{Restrictions of the Dlab--Ringel equivalence}
We consider the restriction of the functor $\projfun$ to $\Dfiltered(\setJ)$.
\begin{definition}
Let $N\in \sub Q_{\setJ}$.
A module $M$ in $\Dfiltered(\setJ)$ with $\projfun M=N$ will be called a {\it  lift}
of $N$ to $\Dfiltered(\setJ)$. A lift $M$ is {\it minimal} if it does not contain summands
from $\add \charT_{\setJ}$. 
\end{definition}

Let $N\in \sub Q_\setJ$. Then there is an injective map $\iota\colon N\ra Q$ with 
$Q\in \add Q_\setJ$. 
Consider the following pullback diagram,
\begin{equation} \label{Eq:lift}
\begin{tikzpicture}[scale=0.7,
 arr/.style={black, -angle 60}, baseline=(bb.base)]
\pgfmathsetmacro{\eps}{0.2}
\coordinate (bb) at (0,2.8);
 
\path (9,5) node(z1) {$0$};
\path (6,5) node(z2) {$0$};
\path (3,3) node(b1) {$\ker \lambda$};
\path (6,3) node(b2) {$\N$};
\path (9,3) node(b3) {$N$};
\path (9,1) node (c3) {$Q$};
\path (6,1) node (c2) {$P$};
\path (3,1) node (c1) {$\ker \mu$};
\path (12, 1) node (c4) {$0$,};
\path (0,1) node (c0) {$0$};

\path(0, 3) node (b0) {$0$};
\path(12, 3) node (b4) {$0$};
\path[arr] (b0) edge (b1);
\path[arr] (b3) edge (b4);
\path[arr] (z1) edge (b3);
\path[arr] (z2) edge (b2);
\path[arr] (b1) edge (b2);
\path[arr] (b2) edge node [auto]{$\lambda$}  (b3);
\path[arr] (c1) edge  (c2);
\path[arr] (c2) edge node [auto]{$\mu$} (c3);
\path[arr] (c0) edge (c1);
\path[arr] (c3) edge (c4);

\path[arr] (b3) edge node [auto]{$\iota$}(c3);
\path[arr] (b2) edge node [auto]{} (c2);
\path[arr] (b1) edge node [auto]{$\verteq$} (c1);
\end{tikzpicture}
\end{equation}
where  $P\rightarrow Q$ is an $\aus$-projective cover.

\begin{lemma} \label{Lem:dense} 
Let $N\in \sub  Q_{\setJ}$. Use the same notation as in \eqref{Eq:lift}.
\begin{itemize}
\item[(1)] The pullback  $\N $ is a lift of $N$ to $ \Dfiltered(\setJ)$ and $\eta \N=\eta P$.

\item[(2)] The lift $\N$ is unique up to summands from $\add\charT_{\setJ}$.  
Moreover, $\N$ is a minimal lift if and only if $\iota$ is an injective envelope.  

\item[(3)] If $\eta \N=\oplus_i \charT_i^{d_i}$, then $(d_i)_i=\dvector \soc Q\geq \dvector \soc N$, 
with equality when $\N$ is a minimal lift.
\end{itemize}
\end{lemma}
\begin{proof}
(1) By \Cref{Lem:chardelta}, $\N$ is a good module, since it
is a submodule of a projective $\aus$-module. So it remains to show that $\N \in \Dfiltered(\setJ)$, 
that is, $\eta \N\in \add \charT_{\setJ}$.

As $\N$ is the pullback of $\mu$ and $\iota$, $\ker \lambda=\ker \mu$ and by 
\eqref{Eq:projmod}, 
\[
\ker \mu = \eta P\in \add \charT_{\setJ}.
\] 
As $N$ and $Q$ are not supported at vertex $0$, by \eqref{Eq:lift}, 
\[
e_0\N =e_0\ker \lambda=e_0\ker \mu=e_0P.
\]
Therefore, by \Cref{Lem:trace} (2), 
\begin{equation}\label{eq:cometa}\eta \N=\eta P=\ker\lambda\in \add \charT_{\setJ}\end{equation} 
and 
\[
N=\N/\ker\lambda=\N/\eta \N.
\]
So $\N$ is a lift of $N$ to $\Dfiltered(\setJ)$.

(2) The uniqueness of $\N$ up to summands follows by the Dlab--Ringel 
equivalence  from \Cref{Thm:DlabRingel}.
Suppose that $\N=T'\oplus N'$ with $N'$ a minimal lift of $N$ and $T'\in \add \charT_\setJ$. Then
\[
\ker \mu=\ker \lambda=T'\oplus K,
\]
with $K=\eta N'$. 
Consequently,  there is a decomposition $P=P(T')\oplus P(K)$ and $Q=Q(T')\oplus Q(K)$ such  that
\[
P(T')/T'\cong Q(T') \text{ and } P(K)/K\cong Q(K).
\]
So $N$ can be embedded into $Q(K)$. 
Therefore, $Q$ is the injective envelope of $N$ if and only if 
$Q(T')=0$,  equivalently $\N$ is a minimal lift of $N$.

(3) Follows from \Cref{Lem:whichrankone} (2) and the equality $\eta \N=\eta  P$.
\end{proof}
\begin{remark}
Part (3) of the lemma shows that the trace $\eta M$ for $M\in \Dfiltered(\setJ)$ plays
the role of the socle of $\projfun M\in \sub Q_{\setJ}$, and this explains our labelling of
 the indecomposable summands of $\charT$.
\end{remark}
\begin{lemma}\label{Lem:welldefined}
Let $M$ be a module in $\Dfiltered(\setJ)$. Then
\[
\projfun M=M/\eta M\in \sub Q_{\setJ}.
\]
\end{lemma}
\begin{proof}
Let $\tilde{M}$ be defined by the pullback  of $\mu$ and $\iota$ as in \eqref{Eq:lift}, 
\[
\begin{tikzpicture}[scale=0.7,
 arr/.style={black, -angle 60}]
\path (0,3) node (b0) {$0$};
\path (3,3) node (b1) {$\ker \mu$};
\path (6,3) node (b2) {$\tilde{M}$};
\path (9,3) node(b3) {$\projfun M$};
\path (12,3) node (b4) {$0$};

\path[arr] (b1) edge  (b2);
\path[arr] (b2) edge node [auto]{} (b3);
\path[arr] (b0) edge (b1);
\path[arr] (b3) edge (b4);

\path (9,1) node (c3) {$Q$};
\path (6,1) node (c2) {$P$};
\path (3,1) node (c1) {$\ker \mu$};
\path (12, 1) node (c4) {$0$};
\path (0,1) node (c0) {$0$};

\path[arr] (c1) edge  (c2);
\path[arr] (c2) edge node [auto]{$\mu$} (c3);
\path[arr] (c0) edge (c1);
\path[arr] (c3) edge (c4);

\path[arr] (b3) edge node [auto]{$\iota$}(c3);
\path[arr] (b2) edge node [auto]{}(c2);
\path[arr] (b1) edge node [auto]{$\verteq$}(c1);
\end{tikzpicture}
\]
 where we assume that $\iota$ is an injective envelope.  
By \Cref{Lem:dense} (1) and \eqref{eq:cometa}, respectively,  
\[\projfun M=\projfun \tilde{M} \text{ and } \eta \tilde{M}=\eta P=\ker \mu,\]
and by \Cref{Lem:dense} (2),
$M = \tilde{M}\oplus T'$ for some $T'\in \add \charT_\setJ$. So
$\eta  \tilde{M}$ is a summand of $\eta M$, which is in $\add \charT_\setJ$. Consequently, 
 \[\eta P \in \add \charT_{\setJ},\] and so 
  by \Cref{Lem:whichrankone} (2),  $Q\in \add Q_{\setJ}$. Therefore 
$\projfun M\in \sub Q_{\setJ}$.
\end{proof}

We have the following refinement of the Dlab--Ringel equivalence.

\begin{proposition}\label{Thm:equiv} The restriction $\projfun$ to
$\Dfiltered(\setJ)$ induces an equivalence,
\[
\Dfiltered(\setJ)/\add \charT_{\setJ} \rightarrow \sub Q_{\setJ}, \; \;
 X\mapsto X/\eta X.
\]
\end{proposition}
\begin{proof} Denote the induced functor by $\projfun_\setJ$.
The functor $\projfun_\setJ$ is well-defined and dense following
\Cref{Lem:welldefined} and \Cref{Lem:dense}.
By definition, $\eta M\in \add\charT_{\setJ}$ for any
$M\in \Dfiltered(\setJ)$. It then follows from the factorisation in
\Cref{Lem:factormaps} that the  category
$\Dfiltered(\setJ)$ modulo maps that factor through $\add \charT$ 
is equivalent to $\Dfiltered(\setJ)/\add \charT_{\setJ}$.
So $\projfun_\setJ$ can be viewed as the restriction of the Dlab--Ringel equivalence
 to $\Dfiltered (\setJ)/\add \charT_\setJ$.
Consequently, $\projfun_\setJ$ is full and faithful, and therefore an equivalence.
\end{proof}

It follows from \Cref{Thm:equiv}  that any map $f\colon M\ra N$ in $\sub Q_{\setJ}$ can be lifted to a homomorphism 
\[
\f\colon \M \ra \N,
\] 
in $\Dfiltered(\setJ)$ (i.e. $\projfun(\f) = f$). We give an explicit construction of this lift.

\begin{lemma} \label{Lem:liftofmaps} Let $f\colon M\ra N$ be a map in
$\sub Q_{\setJ}$.  
Then $\f\colon\M\ra \N$ defined by the following commutative diagram is a lift of 
$f$ to $\Dfiltered(\setJ)$,
\[
\begin{tikzpicture}[scale=0.7,
 arr/.style={black, -angle 60}]
 
 \path(6, 5) node (a1) {$M$};
 \path(10.5, 5) node (a2) {$N$};
 
 \path(4, 3.3) node (b1) {$\M$};
 \path(8.5, 3.3) node (b2) {$\N$};
 
 \path(6, 2) node (c1) {$Q_M$};
 \path(10.5, 2) node (c2) {$Q_N$};
 
 \path(6.5, 3.7) node {$\f$};
 
 \path(4, 0.3) node (d1) {$P_M$};
 \path(8.5, 0.3) node (d2) {$P_N$};
 
\path[arr] (a1) edge node[auto]{$f$} (a2);
\path[arr] (b1) edge    (b2);
\path[arr] (c1) edge (c2);
\path[arr] (d1) edge (d2);

\path[arr] (b1) edge  (a1);
\path[arr] (b2) edge (a2);
\path[arr] (b1) edge (d1);
\path[arr] (b2) edge (d2);

\path[arr] (a1) edge (c1);
\path[arr] (d1) edge (c1);
\path[arr] (a2) edge (c2);
\path[arr] (d2) edge (c2);

\end{tikzpicture}
\]
where $Q_M, Q_N\in \add Q_{\setJ}$,  $P_M$ and $P_N$ are $\aus$-projective covers
of $Q_M$ and $Q_N$, and $\M$ and $\N$ are pullbacks.
Moreover, the lift is unique  in the sense that if $g$
is another lift of $f$, then $\f-g$ factors through $\eta \N$.
\end{lemma}
\begin{proof}
Since $Q_N$ is an injective object in $\sub Q_\setJ$ and the map $M\to Q_M$ is injective, 
the map $Q_M\to Q_N$ exists.  
As for  the map $P_M\to P_N$, the existence is due to the fact that  $P_M$ and $P_N$ are projective covers of
$Q_M$ and $Q_N$, respectively.
Now, by the universal property of a pullback,  we have the map $\f$
such that $\projfun(\f)=f$, as required.  
The uniqueness follows from the factorisation of maps through $\eta \N$
given in \Cref{Lem:factormaps}.
\end{proof}

\section{Frobenius exact structure on $\Dfiltered(\setJ)$}\label{sec4}
The subcategory $\Dfiltered(\setJ)$ inherits an exact structure from $\mod D$, but under 
this exact structure the functor $\projfun$ is not exact. 
In this section, we will introduce a new exact structure $\mathcal{E}$ on $\Dfiltered(\setJ)$ such that the functor
$$
\projfun\colon \Dfiltered(\setJ)\to \sub Q_{\setJ}
$$
is full exact. We show that $\Dfiltered(\setJ)$ equipped with $\mathcal{E}$ is Frobenius stably 
2-CY and investigate the Gabriel quiver for its minimal projective-injective generator. 

\subsection{Construction of the exact structure $\mathcal{E}$ on $\Dfiltered(\setJ)$}
\begin{lemma}\label{Lem:exact1}
Let $L, M, N\in \Dfiltered(\setJ)$, and let
\[
\begin{tikzpicture}[scale=0.6,
 arr/.style={black, -angle 60}]
\path (9,1) node (c3) {$M$};
\path (6,1) node (c2) {$L$};
\path (3,1) node (c1) {$N$};
\path (12, 1) node (c4) {$0$,};
\path (-0.7,1) node  {$\xi\colon $};
\path (0, 1) node (c0) {$0$};

\path[arr] (c1) edge node[auto]{} (c2);
\path[arr] (c2) edge node[auto]{$g$}  (c3);
\path[arr] (c0) edge (c1);
\path[arr] (c3) edge (c4);
\end{tikzpicture}
\]
be an exact sequence in $\mod D$.
The following are equivalent.
\begin{itemize}

\item[(1)] The sequence $\eta(\xi)$ is exact,
\[
\begin{tikzpicture}[scale=0.6,
 arr/.style={black, -angle 60}]
 
\path (9.4,1) node (c3) {$\eta M$};
\path (6.2,1) node (c2) {$\eta L$};
\path (3,1) node (c1) {$\eta N$};
\path (12.4, 1) node (c4) {$0.$};
\path (-1.3,1) node {$\eta (\xi)\colon$};
\path (0,1) node (c0) {$0$};

\path[arr] (c1) edge (c2);
\path[arr] (c2) edge (c3);
\path[arr] (c0) edge (c1);
\path[arr] (c3) edge (c4);
\end{tikzpicture}
\]

\item[(2)] The sequence $\projfun(\xi)$ of $\Pi$-modules is exact,
\[
\begin{tikzpicture}[scale=0.6,
 arr/.style={black, -angle 60}]
 
\path (9.4,1) node (c3) {$\projfun M$};
\path (6.2,1) node (c2) {$\projfun L$};
\path (3,1) node (c1) {$\projfun N$};
\path (12.4, 1) node (c4) {$0.$};
\path (-1.3,1) node {$\projfun(\xi)\colon$};
\path (0,1) node (c0) {$0$};

\path[arr] (c1) edge  (c2);
\path[arr] (c2) edge  (c3);
\path[arr] (c0) edge (c1);
\path[arr] (c3) edge (c4);
\end{tikzpicture}
\]

\item[(3)] The short exact sequence   $e_0(\xi)$ of $e_0\aus e_0$-modules splits,
\[
\begin{tikzpicture}[scale=0.6,
 arr/.style={black, -angle 60}]
 
\path (9.4,1) node (c3) {$e_0M$};
\path (6.2,1) node (c2) {$\e_0L$};
\path (3,1) node (c1) {$e_0N$};
\path (12.4, 1) node (c4) {$0.$};
\path (-1.3,1) node {$\e_0(\xi)\colon$};
\path (0,1) node (c0) {$0$};

\path[arr] (c1) edge  (c2);
\path[arr] (c2) edge  (c3);
\path[arr] (c0) edge (c1);
\path[arr] (c3) edge (c4);
\end{tikzpicture}
\]

\item[(4)] The sequence $\Hom(\charT, \xi)$ of  projective $(\End \charT)^{\text{op}}$-modules is exact,
\[
\begin{tikzpicture}[scale=0.6,
 arr/.style={black, -angle 60}]
 
\path (13.3,1) node (c3) {$\Hom(\charT, M)$};
\path (8.3,1) node (c2) {$\Hom(\charT, L)$};
\path (3.3,1) node (c1) {$\Hom(\charT, N)$};
\path (17.3, 1) node (c4) {$0.$};
\path (-2.4,1) node {$\Hom(\charT, \xi)\colon$};
\path (-0.4,1) node (c0) {$0$};

\path[arr] (c1) edge  (c2);
\path[arr] (c2) edge  (c3);
\path[arr] (c0) edge (c1);
\path[arr] (c3) edge (c4);
\end{tikzpicture}
\]

\item[(5)] $\eta L=\eta  M\oplus \eta N$.
\end{itemize}
\end{lemma}
\begin{proof}
Clearly, the modules in (2) and (3) are the types of modules as claimed.   
Those in (4) are projective $(\End \charT)^{\text{op}}$-modules, because for any $X\in \Dfiltered$, by \Cref{Lem:trace} and  the definition of $\eta X$, respectively, 
\begin{equation}\label{eq:projective}
\eta X\in \add \charT \text{ and }
\Hom(\charT, X)=\Hom(\charT, \eta X).
\end{equation}

We now prove the stated equivalences. 

The equivalence (1) $\Leftrightarrow$ (2) follows directly from 
definition of $\projfun X$ as the cokernel of the inclusion $\eta X\subseteq X$ for 
any $X\in \Dfiltered$, and the Snake Lemma. 

 (1) $\Rightarrow$ (3):
Exactness of the sequence in (1) implies split exactness, because $\charT$ is rigid and 
 $\eta X\in \add \charT$ for any $X\in \Dfiltered$ (see \eqref{eq:projective}). 
 Now (3) follows from the fact that 
 \begin{equation}\label{eq:eX}
 e_0X=e_0\eta X.
 \end{equation}

 (1) $\Rightarrow$ (4): This implication follows since  $\eta(\xi)$ splits and, 
 by the definition of $\eta$,
\begin{equation} \label{eq:homXetaX}
\Hom(\charT, \xi) = \Hom(\charT, \eta(\xi)).
\end{equation}

 (3) $\Rightarrow$ (5): First, following the fact in \eqref{eq:eX}, the split exactness of 
 $\e_0(\xi)$ implies the split exactness of $\e_0\eta(\xi)$.
 Now (5) follows from the fact that  $e_0De_0\cong \CC[t]/(t^n)$ and the equivalence in 
\Cref{Lem:charTequivalence},
\[ \e_0\colon \add \charT\to \mod \CC[t]/(t^n). \]

(4) $\Rightarrow$ (5): the exactness of the sequence in (4) implies that it is split exact. 
Now by \eqref{eq:homXetaX} and the equivalence 
$\Hom(\charT, -)\colon \add \charT\to \proj (\End \charT)^{\text{op}}$, we have
\[\eta L=\eta M\oplus \eta N.\]

(5) $\Rightarrow$ (1): Note that $\eta$ preserves injections and surjections. By comparing dimensions,  
(5) implies (1).
\end{proof}
Let $\mathcal{E}$ be the collection of short exact sequences in $\Dfiltered(\setJ)$
satisfying the equivalent conditions in \Cref{Lem:exact1}. Denote by $\ext^1(M, N)$
the subspace of $\Ext^1(M, N)$ consisting of extensions of $M$ by $N$
in $\mathcal{E}$.
In the sequel, an extension
\[
\begin{tikzpicture}[scale=0.6,
 arr/.style={black, -angle 60}]
\path (9,1) node (c3) {$M$};
\path (6,1) node (c2) {$L$};
\path (3,1) node (c1) {$N$};
\path (12, 1) node (c4) {$0$};
\path (-1,1) node  {};
\path (0, 1) node (c0) {$0$};

\path[arr] (c1) edge node[auto]{$f$} (c2);
\path[arr] (c2) edge node[auto]{$g$}  (c3);
\path[arr] (c0) edge (c1);
\path[arr] (c3) edge (c4);
\end{tikzpicture}
\]
in $\Ext^1(M, N)$ is said to be {\it admissible} if it belongs to $\ext^1(M,N)$,
in which case the injection $f$ and the surjection $g$ are also said to be {\em admissible}. 

\begin{remark} \label{Rem:exact1}By \Cref{Lem:exact1}, $\ext^1(M, N)$
is the kernel of the restriction
\[
\Ext^{1}_{\aus}(M, N)\rightarrow \Ext^{1}_{e_0De_0}(e_0M, e_0N).
\]
\end{remark}
The goal of this subsection is to prove that $\mathcal{E}$ gives $\Dfiltered(\setJ)$ the structure
of an exact category.
We recall the definition of an exact category, following the appendix by 
Keller in \cite{DRSKeller}. Let $\mathcal{C}$ be an additive category. A sequence 
\[
\begin{tikzpicture}[scale=0.6,
 arr/.style={black, -angle 60}]
 
\path (9,1) node (c3) {$M$};
\path (6,1) node (c2) {$L$};
\path (3,1) node (c1) {$N$};

\path[arr] (c1) edge node[auto]{$i$} (c2);
\path[arr] (c2) edge node[auto]{$d$} (c3);
\end{tikzpicture}
\]
is  {\em exact}  if $i$ is the kernel of $d$ and $d$ is the cokernel of $i$.
In this case, $i$ is usually called an {\em inflation} and $d$ 
a {\em deflation}.  

\begin{definition}  \cite{DRSKeller}
Let $\mathcal{E}'$ be a collection of exact sequences in an additive category $\mathcal{C}$.  
Then $(\mathcal{C}, \mathcal{E}')$ is an exact category if the following four
conditions are satisfied.
\begin{itemize}
\item[(E1)] The identity morphism of the zero object is a deflation.
\item[(E2)] The composition of two deflations is a deflation.
\item[(E3)] For each $g\in \Hom_\mathcal{C}(M', M)$ and each deflation
$d \in \Hom_\mathcal{C}(N, M)$, there is a Cartesian square (a pullback)
\[
\begin{tikzpicture}[scale=0.6,
 arr/.style={black, -angle 60}]

\path (3,1) node (b1) {$N$};
\path (7,1) node (b2) {$M$};

\path (3,3) node (c1) {$N'$};
\path (7,3) node (c2) {$M'$};
\path[arr] (c1) edge node[auto]{$d'$} (c2);
\path[arr] (b1) edge node[auto]{$d$} (b2);

\path[arr] (c1) edge node[left]{$g'$} (b1);
\path[arr] (c2) edge node[auto]{$g$} (b2);
\end{tikzpicture}
\]
such that $d'$ is a deflation.
\item[(E4)] Dually, for each $f\in \Hom_\mathcal{C}(M', M)$ and each inflation 
$i \in \Hom_\mathcal{C}(M, N)$, there is a coCartesian square (a pushout)
\[
\begin{tikzpicture}[scale=0.6,
 arr/.style={black, -angle 60}]

\path (3,1) node (b1) {$M'$};
\path (7,1) node (b2) {$N'$};

\path (3,3) node (c1) {$M$};
\path (7,3) node (c2) {$N$};
\path[arr] (c1) edge node[auto]{$i$} (c2);
\path[arr] (b1) edge node[auto]{$i'$} (b2);

\path[arr] (c1) edge node[left]{$f$} (b1);
\path[arr] (c2) edge node[auto]{$f'$} (b2);
\end{tikzpicture}
\]
such that $i'$ is an inflation.
\end{itemize}
\end{definition}
\begin{lemma} \label{Lem:exactcategory}
$(\Dfiltered(\setJ), ~\mathcal{E})$ is an exact category.
\end{lemma}
\begin{proof}
As $0$ is an object in $\Dfiltered(\setJ)$, (E1) holds. Below we will
first prove (E3), then deduce (E2) from (E3). The proof of (E4) is similar to 
that of (E3), and we skip the details. 

(E3) Let $f\colon N\rightarrow M$ be a deflation in $\Dfiltered(\setJ)$,
which we complete to a short exact sequence in $\mathcal{E}$, and let 
$g\colon M'\rightarrow M$ be a morphism in $\Dfiltered(\setJ)$.  Let $N'$ be the 
pullback of $f$ and $g$ in $\mod \aus$,
\begin{equation}\label{Eq:pullback} 
\begin{tikzpicture}[scale=0.7,
 arr/.style={black, -angle 60}, baseline=(bb.base)]
\pgfmathsetmacro{\eps}{0.2}
\coordinate (bb) at (0,1.9);
 
  \path (3,3) node(b1) {$K$};
 \path (6,3) node(b2) {$N'$};
\path (9,3) node(b3) {$M'$};
\path (9,1) node (c3) {$M$};
\path (6,1) node (c2) {$N$};
\path (3,1) node (c1) {$K$};
\path (12, 1) node (c4) {$0$.};
\path (0,1) node (c0) {$0$};
\path (-0.6, 1) node {$\epsilon_1\colon$};

\path (12, 3) node (b4) {$0$};
\path (0,3) node (b0) {$0$};
\path (-0.6, 3) node {$\epsilon_2\colon$};

\path[arr] (b1) edge (b2);
\path[arr] (b2) edge node [auto]{}  (b3);
\path[arr] (c1) edge  (c2);
\path[arr] (c2) edge node [auto]{$f$} (c3);
\path[arr] (c0) edge (c1);
\path[arr] (c3) edge (c4);

\path[arr] (b2) edge node [auto]{} (b3);

\path[arr] (b0) edge (b1);
\path[arr] (b3) edge (b4);

\path[arr] (b3) edge node [auto]{$g$}(c3);
\path[arr] (b2) edge node [auto]{} (c2);
\path[arr] (b1) edge node [auto]{$\verteq$} (c1);
\end{tikzpicture}
\end{equation}

\noindent Restricting to the vertex 0, we have 
\[
\begin{tikzpicture}[scale=0.7,
 arr/.style={black, -angle 60}]
  \path (3,3) node(b1) {$e_0K$};
 \path (6,3) node(b2) {$e_0N'$};
\path (9,3) node(b3) {$e_0M'$};
\path (9,1) node (c3) {$e_0M$};
\path (6,1) node (c2) {$e_0N$};
\path (3,1) node (c1) {$e_0K$};
\path (12, 1) node (c4) {$0$.};
\path (0.2,1) node (c0) {$0$};
\path (-0.8, 1) node {$\e_0(\epsilon_1)\colon$};

\path (12, 3) node (b4) {$0$};
\path (0.2,3) node (b0) {$0$};
\path (-0.8, 3) node {$\e_0(\epsilon_2)\colon$};

\path[arr] (b1) edge (b2);
\path[arr] (b2) edge node [auto]{}  (b3);
\path[arr] (c1) edge  (c2);
\path[arr] (c2) edge node [auto]{$\e_0(f)$} (c3);
\path[arr] (c0) edge (c1);
\path[arr] (c3) edge (c4);

\path[arr] (b2) edge node [auto]{} (b3);

\path[arr] (b0) edge (b1);
\path[arr] (b3) edge (b4);

\path[arr] (b3) edge node [auto]{$\e_0(g)$}(c3);
\path[arr] (b2) edge node [auto]{} (c2);
\path[arr] (b1) edge node [auto]{$\verteq$} (c1);
\end{tikzpicture}
\]

By assumption, the sequence $\epsilon_1$ is admissible and so by \Cref{Lem:exact1}, 
$\e_0(\epsilon_1)$ splits.
Note that $e_0N'$ is the pullback along $\e_0(f)$ and $\e_0(g)$ in 
$\mod e_0De_0$, and so the sequence $\e_0(\epsilon_2)$ also splits. 
Consequently, again by \Cref{Lem:exact1}, 
the sequence $\epsilon_2$ is admissible with $N'\in \Dfiltered(\setJ)$. 

(E2) Let $f\colon N\rightarrow M$ and $h\colon M\rightarrow M''$ be deflations.
We will show that the composition $hf$ is a deflation. Let 
$g\colon M'\to M$
be the kernel of $h$. We have the following commutative diagram, 
extending the pullback in \eqref{Eq:pullback},

\[
\begin{tikzpicture}[scale=0.7,
 arr/.style={black, -angle 60}]
\path (3,3) node(b1) {$K$};
\path (6,3) node(b2) {$N'$};
\path (9,3) node(b3) {$M'$};
\path (9,1) node (c3) {$M$};
\path (6,1) node (c2) {$N$};
\path (3,1) node (c1) {$K$};
\path (12, 1) node (c4) {$0$};
\path (0,1) node (c0) {$0$};

\path (6,-1) node (d2) {$M''$};
\path (9,-1) node (d3) {$M''$};

\path (12, 3) node (b4) {$0$};
\path (0,3) node (b0) {$0$};

\path[arr] (b1) edge (b2);
\path[arr] (b2) edge node [auto]{}  (b3);
\path[arr] (c1) edge  (c2);
\path[arr] (c2) edge node [auto]{$f$} (c3);
\path[arr] (c0) edge (c1);
\path[arr] (c3) edge (c4);

\path[arr] (b2) edge node [auto]{} (b3);

\path[arr] (b0) edge (b1);
\path[arr] (b3) edge (b4);

\path[arr] (b3) edge node [auto]{$g$}(c3);
\path[arr] (b2) edge node [auto]{} (c2);
\path[arr] (b1) edge node [auto]{$\verteq$} (c1);

\path[arr] (c3) edge node [auto]{$h$}(d3);
\path[arr] (c2) edge node [auto]{$h''$} (d2);

\path[arr] (d2) edge node [auto]{$=$}(d3);
\end{tikzpicture}
\]

By (E3), and using \Cref{Lem:exact1}, we know that
\[
\eta N'=\eta K\oplus \eta M'
\]
and so
\[
\eta N=\eta K\oplus \eta M=
\eta K\oplus \eta M'\oplus  \eta M''
=\eta N'\oplus \eta M''.
\]
and so $hf$ (i.e. $h''$) is a deflation, by \Cref{Lem:exact1}.
\end{proof}

\begin{remark}\label{rem:TBr}
\Cref{Lem:exactcategory} can also be deduced from  \cite[Props. 1.4 and 1.10]{DRSKeller}. 
More precisely, \cite[Prop. 1.10]{DRSKeller} implies that  (E2) follows from (E3). 
(E2) is equivalent to that $\ext^1(-, -)$ is a closed  additive bifunctor, and so  
by \cite[Prop. 1.4]{DRSKeller}, $(\Dfiltered(\setJ), \mathcal{E})$ is an exact category.
Nevertheless, we choose to give a direct and more transparent proof.
\end{remark}

Note that surjections in $\Dfiltered$ are not necessarily (categorical) epimorphisms
in $\Dfiltered$. However, admissible surjections are, as the following lemma
shows.

\begin{lemma}
A map $g\colon    L\to N$ in  $\Dfiltered$ is an epimorphism
if and only if $\eta  (g)\colon \eta L\to \eta N$ is surjective.
\end{lemma}
\begin{proof}
Assume that $\eta (g) \colon \eta L\to \eta N$ is surjective.
Then $\cok g$ is not supported at $0$, since 
$\eta L$ and $\eta N$ are submodules generated by 
$e_0L$ and $e_0N$ (\Cref{Lem:trace} (2)), respectively. 
Now, suppose that there is a zero composition 
\[
\begin{tikzpicture}[scale=0.6,
 arr/.style={black, -angle 60}]
\path (9,1) node (c3) {$Z$};
\path (6,1) node (c2) {$N$};
\path (3,1) node (c1) {$L$};

\path (1.5, 1) node {$0=hg$:};

\path[arr] (c1) edge node[auto]{$g$} (c2);
\path[arr] (c2) edge node[auto]{$h$}  (c3);

\end{tikzpicture}
\]
with $Z\in \Dfiltered(\setJ)$. Then $h$ factors as
\[
\begin{tikzpicture}[scale=0.6,
 arr/.style={black, -angle 60}]
\path (9,1) node (c3) {$Z$.};
\path (6,1) node (c2) {$\cok g$};
\path (3,1) node (c1) {$N$};

\path[arr] (c1) edge(c2);
\path[arr] (c2) edge (c3);

\end{tikzpicture}
\]
But $\Hom(\cok g, Z)=0$, since
$\cok g$ is not supported at $0$ and $\soc Z$ is only supported at $0$, by \Cref{Lem:chardelta} (5). 
Therefore $h=0$ and so $g$ is an epimorphism.

Conversely, suppose that $g\colon L\rightarrow N$ is an epimorphism.  
Note that \[g(e_0L)=e_0g(L)\subseteq e_0N \text{ and } \eta N=De_0N,\]
which follows from \Cref{Lem:trace} (2).
Assume for contradiction that $\eta (g)$ is not surjective. Then $e_0\cok g\neq 0$.
Let $Z$ be a quotient of $\cok g$ with $\soc Z$ only supported at vertex $0$, that is 
$Z\in \Dfiltered$ (see \Cref{Lem:chardelta}). Note that such a $Z$ exists by starting with $\cok g$ and
repeatedly removing summands of socles that are not supported at $0$.
Let $h$ be the composition 
\[
\begin{tikzpicture}[scale=0.6,
 arr/.style={black, -angle 60}]
\path (9,1) node (c3) {$Z$,};
\path (6,1) node (c2) {$\cok g$};
\path (3,1) node (c1) {$N$};
\path (2.2,1) node {$h$:};

\path[arr] (c1) edge  (c2);
\path[arr] (c2) edge   (c3);
\end{tikzpicture}
\]
which is in particular not zero, 
but  the composition 
\[
\begin{tikzpicture}[scale=0.6,
 arr/.style={black, -angle 60}]
\path (9,1) node (c3) {$Z$};
\path (6,1) node (c2) {$N$};
\path (3,1) node (c1) {$L$};

\path[arr] (c1) edge node[auto]{$g$} (c2);
\path[arr] (c2) edge node[auto]{$h$}  (c3);
\end{tikzpicture}
\]
is zero, contradicting that $g$ is an epimorphism.
Therefore, $\eta (g)$ must be surjective.
\end{proof}

\subsection{Exactness of the functor $\projfun$.}
\begin{proposition} \label{Lem:liftedprecise}
Let
\[
\begin{tikzpicture}[scale=0.6,
 arr/.style={black, -angle 60}]
 \path (5,16) node(a1) {$L$};
 \path (8,16) node(a2) {$M$};
\path (11,16) node(a3) {$N$};
 \path (2,16) node(a0) {$0$};
 \path (14,16) node(a4) {$0$};
 \path (1.2,16) node {$\alpha$:};

 \path[arr] (a0) edge  (a1);
 \path[arr] (a1) edge  (a2);
 \path[arr] (a2) edge  (a3);
 \path[arr] (a3) edge  (a4);
 \end{tikzpicture}
\]
be an extension in $\sub Q_{\setJ}$. Then $\alpha$ has a lift in $\mathcal{E}$ as follows, 
\[
\begin{tikzpicture}[scale=0.6,
 arr/.style={black, -angle 60}]
 \path (5,16) node(a1) {$\L$};
 \path (9,16) node(a2) {$\M\oplus \charT'$};
\path (13,16) node(a3) {$\N$};
 \path (2,16) node(a0) {$0$};
 \path (16,16) node(a4) {$0$,};
 
 \path[arr] (a0) edge  (a1);
 \path[arr] (a1) edge  (a2);
 \path[arr] (a2) edge  (a3);
 \path[arr] (a3) edge  (a4);
 \end{tikzpicture}
\]
where $\L, \M$ and $\N$ are minimal lifts, and
$\charT' = \oplus_i \charT_{i}^{d_i}\in \add \charT_\setJ$ with
\begin{equation}\label{eq:nsocdim}
(d_i)_i = \dvector \soc (L\oplus N) - \dvector \soc M.
\end{equation}
\end{proposition}
\begin{proof}
We construct the commutative diagram

\[
\begin{tikzpicture}[scale=0.42,
 arr/.style={black, -angle 60}]
 
\path(1, 16) node(a) {$\alpha$:};
\path (6,16) node(a1) {$L$};
\path (12,16) node(a2) {$M$};
\path (18,16) node(a3) {$N$};
\path (2.2,16) node(a0) {$0$};
\path (21.2,16) node(a4) {$0$};
\path[arr] (a0) edge (a1);
\path[arr] (a3) edge (a4);

\path(-1, 13) node(b) {$\beta$:};
\path (3,13) node(b1) {$\hat{L}$};
\path (9,13) node(b2) {$\M$};
\path (15,13) node(b3) {$\N$};
\path (0,13) node(b0) {$0$};
\path (19,13) node(b4) {$0$};
\path[arr] (b0) edge (b1);
\path[arr] (b3) edge (b4);

\path(1, 10) node(g) {$\gamma$:};
\path (6,10) node(g1) {$Q_L$};
\path (12,10) node(g2) {$Q_M$};
\path (18,10) node(g3) {$Q_N$};
\path (2,10) node(g0) {$0$};
\path (21.2,10) node(g4) {$0$};
\path[arr] (g0) edge (g1);
\path[arr] (g3) edge (g4);

\path(-1, 7)  node(d) {$\delta$:};
\path (3,7) node(d1) {$P_L$};
\path (9,7) node(d2) {$P_M$};
\path (15,7) node(d3) {$P_N$};
\path (0,7) node(d0) {$0$};
\path (19,7) node(d4) {$0$};
\path[arr] (d0) edge (d1);
\path[arr] (d3) edge (d4);
\path[arr] (a1) edge  (a2);
\path[arr] (a2) edge node[auto]{}  (a3);
\path[arr] (b1) edge  (b2);
\path[arr] (b2) edge node[auto]{} (b3);
\path[arr] (g1) edge  (g2);
\path[arr] (g2) edge  (g3);
\path[arr] (d1) edge  (d2);
\path[arr] (d2) edge  (d3);
\path[arr] (a1) edge  (g1);
\path[arr] (a2) edge  (g2);
\path[arr] (a3) edge  (g3);
\path[arr] (d1) edge  (g1);
\path[arr] (d2) edge  (g2);
\path[arr] (d3) edge  (g3);
\path[arr] (b1) edge  (a1);
\path[arr] (b2) edge  (a2);
\path[arr] (b3) edge  (a3);
\path[arr] (b1) edge  (d1);
\path[arr] (b2) edge  (d2);
\path[arr] (b3) edge  (d3);

\end{tikzpicture}
\]
where

\begin{itemize}
\item[(i)] $Q_L, ~ Q_M, ~Q_N\in \add Q_{\setJ}$, 
$P_L, ~ P_M, ~P_N\in \add D$, and the existence of $\gamma$ and 
$\delta$ follows from the Horseshoe Lemmas.

\item[(ii)]  The modules $\L$, $\M$ and $\N$  are pullbacks.
Since the sequences $\alpha$, $\gamma$ and $\delta$ are exact,  the sequence $\beta$ is  exact.
\end{itemize}
Note that
 \[  
 \eta P_M=\eta P_L\oplus \eta P_N,
 \] and so
 by \Cref{Lem:dense} (1),
 \[  
 \eta \M=\eta \L\oplus \eta \N.
 \]
That is, $\beta$ is admissible (\Cref{Lem:exact1} (6)). By construction, $\projfun(\beta)=\alpha$. 
Furthermore, 
\[
\soc Q_L\oplus \soc Q_N=\soc Q_M,
\] 
So the equality \eqref{eq:nsocdim} follows from Lemma \ref{Lem:dense} (3) and
 $\beta$ is a sequence as required.
\end{proof}
\begin{theorem} \label{Thm:fullexact}
The functor
$$
\projfun\colon (\Dfiltered(\setJ),\mathcal{E}) \to \sub Q_{\setJ}, M\mapsto M/\eta M
$$
is dense, full, exact and induces isomorphism on the extension groups
$$
\ext^1(M,N)\cong \Ext^1(\projfun M, \projfun N).
$$
\end{theorem}
\begin{proof}
That $\projfun$ is exact follows from \Cref{Lem:exact1}. That the functor is dense and full
on maps follows from the refinement of the Dlab--Ringel equivalence given in
\Cref{Thm:equiv}. The fullness on extensions follows from \Cref{Lem:liftedprecise}.
Finally, let
\[
\begin{tikzpicture}[scale=0.6,
 arr/.style={black, -angle 60}]
 \path (5,16) node(a1) {$L$};
 \path (8,16) node(a2) {$M$};
\path (11,16) node(a3) {$N$};
 \path (2,16) node(a0) {$0$};
 \path (14,16) node(a4) {$0$};
 \path (1.2,16) node {$\xi$:};
 
 \path[arr] (a0) edge  (a1);
 \path[arr] (a1) edge  (a2);
 \path[arr] (a2) edge node[auto] {$g$} (a3);
 \path[arr] (a3) edge  (a4);
 \end{tikzpicture}
\]
be an admissible extension in $\Dfiltered(\setJ)$ with $\projfun(\xi)$ split.
In particular, $\projfun(g)$ is a split epimorphism. Let $\overline{h}\colon \projfun N \to \projfun M$ 
be a homomorphism such that $\overline{h}\projfun(g)=\mathrm{Id}_{\projfun N}$ 
and let $h$ be a lift of $\overline{h}$. Then, by the lifting of maps in 
\Cref{Lem:liftofmaps}, there exists a map $u\colon N\to \eta N$, so that
\[gh =\mathrm{Id}_N+u\] 
Note that $\eta(\xi)$ splits (see \Cref{Lem:exact1}). Let $u'$ be the composition of $u$ with the splitting  
of $\eta(g)$ so that
\[gu'=u.\]
Let 
 $h'=h-u'$. Then 
 \[
 gh'=gh-gu'=\mathrm{Id}_N+u-u=\mathrm{Id}_N.
 \]
Therefore $\xi$ splits. Consequently, the two extension groups are isomorphic. 
This completes the proof.
\end{proof}
\subsection{Frobenius and $2$-CY structure of $\Dfiltered(\setJ)$}
Let $P_{n-j}$ be the {\em minimal lift of the projective-injective} $Q_j\in \sub Q_{\setJ}$ to
$\Dfiltered(\setJ)$.  Note that $P_j$ has top $S_j$ and $Q_j$ has socle $S_j$, 
and that $P_{n-j}=D_{n-j}$ if and only if $j\in \setJ$, by \Cref{Lem:whichrankone}. Let 
\begin{equation}\label{eq:projgenerator}
P=(\oplus_{j\in [n-1]}P_j) \bigoplus \charT_\setJ.
\end{equation}

\begin{theorem} \label{Thm:Frob2CY} The category 
$\Dfiltered(\setJ)$ with the exact structure $\mathcal{E}$ is a Frobenius stably 
$2$-Calabi-Yau category with indecomposable projective-injective objects 
$P_1,\cdots,P_{n-1}$ and $\charT_j$ for $j\in \setJ$.
\end{theorem}
\begin{proof}
Note that $\projfun P=Q_{\setJ}$ and 
 $\projfun$ induces 
 isomorphisms on extension groups (\Cref{Thm:fullexact}), it follows that 
modules in $\add P$ are projective and injective in 
$\Dfiltered(\setJ)$. By the uniqueness of lifts, up to summands in $\add \charT_J$,
shown in \Cref{Lem:dense}, there are no other indecomposable 
projective-injectives in $\Dfiltered(\setJ)$. By lifting extensions, using   
\Cref{Lem:liftedprecise}, we can conclude that  $\Dfiltered(\setJ)$ has enough 
projective-injectives. Hence $\Dfiltered(\setJ)$ is a Frobenius exact category.

By \Cref{Thm:fullexact} and using that $\sub Q_{\setJ}$ is 
stably 2-CY \cite{GLS08}, it follows that $\Dfiltered(\setJ)$ is stably 
$2$-Calabi-Yau.
\end{proof}
\begin{remark} \label{Rem:missingobject}
Both $\mathcal{F}_\Delta/\add \charT$ and $\Dfiltered([n-1])/\add \charT_{[n-1]}$ 
are equivalent to $\mod \Pi$. Similar to \Cref{Thm:Frob2CY}, 
$\Dfiltered$ is also a Frobenius stably $2$-Calabi-Yau category; it has one additional
indecomposable object,  $\charT_n=D_0$, compared to
$\Dfiltered([n-1])$. More generally, we can add  the module $\charT_n$ to any
$\Dfiltered(\setJ)$ to create a Frobenius stably $2$-Calabi-Yau category with
one additional indecomposable projective-injective object. This amounts to adding 
$n$ to $\setJ$. 
However, 
unless otherwise stated,  we keep the assumption that  $\setJ\subseteq [n-1].$ 
\end{remark}
\subsection{The Gabriel quiver of $(\End  P)^{\mathrm{op}}$}
In this subsection, we show that the Gabriel quiver of $(\End  P)^{\text{op}}$ has no loops, where
$P$ is the projective-injective generator of
 $\Dfiltered(\setJ)$ defined in \eqref{eq:projgenerator}.
The lemma below  follows from the description of the socles of the indecomposable 
projective-injective $Q_j\in \sub Q_{\setJ}$ in \Cref{Lem:soclesub} and properties of 
$\eta M$ in  \Cref{Lem:trace}.

\begin{lemma} \label{Lem:charTsub}
Let $\setJ=\{j_1<j_2<...< j_r\}$.
\begin{itemize}
\item[(1)] If $j\in \setJ$, then $P_{n-j}=D_{n-j}$  and $\eta P_{n-j}=\charT_j$.
\item[(2)] If $j_i< j<j_{i+1}$, then $\eta P_{n-j}=\charT_{j_i}\oplus \charT_{j_{i+1}}$.
\item[(3)] If $1\leq j<j_1$, then $\eta P_{n-j}= \charT_{j_1}$.
\item[(4)] If $j_r<j \leq n-1$, then $\eta P_{n-j}= \charT_{j_r}$.
\end{itemize}
\end{lemma}

\begin{lemma}\label{Lem:noloops1}
Let $\setJ=\{j_1<j_2<...< j_r\}$ and let $s=j_i<j<j_{i+1}=t$. 
Then there is an exact sequence
\[
\begin{tikzpicture}[scale=0.6,
 arr/.style={black, -angle 60}]
 \path (5,16) node(a1) {$D_{n-j}$};
 \path (8.3,16) node(a2) {$P_{n-j}$};
\path (11.6,16) node(a3) {$\charT_{s+t-j}$};
 \path (2,16) node(a0) {$0$};
 \path (14.6,16) node(a4) {$0$.};
 
 \path[arr] (a0) edge  (a1);
 \path[arr] (a1) edge  (a2);
 \path[arr] (a2) edge node[auto] {} (a3);
 \path[arr] (a3) edge  (a4);
 \end{tikzpicture}
\]
\end{lemma}
\begin{proof}
The module $P_{n-j}$ can be illustrated as follows, 

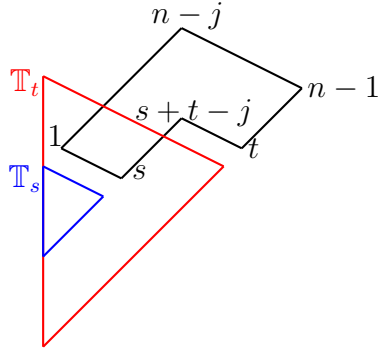
\[
\begin{tikzpicture}[scale=0.7,arr/.style={black,  -angle 60}, thick]
\path (4, 2.2) node (1) {$\small{1}$};
\path (6.2, 4.2) node (2) {$\small{n-j}$};
\path (8.8, 3) node (3) {$\small{n-1}$};
\path (5.4, 1.6) node (4){$\small{s}$};
\path (6.3, 2.6) node (5) {$\small{s+t-j}$};
\path (7.3, 2) node (6) {$\small{t}$};

\draw[-] (4.1, 2)  to (6.1, 4);
\draw[-] (4.1, 2)  to (5.1, 1.5);
\draw[-] (6.1, 4) to  (8.1, 3);
\draw[-] (8.1, 3)  to (7.1, 2);
\draw[-] (5.1, 1.5)  to (6.1, 2.5);
\draw[-] (6.1, 2.5) to (7.1, 2);

\path(3.5, 3.1) node{ \color{red} $\charT_t$};
\draw[red, -]     (3.8, -1.3) to (3.8, 3.2);
\draw[red, -]     (3.8, 3.2) to (6.8, 1.7);
\draw[red, -]     (6.8, 1.7) to (3.8,  -1.3);

\path(3.5, 1.5) node {\color{blue} $\charT_s$};
\draw[blue, -] (3.8, 1.7) to (4.8, 1.2);
\draw[blue, -]  (3.8, 1.7) to (3.8, 0.2);
\draw[blue, -]  (4.8, 1.2) to (3.8, 0.2);
\end{tikzpicture}
\]
where the black part is $Q_{j}$, sitting on top of $\charT_s$ and $\charT_t$, and 
$\soc Q_j=S_s\oplus S_t$.

The submodule generated by $e_{n-j}P_{n-j}$ is isomorphic to the projective 
$D$-module $D_{n-j}$. The quotient module $P_{n-j}/ D_{n-j}$ is 
$\charT_{s+t-j}$. So we have the short exact sequence as claimed.
\end{proof}

\begin{example}\label{Ex:liftsofproj}
Let $n=6$ and $\setJ=\{2, 4\}$. We illustrate the projective objects $P_1, \dots, P_5$ in 
$\Dfiltered(\setJ)$ as follows,  where in each projective object, the black part is the corresponding projective object in $\sub Q_{\setJ}$.
In particular,  
\[
 Q_3= S_2\oplus S_4 \text{ and }
\eta P_3=\charT_2 \oplus \charT_4.
\]
The summands $\charT_2$ and $\charT_4$ in each $\eta P_i$ are 
in blue and red, respectively. 
We can check directly that 
\[
P_3/D_3\cong \charT_3,
\]
confirming \Cref{Lem:charTsub}.

\[ 
\begin{matrix} 
P_1\colon && &1   \\   &&  {\color{red} 0}& &2  \\   &&& {\color{red} 1} && 3   \\ &&  {\color{red} 0} &&{\color{red} 2} &&4   \\
 && & {\color{red} 1} && {\color{red} 3}    \\
&&  {\color{red} 0}& &{\color{red} 2} \\   && &{\color{red} 1}    \\ && {\color{red} 0}   \\        &&
\end{matrix}~~~~
\begin{matrix} &P_2\colon&   &&&2\\
&& &&1 &&3\\ &&  &{\color{red} 0}& &2 &&4\\&&  && {\color{red} 1} && 3 &&5\\&& &{\color{red} 0} &&{\color{red} 2} &&4\\
&&  && {\color{red} 1} && {\color{red} 3} \\ &&
 & {\color{red} 0}& &{\color{red} 2} \\ && &&{\color{red} 1}\\ && &{\color{red} 0}
\end{matrix}
\begin{matrix} 
&& P_3\colon &&&&&3  \\&&&&  {\color{red} 0} &&2&&4 \\
&&&&&  {\color{red} 1}1& &3&&5  \\
&&&&  {\color{red} 0} {\color{blue} 0} &&{\color{red} 2}2&&4 \\
&&&&&  {\color{red} 1} {\color{blue} 1}& & {\color{red} 3}  \\
&&&&  {\color{red} 0} {\color{blue} 0}&&{\color{red} 2} \\
&&& &&{\color{red} 1}\\ && &&{\color{red} 0}
\end{matrix}
\]

\[
\begin{matrix}
P_4\colon  &&&&&4\\
 &&&&3&&5\\   &&&2&&4\\  &&1&&3\\    &{\color{blue} 0}  &&2\\ 
 &&{\color{blue} 1} \\  &{\color{blue} 0} \\ 
\end{matrix}
\begin{matrix}
&  &&P_5\colon &&&&&&5 \\ 
&  && &&&&&4\\&  && &&&&3\\&   && &{\color{blue} 0}  &&2\\ 
&  && &&{\color{blue} 1} \\&  && & {\color{blue} 0} \\&
\end{matrix}
\]

\end{example}
\begin{proposition} \label{Lem:noloops}
The Gabriel quiver of $(\End  P)^{\text{op}}$ has no loops.
\end{proposition}
\begin{proof} Let $Q$ be the Gabriel quiver of $(\End  P)^{\text{op}}$.
Note that  $\projfun P=Q_\setJ$ and  $(\End Q_\setJ)^{\text{op}}$
is a quotient algebra of $\Pi$.
Therefore,
 the Gabriel quiver of $(\End \projfun P)^{\text{op}}$ has no loops. As
$\pi$ is full (see \Cref{Thm:fullexact}) and by \Cref{Lem:factormaps}, there is no loop in $Q$ incident at a vertex 
corresponding to $P_j$. Thus, any loop in $Q$ must occur at a 
vertex corresponding to some $\charT_j$ with $j\in \setJ$. 

Let $\setJ = \{j_1 < \cdots < j_r\}$ and $j\in \setJ$.
By \Cref{Lem:charTequivalence}, 
 \[(\End \charT_j)^{\text{op}}=\CC[t]/(t^j). \]
We proceed with the proof by examining each case as follows. 
\begin{itemize}
\item[(i)] Either $j-1$ or $j+1$ belongs to $\setJ$ and denote this element by $s$. Then $t\colon \charT_j\ra \charT_j$ factors through $ \charT_s$ and so there is no loop at $\charT_j$. 

\item[(ii)]  $j=j_1$. If $j_1=1$, then $\charT_j$ is simple and so there are no loops at 
$\charT_j$. Otherwise, the map $t\colon \charT_j\to \charT_j$ factors through $P_{n-1}$.

\item[(iii)] $j=j_r$. Then $t$ factors through $P_1$, because
$\eta P_1=\charT_{j_r}$, by \Cref{Lem:charTsub} (4).
\item[(iv)] Finally, $j=j_i <j+1< j_{i+1}$. Denote $j_{i+1}$ by $l$.   Then $t$ factors through 
 $P_{n-(l-1)}$. More precisely, the map $t\colon \charT_j\ra \charT_j$ is the following composition, 
 \[
 \charT_j \hookrightarrow  P_{n-(l-1)} \twoheadrightarrow \charT_{j+1} \twoheadrightarrow \charT_j,
 \] 
 where the first map is the embedding of $\charT_j$ into $\eta P_{n-(l-1)}=\charT_j\oplus \charT_l$, and 
 the second map is the surjection described in  \Cref{Lem:noloops1}. 
 \end{itemize}
 Therefore, there are no loops at $\charT_j$. This completes the proof.
\end{proof}

\section{Grothendieck groups, weights and bilinear forms}\label{sec5}
We identify weight and root lattices of $\gl_n$ with Grothendieck groups in such a way 
that the class of a module  will be the weight of its cluster  character (constructed 
in \Cref{sec9}). We also construct bilinear forms on the Grothendieck groups, which will play 
important roles in understanding flag combinatorics and the quantum cluster structure on the 
coordinate rings of flag varieties in later sections.
\subsection{The Grothendieck group $K(\Dfiltered)$ and the weight lattice $\wtl$}
\label{sec:5.1}
Let  $K(D)$, $K(\Pi)$ be the Grothendieck groups of 
$\mod \aus$ and $\mod \Pi$, respectively. Note that the  \emph{inclusion (of groups)}, 
\[K(\Pi)\subseteq K(D),\] 
since $\mod \Pi$ is an abelian subcategory of $\mod D$.
The category $\Dfiltered$, has 
two exact structures, one inherited from $\mod D$ and the other one  
 $\mathcal{E}$. We denote the two Grothendieck groups by $K(\Dfiltered)$ and 
$\grothE$, respectively. 
We have the inclusion, 
\[ 
K(\Dfiltered) \subseteq K(D), 
\]
by mapping a class in $K(\Dfiltered)$ to the corresponding class in $K(D)$.
This inclusion is an isomorphism 
\[ 
K(\Dfiltered)= K(D), 
\]
since we have $S_0=\Delta_1$, and the short exact sequence,
\begin{equation} \label{Eq:simpledelta}
\begin{tikzpicture}[scale=0.6,
 arr/.style={black, -angle 60}, baseline=(bb.base)]
 \coordinate (bb) at (0, 1.8);
 
 \path (5,2) node(a1) {$\Delta_{i}$};
 \path (8.1,2) node(a2) {$\Delta_{i+1}$};
\path (11.2,2) node(a3) {$S_i$};
 \path (2,2) node(a0) {$0$};
 \path (14.2,2) node(a4) {$0$}; 
 
 \path[arr] (a0) edge  (a1);
 \path[arr] (a1) edge  (a2);
 \path[arr] (a2) edge  (a3);
 \path[arr] (a3) edge  (a4);
 \end{tikzpicture}
\end{equation}
for any $i\geq 1$, and therefore $[S_i]= -[\Delta_{i}] +[\Delta_{i+1}]\in K(\Dfiltered)$.

We have the inclusion 
\[
K(\add \charT)\subseteq K(\Dfiltered)
\]
by mapping a class in $K(\add \charT)$ to its corresponding class in $K(\Dfiltered)$.
This inclusion is an isomorphism, since for any $1\leq i\leq n-1$, there is a short exact sequence
\[
\begin{tikzpicture}[scale=0.6,
 arr/.style={black, -angle 60}]
 \path (5,16) node(a1) {$\Delta_{i+1}$};
 \path (8.1,16) node(a2) {$\charT_{i+1}$};
\path (11.2,16) node(a3) {$\charT_i$};
 \path (2,16) node(a0) {$0$};
 \path (14.2,16) node(a4) {$0$,};

 \path[arr] (a0) edge  (a1);
 \path[arr] (a1) edge  (a2);
 \path[arr] (a2) edge  (a3);
 \path[arr] (a3) edge  (a4);
 \end{tikzpicture}
\]
and $\Delta_1=\charT_{1}=S_0$. Note that $\charT_n=D_0$ is the projective $D$-module associated to 
vertex $0$.
So $K(\Dfiltered)$ has four natural bases,
\begin{equation*}
\{[\Delta_1], \cdots, [\Delta_{n}]\}, ~ 
\{[S_0], [S_1], \cdots, [S_{n-1}]\}, ~
\end{equation*}
and the bases of classes of indecomposable summands of $\charT$ and $\aus$.
Expressing the class of a $\aus$-module $M$ in the first two bases gives us
the $\Delta$-dimension vector $\dvector_\Delta M$ and dimension vector $\dvector M$
of $M$. Note that entries in $\dvector_\Delta M$ can be 
negative, when $M$ is not in $\Dfiltered$.  

Denote by $\wtl$ and $\rootlat$ the weight lattice and root lattice of $\gl_n$, respectively. 
Let $\epsilon_1, \dots, \epsilon_n$ be the standard basis of the weight lattice
$\wtl$. For any $1\leq i\leq n$ and $1\leq j\leq n-1$, 
let 
\[
\omega_i=\epsilon_1+\dots +\epsilon_i \text{ and } \alpha_j=-\epsilon_{j}+\epsilon_{j+1}.
\] 
Then $\omega_1, \dots, \omega_n$ are the
 fundamental weights of $\gl_n$, and 
 $\alpha_1, \dots, \alpha_{n-1}$ are the ({\em negative}) simple roots of $\gl_n$.  

Denote by $\wtl^+$ the {\em monoid of dominant weights} in $\wtl$, which 
are non-negative linear combinations of the fundamental weights. The dominant weights 
are the {\em highest weights} of irreducible representations of $\gl_n$. 
For the Grothendieck group $K(\mathcal{C})$ of an exact category $\mathcal{C}$, denote by 
$K^{+}(\mathcal{C})$ the {\em monoid of   classes of objects} in $\mathcal{C}$.   
Let
\begin{equation}\label{Eq:gtowt1}
\zeta \colon \wtl \to K(\Dfiltered), \;~~ \epsilon_i \mapsto  [\Delta_{i}],
\end{equation}

\begin{lemma} \label{Lem:Gammaisom}
The map $\zeta$ is an isomorphism, identifying $\wtl$ with $K(\Dfiltered)$, and induces the following isomorphisms.  
\begin{itemize}
\item[(1)] $\zeta\colon\rootlat \to K(\Pi)\subseteq K(\Dfiltered), ~ \alpha_i\mapsto [S_i]$. 

\item[(2)] $\zeta\colon \wtl^+\to K^+(\add\charT)$, $\omega_i \mapsto [\charT_i]$.
\end{itemize}
Moreover, it maps the (lowest) weight $\sum_{j> i} \epsilon_j$  to $[D_{i}]$.
\end{lemma} 
\begin{proof} It is clear from the definition that $\zeta$ defines an isomorphism. 
The remaining statements  follow from the direct computations below, 
\[
\zeta(\alpha_i)=\zeta(-\epsilon_{i}+\epsilon_{i+1})=-[\Delta_{i}]+[\Delta_{i+1}]=[S_i], 
\]
where the last equality follows from \eqref{Eq:simpledelta};
\[
\zeta(\omega_i)=\zeta(\epsilon_1+\dots+ \epsilon_{i})=[\Delta_1]+\dots+[\Delta_{i}]=[\charT_i];
\]
and finally 
\[
\zeta(\sum_{j> i} \epsilon_i)=[\Delta_{i+1}]+\dots+[\Delta_{n}]=[D_{i}].
\]
\end{proof}

\subsection{A bilinear form on $K(\Dfiltered)$}
Recall the bilinear form on $K(\Dfiltered)$ from \cite{JS1}, 
\begin{equation}\label{Eq:bilin0}
\langle [X], [Y] \rangle = 
\dim \Hom(X,Y) - \dim \Ext^1(X,Y) = \sum_{i=1}^{n} x_i y_i, 
\end{equation}
where $X, Y\in \Dfiltered$ with $\dvector_\Delta X = (x_i)_i$ and 
$\dvector_\Delta Y = (y_i)_i$.
Observe that this bilinear form computes the dot product of $\Delta$-dimension vectors,
and via the isomorphism $\zeta$ in \eqref{Eq:gtowt1},
it is the dot product on $\wtl$ with respect to the standard basis. Therefore for any $1\leq i, j\leq n-1$, 
\begin{equation}\label{Eq:bilin1}
\langle [S_i], [S_j] \rangle=\left\{ \begin{tabular}{ll} 2 & if $i=j$;\\ -1 & if $i=j+1$ or $j=i+1$;\\
0 & otherwise. \end{tabular} \right.
\end{equation}
This also follows by  the direct computation, 
\[
\langle [S_i],  [S_j] \rangle = \langle -[\Delta_i] + [\Delta_{i+1}], -[\Delta_j] - [\Delta_{j+1}]\rangle 
= \delta_{ij} - \delta_{i+1,j} - \delta_{i,j+1} + \delta_{i+1,j+1}.\]
Therefore, for any $\dv c=(c_i)$ and $\dv d=(d_i)\in K(\Pi)\subseteq K(\Dfiltered)$ written as coordinate vectors 
with respect to  the classes of the simples, we have 
\begin{equation}\label{Eq:bilin2}
\langle \dv c, \dv d\rangle=2\sum_{i=1}^{n-1} c_id_i- \sum_{i=1}^{n-2} (c_id_{i+1}+c_{i+1}d_i),
\end{equation}
which is the bilinear form defined by the Cartan matrix of type $\mathbb{A}_{n-1}$.

For completeness,  we  compute the bilinear form of $[S_0]$ with classes of the simples, 
\[\langle [S_0], [S_0] \rangle = 1, ~\langle [S_{0}], [S_1]\rangle = -1 
\text{ and } \langle [S_0], [S_i] \rangle = 0, \text{ when } i>0.
\]

\subsection{The Grothendieck group $K(\Dfiltered(\setJ))_{\mathcal{E}}$}
Let $\grothEJ$ be the Grothendieck group 
of $\Dfiltered(\setJ)$ with the exact structure $\mathcal{E}$.
Note that this is the subgroup of the Grothendieck group $\grothE$
 spanned by the classes of modules in  
$\Dfiltered(\setJ)$.

\begin{proposition}\label{Lem:isom1}
We have an isomorphism
$$
\sigma\colon \grothE \rightarrow K(\Pi) \oplus K(\add \charT), ~~
[M]\mapsto [\projfun M] + [\eta M].
$$
Moreover, restriction of $\sigma$ induces  an isomorphism, 
$$
\sigma\colon \grothEJ \rightarrow K(\sub Q_{\setJ})\oplus K(\add \charT_{\setJ}).
$$
\end{proposition}
\begin{proof} First note that  the map $\sigma$ is well-defined,
since $\pi\colon(\Dfiltered, \mathcal{E}) \to \mod \Pi$ and $\eta\colon (\Dfiltered, \mathcal{E})\to \add \charT$ 
are exact, by \Cref{Thm:fullexact} and \Cref{Lem:exact1}, respectively.  We have 
\begin{equation}\label{Eq:simplescanlift}
\sigma([M_{[i-1]\cup\{i+1\}}])= [S_i]+[\charT_i] \text{ and } \sigma([\charT_i])=  [\charT_i].
\end{equation}
Therefore, the map is surjective, since  
\[
 \{ [S_i] + [\charT_i], ~[\charT_j]\colon i>0, j\}
\] 
is a basis of 
$K(\Pi)\oplus K(\add\charT)$. To prove the first isomorphism, it remains to show that 
\[
\alpha=\{ [S_i]+ [\charT_i], [\charT_j] \colon  i>0, j\}\subseteq K(\Dfiltered)_{\mathcal{E}}.
\] 
 spans $K(\Dfiltered)_{\mathcal{E}}$. 

We proceed by induction on $\dim \pi M$. First, classes of minimal lifts of simples from $\mod \Pi$ are in $\alpha$,
by \eqref{Eq:simplescanlift}. Let $M\in \Dfiltered(\setJ)$ and let $S$ be a summand 
of the top of $\projfun M$. Consider the short exact sequence, 
\[
\begin{tikzpicture}[scale=0.6,
 arr/.style={black, -angle 60}]
 \path (5,16) node(a1) {$Y$};
 \path (8.1,16) node(a2) {$\projfun M$};
\path (11.2,16) node(a3) {$S$};
 \path (2,16) node(a0) {$0$};
 \path (14.2,16) node(a4) {$0$.}; 
 
 \path[arr] (a0) edge  (a1);
 \path[arr] (a1) edge  (a2);
 \path[arr] (a2) edge  (a3);
 \path[arr] (a3) edge  (a4);
 \end{tikzpicture}
\]
By \Cref{Lem:liftedprecise}, there is a lift to $\Dfiltered$,
\[
\begin{tikzpicture}[scale=0.6,
 arr/.style={black, -angle 60}]
 \path (5,16) node(a1) {$\hat{Y}$};
 \path (9,16) node(a2) {$\widehat{\projfun M}\oplus \charT'$};
\path (13,16) node(a3) {$\hat{S}$};
 \path (2,16) node(a0) {$0$};
 \path (16,16) node(a4) {$0$,}; 
 
 \path[arr] (a0) edge  (a1);
 \path[arr] (a1) edge  (a2);
 \path[arr] (a2) edge  (a3);
 \path[arr] (a3) edge  (a4);
 \end{tikzpicture}
\]
where $\charT'\in \add \charT$. 
By induction, $[\hat{Y}]$ and $[\hat{S}]$ are in the span of $\alpha$ and thus so is 
$[\widehat{\projfun M}]$. By \Cref{Lem:dense}, $M$ and $\widehat{\projfun M}$ 
are isomorphic up to summands from $\add \charT$. It follows that 
$\alpha$ spans $K(\Dfiltered(\setJ))$.

For the second isomorphism, first note that 
\[[S_{n-i}]=[Q_i]-[\rad Q_i]\in 
K(\sub Q_\setJ)\subseteq K(\Pi).\] 
Therefore 
\[K(\sub Q_\setJ)=K(\Pi).\]
Also, as $M\in \Dfiltered(\setJ)$, by definition, $\eta M\in \add \charT_\setJ$. 
So the restriction is well defined. By \Cref{Lem:dense}, any $N\in \sub Q_{\setJ}$
has a lift $\N\in \Dfiltered(\setJ)$ and so the restriction is surjective, thus
an isomorphism.
\end{proof}
\begin{remark}\label{Rem:missingobject2}
By \Cref{Rem:missingobject}, it is possible to include  $n$ in $\setJ$. In this case,   
the Grothendieck group $\grothEJ$ contains the fundamental weight $w_n=[\charT_n]$ and 
 \Cref{Lem:isom1} is still true.
\end{remark}
\subsection{A bilinear form on $\grothEJ$}
Similar to \eqref{Eq:bilin0}, we define a symmetric bilinear form on 
$K(\Dfiltered)_{\mathcal{E}}$ by,
\begin{equation}\label{eq:roundbil}
([X], [Y]) = \langle [X], [Y]\rangle + \langle [\eta X], [\eta Y] \rangle,
\end{equation}
which, by \Cref{Lem:Symbil} below, makes computation of $\ext^1$ easier
(see \Cref{Ex:newex}).
We emphasise that we are using the exact structure $\mathcal{E}$, 
which ensures that $\eta$ is additive on short exact sequences, and 
so \eqref{eq:roundbil} does indeed give a well-defined and 
bilinear form on $\grothEJ$.

\begin{lemma} \label{lem:dimlemma}
Let $X,Y\in \Dfiltered(\setJ)$. 
\begin{itemize}
\item[(1)]  $\Hom(X, \projfun Y) \cong \Hom(\projfun X, \projfun Y)$.
\item[(2)]$\dim \Hom(X, Y) = \dim \Hom(X, \eta Y) + \dim \Hom(\projfun X, \projfun Y)$.
\end{itemize}
\end{lemma}
\begin{proof}
(1) is true,  since any
map $X\rightarrow \projfun Y$ factors through $X/\eta X$.
(2) follows by applying $\Hom(X,-)$ to the exact sequence
\[
\begin{tikzpicture}[scale=0.6,
 arr/.style={black, -angle 60}]
 \path (5,16) node(a1) {$\eta Y$};
 \path (8,16) node(a2) {$Y$};
\path (11,16) node(a3) {$\projfun Y$};
 \path (2,16) node(a0) {$0$};
 \path (14,16) node(a4) {$0$,}; 
 
 \path[arr] (a0) edge  (a1);
 \path[arr] (a1) edge  (a2);
 \path[arr] (a2) edge  (a3);
 \path[arr] (a3) edge  (a4);
 \end{tikzpicture}
\]
and the fact that $\Ext^1(-, \charT)=0$ on $\Dfiltered(\setJ)$ from \Cref{Lem:tiltingchar}.
\end{proof}

\begin{proposition}\label{Lem:Symbil}
Let $X, Y\in \Dfiltered(\setJ)$. Then 
$$
( [X], [Y]) = \dim \Hom(X,Y) + \dim \Hom(Y,X) - \dim \ext^1(X,Y).
$$
\end{proposition}
\begin{proof}
Viewing $\projfun X, \projfun Y$ as $D$-modules, we have 
 \[
[\projfun X]=[X]-[\eta  X] , ~~[\projfun Y]=[Y]-[\eta  Y]\in K(\Dfiltered).
\]
So by the definition of the bilinear form \eqref{Eq:bilin0},
\[\langle [\projfun X], [\projfun Y] \rangle =
\langle [X], [Y] \rangle  - \langle [X], [\eta Y] \rangle 
- \langle [Y], [\eta X] \rangle  + \langle [\eta X], [\eta Y] \rangle, 
\]
using the symmetry and bilinearity of the form.
On the other hand, by the homological interpretation of \eqref{Eq:bilin2} 
by  Crawley-Boevey \cite{CB}, 
\[
\langle [\projfun X], [\projfun Y]\rangle=
\dim \Hom(\projfun X,\projfun Y) + \dim \Hom(\projfun Y,\projfun X) 
- \dim \Ext^1(\projfun X,\projfun Y).
\]
Therefore, 
\begin{eqnarray*}
&& \langle [X], [Y] \rangle  - \langle  [X], [\eta Y] \rangle 
- \langle [Y], [\eta X] \rangle  + \langle[ \eta X], [\eta Y] \rangle \\
&=&\dim \Hom(\projfun X,\projfun Y) + \dim \Hom(\projfun Y,\projfun X) 
- \dim \Ext^1(\projfun X,\projfun Y).
\end{eqnarray*}
Note that $\Ext^1(\projfun X, \projfun Y)=\ext^1(X,Y)$ (\Cref{Thm:fullexact}),  and 
\begin{equation}\label{eq:bilhom}
\langle [X], [\eta Y] \rangle=\dim \Hom(X, \eta Y)\end{equation}
since $ \Ext^1(X, \eta Y)=0$ by \Cref{Lem:tiltingchar}.
Now the equality in the statement follows from \Cref{lem:dimlemma}.
\end{proof}

As a consequence, we have the following symmetric form describing the difference of 
the extension groups.
\begin{corollary} Let $X, Y\in \Dfiltered(\setJ)$.
$$\dim \Ext^1(X,Y)+\dim \Ext^1(Y,X) - \dim \ext^1(X,Y) = \langle [\eta X], [\eta Y] \rangle - \langle  [X], [Y]\rangle.$$ 
\end{corollary}

\begin{proof}
This is a direct consequence of \Cref{Lem:Symbil} and \eqref{Eq:bilin0}.
\end{proof}

\begin{example}\label{Ex:newex}
Let $n=4$ and let $X\in \Dfiltered$ be the following module, 
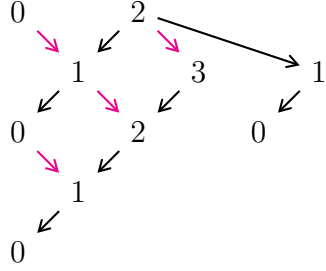
\[
\begin{tikzpicture}[scale=0.8,black arr/.style={magenta, -angle 60}, 
blue arr/.style={blue, -angle 60}, magenta arr/.style={black, -angle 60}, thick]

\path(3, 5) node (a1) {0};
\path(5, 5) node (a2) {2};
\path(8, 4) node (b1) {1};
\path(4, 4) node (b2) {1};
\path(6, 4) node (b3) {3};
\path(7, 3) node (c1) {0};
\path(3, 3) node (c2) {0};
\path(5, 3) node (c3) {2};
\path(4, 2) node (d) {1};
\path(3, 1) node (e) {0};

\draw[black arr] (a1) edge (b2);
\draw[magenta arr] (a2) edge (b1);
\draw[magenta arr] (a2) edge (b2);
\draw[black arr] (a2) edge (b3);
\draw[magenta arr] (b1) edge (c1);
\draw[magenta arr] (b2) edge (c2);
\draw[black arr] (b2) edge (c3);
\draw[magenta arr] (b3) edge (c3);
\draw[black arr] (c2) edge (d);
\draw[magenta arr] (c3) edge (d);
\draw[magenta arr] (d) edge (e);
\end{tikzpicture}
 \]
We have 
$\dvector_\Delta X=(1,1,1,1)$,  
$\eta X=\charT_1\oplus \charT_3$,
and by direct computation, \[\dim \End  X=5.\]
Therefore, following  \Cref{Lem:Symbil}, we have 
\[
\ext^1(X, X)=0, 
\]
which is not easy to verify directly. 
Therefore $X$ is rigid in $(\Dfiltered, \mathcal{E})$. 
Note however that $X$ is not rigid in $\mod D$, 
by an easy computation following \eqref{Eq:bilin0}.
\end{example}

\begin{remark}
The extension group $\Ext^1(M_I, M_J)$
can be used to interpret when $I$ and $J$ are strongly separated and thus detect when 
the product of two quantum minors is invariant under the bar involution (see \cite{LZ}).
We emphasise that this uses the exact structure from $\mod \aus$, 
which can not easily be understood using the existing exact structure on $\sub Q_{\setJ}$ 
(or $\mathcal{E}$ on $\Dfiltered(\setJ)$). See  \Cref{Sec:strongExt} 
for more details.
\end{remark}

\section{The AR-quiver of $\Dfiltered(\setJ)$}\label{sec6}
We compute AR-sequences in $\Dfiltered(\setJ)$ by lifting their counterparts
from $\sub Q_\setJ$. Note that this is possible by \Cref{Lem:liftedprecise}.
We also give some explicit examples of AR-quivers. 

\subsection{Computation of AR-sequences. }
\begin{theorem}\label{thm:liftARseq}
Let 
\[
\begin{tikzpicture}[scale=0.6,
 arr/.style={black, -angle 60}] 
 \path (5,12) node(a1) {$N$};
 \path (8,12) node(a2) {$X$};
\path (11,12) node(a3) {$M$};
 \path (2,12) node(a0) {$0$}; 
 \path (14,12) node(a4) {$0$};
 \path(1.2, 12) node (a) {$\xi$:};
  
 \path[arr] (a0) edge  (a1);
 \path[arr] (a1) edge node [auto]{$f$} (a2);
 \path[arr] (a2) edge node [auto]{$g$} (a3);
 \path[arr] (a3) edge  (a4);
 \end{tikzpicture}
\]
be an AR-sequence in $\sub Q_{\setJ}$. Then the lift of $\xi$,
\[
\begin{tikzpicture}[scale=0.6,
arr/.style={black, -angle 60}]
\path (5,12) node(a11) {$\N$};
\path (8,12) node(a22) {$\X$};
\path (11,12) node(a33) {$\M$};
\path (2, 12) node(a0) {$0$}; 
\path (14, 12) node(a44) {$0$};
\path(1.2, 12) node (a) {$\hat{\xi}$:};
  
\path[arr] (a0) edge  (a11);
\path[arr] (a11) edge node [auto]{$\f$} (a22);
\path[arr] (a22) edge node [auto]{$\g$} (a33);
\path[arr] (a33) edge  (a44);
\end{tikzpicture}
\]
to $\Dfiltered(\setJ)$ with $\N$ and $\M$ indecomposable is an AR-sequence.
\end{theorem}
\begin{proof}
We show that $\g$ is almost split. The proof for $\f$ can be done similarly. 
Let $Y\in \Dfiltered(\setJ)$ be indecomposable and $h\colon Y\rightarrow \M$ not an isomorphism. Then 
\[\overline{h}= \projfun(h)\colon \projfun Y\rightarrow M\] 
is not a split epimorphism  and so
there exists 
$l\colon \pi Y\rightarrow X$ such that 
\[
\overline{h}= gl
\]
By \Cref{Lem:liftofmaps}, $\overline{h}$ can be lifted to a map 
\[
\tilde{h}=\g \hat{l}\colon Y\rightarrow \M.
\]
By the equivalence in \Cref{Thm:equiv}, $h- \tilde{h}$ factors through $\add \charT_{\setJ}$ and so
\[
\im (h-\tilde{h})\subseteq \eta \M. 
\]
By \Cref{Lem:exact1}, the sequence $\eta(\hat{\xi})$ splits.  Let 
$\mu\colon\eta \M \rightarrow \eta\X$ be the splitting of $\eta(\hat{\xi})$. Then $\g \mu$ is the identity map on $\eta \M$. Let 
\[
l'=\mu(h-\g\hat{l})+ \hat{l}\colon Y\rightarrow \X
\]
Then
\[
\g l'=\g(\mu(h-\g\hat{l})+ \hat{l}) = \g\mu(h-\g\hat{l})+ \g\hat{l} =h. 
\]
That is,  $h$ factors through $\g$ and so $\g$ almost splits. 
\end{proof}

The AR-quivers of $\sub Q_{\setJ}$ and $\Dfiltered(\setJ)$ are essentially identical.
The difference is that  the latter has 
 the extra projective-injective indecomposable objects 
$\charT_j$ for $j\in \setJ$.
When $\setJ=\{1\}$ or $\{n-1\}$, any object in $\Dfiltered(\setJ)$ is
 projective-injective, 
and so there are no AR-sequences. In all other cases, the following proposition 
identifies where the
projective injective objects $\charT_j$ are placed in the AR-quiver of $\Dfiltered(\setJ)$. 
Note that the relative cosyzygy functor 
$\Sigma$ on $\sub Q_\setJ$  defines the AR-translation \cite{GLS08}.

\begin{proposition}
Let $j\in \setJ$. Then $\charT_j$ appears in the 
middle term of  the AR-sequence
\[
\begin{tikzpicture}[scale=0.6,black
 arr/.style={black, -angle 60}]
 \path (2, 4) node (1){$0$};
\path (5, 4) node (2){$\widehat{\Sigma S_j}$};
\path (9, 4) node (3){$\hat{X}\oplus \charT_j$};
\path (13, 4) node (4){$\hat{S_j}$};
\path (16, 4) node (5){$0$,};

\draw[->] (1) --  (2);
\draw[->]  (2) -- (3);
\draw[->]  (3) to (4);
\draw[->]  (4) to (5);
\end{tikzpicture}
\]
where 
$X$ is the middle term of the AR-sequence in $\sub Q_\setJ$ ending with the 
simple $\Pi$-module $S_j$ at vertex $j$, and 
the three lifts $\widehat{\Sigma S_j}$, $\hat{X}$ and $\hat{S_j}$ are all minimal.
\end{proposition}
\begin{proof}
Consider the AR-sequence ending at $S_j$,
\[
\begin{tikzpicture}[scale=0.6,black
 arr/.style={black, -angle 60}]
 \path (1.2, 4) node (1){$\xi$:};
 \path (2, 4) node (1){$0$};
\path (5, 4) node (2){$\Sigma S_j$};
\path (8, 4) node (3){$X$};
\path (11, 4) node (4){$S_j$};
\path (14, 4) node (5){$0$.};

\draw[->] (1) --  (2);
\draw[->]  (2) -- (3);
\draw[->]  (3) to (4);
\draw[->]  (4) to (5);
\end{tikzpicture}
\]
Then $\dim \soc(\Sigma S_j)=\dim \soc X$. Indeed,  otherwise 
there would be a split short exact sequence of socles 
\[
\begin{tikzpicture}[scale=0.6,black
 arr/.style={black, -angle 60}]
 \path (2, 4) node (1){$0$};
\path (5, 4) node (2){$\soc \Sigma S_j$};
\path (9, 4) node (3){$\soc X$};
\path (13, 4) node (4){$\soc S_j$};
\path (16, 4) node (5){$0$,};

\draw[->] (1) --  (2);
\draw[->]  (2) -- (3);
\draw[->]  (3) to (4);
\draw[->]  (4) to (5);
\end{tikzpicture}
\] 
which together with the fact that $\soc S_j=S_j$ would imply that 
$\xi$ splits, contradicting that $\xi$ is an AR-sequence. 
So by \Cref{Lem:liftedprecise}, the middle 
term of the lift of $\xi$ has $\charT_j$ as a direct summand as claimed.
\end{proof}

\subsection{Examples of AR-quivers of $\sub Q_{\setJ}$ and $\Dfiltered(\setJ)$}
\label{Ex:ARquivers}
Let $n=4$.
In both examples below, the blue and red modules are the indecomposable  projective 
and injective objects in the categories, with those
in blue the indecomposable summands of $\charT_{\setJ}$. 

(1) The AR-quivers of $\sub Q_{2}$ and  $\Dfiltered(\{2\})_{\mathcal{E}}$ are as follows, respectively.
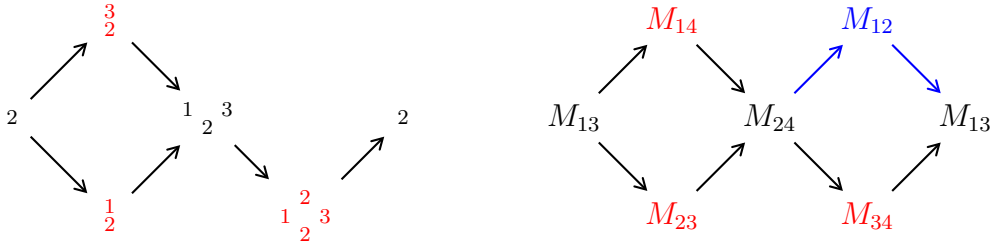
\[
\begin{tikzpicture}[scale=0.7, black arr/.style={black, -angle 60}, 
blue arr/.style={blue, -angle 60}, thick]

\path (1, 1) node(b1) {$\begin{smallmatrix}2\end{smallmatrix}$};
\path (3, 3) node(b2) {\color{red}$\begin{smallmatrix}3\\2\end{smallmatrix}$};
\path (3, -1) node(c1) {\color{red}$\begin{smallmatrix}1\\2\end{smallmatrix}$};
\path (5, 1) node(c2) {$\begin{smallmatrix}1 &&3\\&2&\end{smallmatrix}$};
\path (9, 1) node(d2) {$\begin{smallmatrix}2 \end{smallmatrix}$};
\path (7, -1) node(d1) {\color{red}$\begin{smallmatrix}&2&\\1 &&3\\&2&\end{smallmatrix}$};

\path[black arr] (c1) edge  (c2);
\path[black arr] (b1) edge  (c1);
\path[black arr] (b2) edge  (c2);
\path[black arr] (b1) edge  (b2);
\path[black arr] (c2) edge  (d1);
\path[black arr] (d1) edge (d2);


\path (12.5, 1) node(e1) {$M_{13}$};
\path (14.5, 3) node(e2) {\color{red}$M_{14}$};
\path (14.5, -1) node(f1) {\color{red}$M_{23}$};
\path (16.5, 1) node(f2) {$M_{24}$};
\path (18.5, 3) node(f3) {\color{blue}$M_{12}$};
\path (18.5, -1) node (g1) {\color{red}$M_{34}$};
\path (20.5, 1) node(g2) {$M_{13}$};

\draw[blue arr] (f2) -- (f3);
\draw[blue arr] (f3) -- (g2);

\path[black arr] (g1) edge (g2);
\path[black arr] (e1) edge  (e2);
\path[black arr] (f1) edge  (f2);
\path[black arr] (e1) edge  (f1);
\path[black arr] (e2) edge  (f2);
\path[black arr] (f2) edge  (g1);
\end{tikzpicture}
\]

(2) The AR-quivers of $\sub Q_{\{1,3\}}$ and  $\Dfiltered(\{1, 3\})_{\mathcal{E}}$ are as follows, respectively.
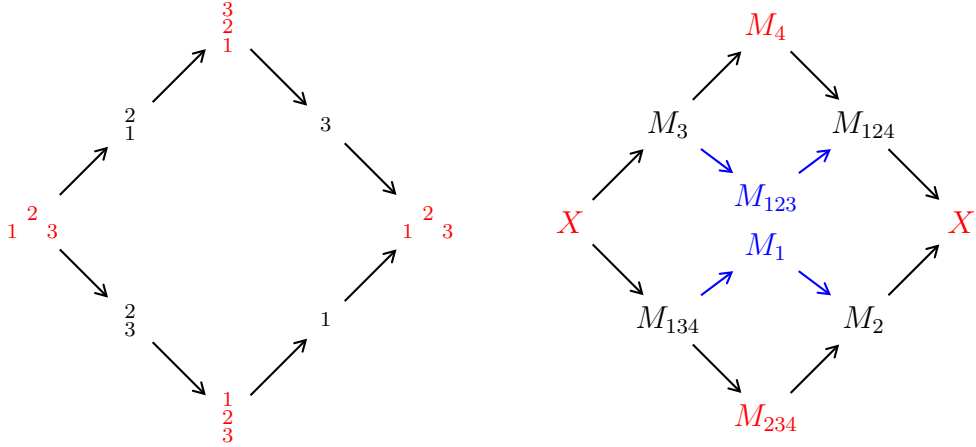
\[
\begin{tikzpicture}[scale=0.65, black arr/.style={black, -angle 60}, 
blue arr/.style={blue, -angle 60}, thick]

\path (3, 1) node(a1) {\color{red}$\begin{smallmatrix}&2&\\1&&3\end{smallmatrix}$};
\path (5, 3) node(a2) {\color{black}$\begin{smallmatrix}2\\1\end{smallmatrix}$};
\path (7, 5) node(a3) {\color{red}$\begin{smallmatrix}3\\2\\1\end{smallmatrix}$};
\path (9, 3) node(b2) {$\begin{smallmatrix}3\end{smallmatrix}$};
\path (5, -1) node(b0) {$\begin{smallmatrix}2\\3\end{smallmatrix}$};
\path (7, -3) node(c0) {\color{red}$\begin{smallmatrix}1\\2\\3\end{smallmatrix}$};
\path (9, -1) node(c1) {$\begin{smallmatrix}1\end{smallmatrix}$};
\path (11, 1) node(c2) {\color{red}$\begin{smallmatrix}&&2&\\&1&&3\\\end{smallmatrix}$};

\draw[black arr] (a1) edge (b0);
\draw[black arr] (b0) edge (c0);
\draw[black arr] (c0) edge (c1);
\draw[black arr] (a1) -- (a2);
\draw[black arr] (a2) -- (a3);
\draw[black arr] (a3) edge  (b2);
\draw[black arr] (b2) edge (c2);
\draw[black arr] (c1) edge (c2);


\path (14, 1) node(d1) {\color{red}$X$};
\path (16, 3) node(d2) {\color{black}$M_{3}$};
\path (18, 5) node(d3) {\color{red}$M_{4}$};
\path (20, 3) node(e2) {$M_{124}$};
\path (18, 1.5) node(e1) {\color{blue}$M_{ 123}$};
\path (18, 0.5) node(e3) {\color{blue}$M_{1}$};
\path (16, -1) node(e0) {$M_{134}$};
\path (18, -3) node(f0) {\color{red}$M_{234}$};
\path (20, -1) node(f1) {$M_{2}$};
\path (22, 1) node(f2) {\color{red}$X$};

\draw[black arr] (d1) edge (e0);
\draw[black arr] (e0) edge (f0);
\draw[black arr] (f0) edge (f1);
\draw[black arr] (d1) -- (d2);
\draw[black arr] (d2) -- (d3);
\draw[black arr] (d3) edge  (e2);
\draw[black arr] (e2) edge (f2);
\draw[black arr] (f1) edge (f2);

\draw[blue arr] (d2) -- (e1);
\draw[blue arr] (e1) -- (e2);
\draw[blue arr] (e0) -- (e3);
\draw[blue arr] (e3) -- (f1);
\end{tikzpicture}
\]
We have $M_{1}=\charT_1$,
$M_{12}=\charT_2$, $M_{123}=\charT_3$, 
and the module $X$ is the same as the one in \Cref{Ex:newex}.

\section{Lifting the cluster structure on $\sub Q_{\setJ}$ to $\Dfiltered(\setJ)$} 
We show that the cluster structure on $\sub Q_{\setJ}$  can be lifted to $\Dfiltered(\setJ)$. In particular, 
the functor $\pi$ induces bijections on cluster tilting objects and mutation sequences.

\subsection{Cluster structure}\label{sec:clstructnew}
Let $\mathcal{C}$ be a Hom-finite, Frobenius stably 2-CY category, for instance $\sub Q_\setJ$
or $\Dfiltered(\setJ)$.
Let $T=\oplus_{i=1}^{\cn}T_i\in \mathcal{C}$ be a cluster tilting object, where each $T_i$ is indecomposable.  
For any summands $T_i$, $ T_j$, let
\begin{equation}\label{eq:Irr}
\Irr(T_i, T_j)=\rad(T_i, T_j)/\rad^2(T_i, T_j),
\end{equation}
which denotes the space of irreducible maps
$T_i\to T_j$ in $\add T$. 
Let $Q_T$ be the quiver with vertices $[\cn]=\{1, \cdots, \cn\}$ and the number of arrows 
from $i$ to $j$ equal to  $\dim \Irr(T_i, T_j)$, for any $i, j\in [\cn]$. 
Note that $Q_T$ is the opposite quiver 
of the Gabriel quiver of \[A=(\End_{\mathcal C}T)^{\text{op}}.\]
\begin{definition}\cite{BIRS, Iy}
An object $T\in \mathcal{C}$ is called a  {\em cluster tilting object} if  
$\Ext^1_{\mathcal{C}}(T,M)=0$ implies  $M\in \add T$. 
\end{definition}
Unless otherwise stated, all cluster tilting objects are assumed to be \emph{basic}, 
that is, if $T=\oplus_i T_i$ is a direct sum of indecomposable summands, 
then the summands $T_i$ are pairwise non-isomorphic.

Assume that $ \mathcal{C}$ admits a cluster structure \cite{BIRS} with a cluster tilting object $T$. 
A summand $T_i$ is said to be {\em mutable}
if it is not projective. The corresponding vertex in $Q_T$ is also said to be mutable. 
Non-mutable summands and vertices are referred to as  {\em frozen summands} 
and {\em frozen vertices}.
The assumption on $\mathcal{C}$ and $T$ has a number of consequences.
First, the quiver $Q_T$  has no loops and no 2-cycles incident 
at mutable vertices. Second, associated to any mutable summand 
$T_k$, there are exact {\em mutation sequences}, 
\begin{equation}
\begin{tikzpicture}[scale=0.6,
 arr/.style={black, -angle 60}, baseline=(bb.base)]
\coordinate (bb) at (0, 0.8);

\path (9,1) node (c3) {$T_k$};
\path (6,1) node (c2) {$E_k$};
\path (3,1) node (c1) {$T_k^*$};

\path[arr] (c1) edge node[auto]{} (c2);
\path[arr] (c2) edge node[auto]{$f$} (c3);
\end{tikzpicture}
\end{equation}
and 
\begin{equation}
\begin{tikzpicture}[scale=0.6,
 arr/.style={black, -angle 60}, baseline=(bb.base)]
\coordinate (bb) at (0, 0.8); 
 
\path (9,1) node (c3) {$T^*_k$};
\path (6,1) node (c2) {$F_k$};
\path (3,1) node (c1) {$T_k$};

\path[arr] (c1) edge node[auto]{$g$} (c2);
\path[arr] (c2) edge node[auto]{}  (c3);
\end{tikzpicture}
\end{equation}
where $f$ and $g$ are minimal $\add (T \minus\{T_k\})$-approximations and $T_k^*$ 
is indecomposable.
Third, 
$$T'=\mu_k(T)=(T\minus\{T_k\}) \oplus T_k^*$$ 
is again a cluster tilting object, which is called the 
{\em mutation of $T$ in direction $k$}.
Fourth,  $Q_{T'}$ 
is obtained from $Q_T$ by Fomin--Zelevinsky mutation \cite{FZ1} at vertex $k$. 
In fact, these four properties characterise the existence of a cluster structure
on $\mathcal{C}$.

Two cluster tilting objects $T$ and $T'$ are 
 {\em mutation equivalent} if they can be obtained from each other by 
a sequence of mutations. In this case, we also say that $T$ is {\em reachable} 
from $T'$. 

Assume that the first $\cm$ summands of $T$ are the mutable summands of $T$.
Let $\mat{B}=\mat{B}_T=(b_{ij})$ be the $\cn\times \cm$-matrix defined by
\begin{equation}
\label{eq:BfromT}
  b_{ij}={\dim \Irr(T_i, T_j)-\dim \Irr(T_j, T_i),}
\end{equation}
where $1\leq i\leq \cn$ and $1\leq j\leq \cm$. That is,  $\mat{B}$ is the \emph{exchange 
matrix} defined by $Q_T$. The submatrix of $\mat{B}$ consisting of the first 
$\cm$ rows is called the  {{\em principal part} of $\mat{B}$.}
As $Q_T$ has no 2-cycles incident at a mutable vertex, 
at least one of $ \Irr(T_j, T_i) $ and $ \Irr(T_i, T_j)$ is zero. 
So in the mutation sequences 
\eqref{Eq:exseq1} and \eqref{Eq:exseq2},  
\begin{equation*}
F_k = \bigoplus_{j:-b_{jk}>0} T_j^{-b_{jk}}
\quad\text{and}\quad
E_k = \bigoplus_{j:b_{jk}>0} T_j^{b_{jk}},
\end{equation*} 
and the matrix $\mat B$ is also called the \emph{mutation matrix defined}
by $T$.

\begin{remark}  \label{sec:clstruct1}
A module $T$ in $\sub Q_\setJ$ is a cluster tilting object 
if and only if $T$ is maximal rigid (see for instance \cite{GLS06} and \cite[Thm 2.9]{GLS11})
and  $\sub Q_{\setJ}$  admits a  cluster structure  
consisting of the cluster tilting objects \cite[\S 8]{GLS08}.
\end{remark}

\subsection{Cluster tilting objects and mutation sequences in 
$(\Dfiltered(\setJ), \mathcal{E})$} \label{subsect:liftingcto} \label{sec:7.2}

Recall that $\Dfiltered(\setJ)$ is a Frobenius stably 2-CY category under the exact structure $\mathcal{E}$
(see \Cref{Thm:Frob2CY}).
\begin{proposition} \label{Lem:liftcto} 
The exact functor $\projfun$ induces a bijection between the cluster tilting objects
in $\Dfiltered(\setJ)$ and those in $\sub Q_{\setJ}$. 
More precisely, for 
$T=\charT_\setJ\oplus X\in \Dfiltered(\setJ)$ with $X$ having no summands in $\add \charT_{\setJ}$, the object $T$
is a cluster tilting object 
if and only if $\projfun X$ is a cluster tilting object in $\sub Q_\setJ$.
\end{proposition}
\begin{proof} 
The proposition follows from the properties of the full exact functor $\projfun$ 
in \Cref{Thm:fullexact} and the uniqueness of lifts from $\sub Q_\setJ$
to $\Dfiltered(\setJ)$ (up to summands from $\add \charT_\setJ$) proven in \Cref{Lem:dense}.
\end{proof}

\begin{corollary}\label{Lem:liftcto1} 
$T\in \Dfiltered(\setJ)$ is a cluster tilting object if and only 
if it is maximal rigid. 
\end{corollary}
\begin{proof}
This follows from 
\Cref{Lem:liftcto} and the characterisation of cluster tilting objects in $\sub Q_\setJ$
as maximal rigid objects (see  \Cref{sec:clstruct1}).
\end{proof}

\begin{lemma} \label{Lem:liftapprox}
Let $M,N\in \Dfiltered(\setJ)$.
\begin{itemize}
\item[(1)] If $f\colon M'\rightarrow N$ is a surjective left $\add (M\oplus \charT_\setJ)$-approximation, then $\ker f\in \Dfiltered(\setJ)$ and the following extension is admissible, 
\[
\begin{tikzpicture}[scale=0.6, arr/.style={black, -angle 60}]
\path (3,12) node(a1) {$\ker f$};
\path (6.2,12) node(a2) {$M'$};
\path (9.2,12) node(a3) {$N$};
\path (0,12) node(a0) {$0$}; 
\path (-0.8,12) node(a) {$\varepsilon$:}; 
\path (12.2,12) node(a4) {$0$.};

\path[arr] (a0) edge (a1);
\path[arr] (a1) edge (a2);
\path[arr] (a2) edge node [auto]{$f$} (a3);
\path[arr] (a3) edge (a4);
\end{tikzpicture}
\]
\item[(2)] If $g\colon X\rightarrow \projfun N$ is a surjective left 
$\add \projfun M$-approximation, then $g$ has a lift $\hat{g}$ in 
$\Dfiltered(\setJ)$ such that $\hat{g}$ is a surjective left 
$\add (M\oplus \charT_{\setJ})$-approximation of $N$.
\end{itemize}
\end{lemma}
\begin{proof}
(1) First note that $f$ is also an $\add(M\oplus \charT)$-approximation, since 
$N\in \Dfiltered(\setJ)$ implies that $\eta N=\eta_{\charT_\setJ} N\in \add \charT_\setJ$. 
Therefore, $\Hom(\charT, M')\rightarrow \Hom(\charT, N)$ is surjective. 
So the extension $\varepsilon$ is admissible in $\Dfiltered$, by \Cref{Lem:exact1} (4),
and further by  \Cref{Lem:exact1} (1), $\eta (\varepsilon)$ is 
split exact. Therefore, $\ker f\in \Dfiltered(\setJ)$.

(2) Consider the short exact sequence with $K=\ker g$, 
\[
\begin{tikzpicture}[scale=0.6,
 arr/.style={black, -angle 60}]
 \path (3,12) node(a1) {$K$};
 \path (6,12) node(a2) {$X$};
\path (9,12) node(a3) {$\projfun N$};
 \path (0,12) node(a0) {$0$}; 
  \path (-0.8,12) node(a) {$\epsilon$:}; 
 \path (12,12) node(a4) {$0$.};
  
 \path[arr] (a0) edge  (a1);
 \path[arr] (a1) edge  (a2);
 \path[arr] (a2) edge [auto] node{$g$}  (a3);
 \path[arr] (a3) edge  (a4);
 \end{tikzpicture}
\]
By assumption, $X\in \sub Q_{\setJ}$. So $K\in \sub Q_\setJ$ 
and $\epsilon$ is a short exact sequence in $\sub Q_\setJ$.
By \Cref{Lem:liftedprecise},  the sequence 
 $\epsilon$ lifts to an admissible sequence in $\Dfiltered$, 
\[
\begin{tikzpicture}[scale=0.6,
 arr/.style={black, -angle 60}]
 \path (3,12) node(a1) {$\hat{K}$};
 \path (6,12) node(a2) {$\hat{X}$};
\path (9,12) node(a3) {$N$};
 \path (0,12) node(a0) {$0$}; 
 \path (12,12) node(a4) {$0$,};

\path (-0.8,12) node(a) {$\hat{\epsilon}$:}; 

 \path[arr] (a0) edge  (a1);
 \path[arr] (a1) edge  (a2);
 \path[arr] (a2) edge[auto] node {$\hat{g}$} (a3);
 \path[arr] (a3) edge  (a4);
 \end{tikzpicture}
\]
and thus further by \Cref{Lem:exact1}, $\eta(\hat{g})$ splits. We have, 
\begin{equation}\label{eq:neweta}
\eta(\hat{X})=\eta(\hat{K})\oplus \eta(N).
\end{equation}
By construction, any map $h\colon M'\to N$ with $M' \in \add M$ factors through $\g$, 
up to some summand $Y_h$ in $\add \charT$, that is, 
\[\xymatrix{
h\colon M' \ar[r] & \hat{X}\oplus Y_h\ar[r]^{~~~~~~~l} &N. }
\]
 Now, by \Cref{Lem:factormaps},
any map from an object in $\add \charT$ to $N$ factors through $\eta(N)$. Therefore, further by 
\eqref{eq:neweta}, 
any such a map, 
including $l|_{Y_h}$,  factors through $\g$. So $\g$
is an approximation as claimed.
\end{proof}

\begin{proposition} \label{Lem:liftexchange}  
For a mutable summand $T_k$ in a cluster tilting object $T\in \Dfiltered(\setJ)$, 
there are two admissible mutation sequences in $\Dfiltered(\setJ)$,
\begin{equation}\label{Eq:exseq1}
\begin{tikzpicture}[scale=0.6,
 arr/.style={black, -angle 60}, baseline=(bb.base)]
\coordinate (bb) at (0, 11.8); 
 \path (3,12) node(a1) {$T_k^*$};
 \path (6,12) node(a2) {$E_k$};
\path (9,12) node(a3) {$T_k$};
 \path (0,12) node(a0) {$0$}; 
 \path (12,12) node(a4) {$0$};
  
 \path[arr] (a0) edge  (a1);
 \path[arr] (a1) edge  (a2);
 \path[arr] (a2) edge  (a3);
 \path[arr] (a3) edge  (a4);
\end{tikzpicture}
\end{equation}
and 
\begin{equation}\label{Eq:exseq2}
\begin{tikzpicture}[scale=0.6,
 arr/.style={black, -angle 60}, baseline=(bb.base)]
\coordinate (bb) at (0, 11.8); 
 
 \path (3,12) node(a1) {$T_k$};
 \path (6,12) node(a2) {$F_k$};
\path (9,12) node(a3) {$T_k^*$};
 \path (0,12) node(a0) {$0$}; 
 \path (12,12) node(a4) {$0$};

 \path[arr] (a0) edge  (a1);
 \path[arr] (a1) edge  (a2);
 \path[arr] (a2) edge  (a3);
 \path[arr] (a3) edge  (a4);
 \end{tikzpicture}
\end{equation}
which are lifts of the mutation sequences 
for $\pi T_k, ~\pi T_k^*$ in $\sub Q_\setJ$. Moreover, 
for each $j\in \setJ$, at most one of $E_k$ and $F_k$ 
has $\charT_j$ as a summand.
\end{proposition}
\begin{proof} 
By \Cref{Lem:liftedprecise}, we can lift mutation sequences from $\sub Q_\setJ$ to exact sequences in $\Dfiltered$. That the lifted maps are approximations follows from 
\Cref{Lem:liftapprox}. We therefore have the mutation sequences as claimed. 

Write 
\[E_k= \widehat{\projfun E_k}\oplus \bigoplus_j\charT_j^{c_j} ~\text{ and  }~
F_k= \widehat{\projfun F_k}\oplus \bigoplus_j\charT_j^{d_j},\]
where $\widehat{\projfun E_k}$ and $\widehat{\projfun F_k} $ are minimal lifts of 
$\projfun E_k$ and $\projfun F_k$.
By \Cref{Lem:liftedprecise}, 
\[
c_j= \dim \Hom(S_j, \projfun T_k) + \dim \Hom(S_j, \projfun T_k^*) -
\dim \Hom(S_j, \projfun E_k)
\]
and 
$$
d_j=\dim \Hom(S_j, \projfun T_k) + \dim \Hom(S_j, \projfun T_k^*)
-\dim \Hom(S_j, \projfun F_k).
$$
Now by  \cite[Prop.]{GLS11} (see also \cite[Lem. 6.14]{DI}), 
\begin{eqnarray*}
&&\mathrm{min}\{\dim \Hom(S_j, \projfun E_k), \dim \Hom(S_j, \projfun F_k)\}\\
&<&\dim \Hom(S_j, \projfun T_k) + \dim \Hom(S_j, \projfun T_k^*)\\
&=&\mathrm{max}\{\dim \Hom(S_j, \projfun E_k), \dim \Hom(S_j, \projfun F_k)\}
\end{eqnarray*}
Therefore, either $c_j=0$ or $d_j=0$. That is, 
at most one of  $E_k$ and $F_k$ has $\charT_j$ as a summand.
\end{proof}

\begin{remark}\label{rem:reach}
Following \Cref{Lem:liftcto1} and \Cref{Lem:liftexchange}, 
a cluster tilting object $T'$ is reachable from $T$ if and only if $\projfun T'$ 
is reachable from $\projfun T$.
\end{remark}

\subsection{Cluster structure on $(\Dfiltered, \mathcal{E} )$} \label{Subsection:matrixB}
\label{sec:7.3}

\begin{proposition}\label{Prop:2-cycles}
Let $T\in \Dfiltered(\setJ)$ be a cluster tilting object. Then 
the quiver $Q_T$ has no loops at any vertex and no 2-cycles incident to any mutable vertex. 
\end{proposition}
\begin{proof}
Note that 
the Gabriel quiver  $Q_{\pi T}$ of 
$(\End  \projfun T)^{\text{op}}$ does not have loops
nor 2-cycles (see \cite[Sect. 8.1]{GLS08}). 
Together with \Cref{Lem:noloops}, this implies that there are no loops in $Q_T$.

We show that there are no $2$-cycles incident to a mutable vertex in $Q_T$.
By the Dlab--Ringel equivalence \eqref{Thm:equiv},
$\End  \pi T$ is the quotient of $\End  T$ by the maps 
that factor through $\add \charT_\setJ$. Then, since there are no $2$-cycles in 
$Q_ {\pi T}$, there are no $2$-cycles in $Q_T$ between mutable vertices, or between
a mutable vertex and a vertex corresponding to a projective summand $P_j$.  
It remains to consider summands $\charT_j$ of $\charT_{\setJ}$. By \Cref{Lem:liftexchange}, $\charT_j$ occurs in at most one of $E_k$ and $F_k$.
 There are therefore
no $2$-cycles between mutable vertices and those corresponding to 
$\charT_j \in \add \charT_\setJ$. 
\end{proof}

\begin{theorem} \label{thm:clusterstructure} 
The cluster tilting objects in $\Dfiltered(\setJ)$ form a cluster structure. 
\end{theorem}
\begin{proof}
By  \cite[Thm. II.1.6]{BIRS}, if the cluster tilting objects in a 2-CY category have no loops or 
2-cycles incident at mutable vertices, then they determine a cluster structure on the category. 
So the theorem follows from  \Cref{Prop:2-cycles}. 
\end{proof}

\begin{example} \label{Ex:GabrielQuiversEndT}
 We continue from the example in \Cref{Ex:ARquivers} (1). Let
\[P=M_{12}\oplus M_{23}\oplus  M_{34} \oplus M_{14} \text{ and } T= M_{13}\oplus P.\]
Then $T$ is a cluster tilting object with one mutable summand $M_{13}$. 
Denote the summands of $T$ in turn by $T_1, ~T_2, ~T_3, ~T_4$ and $T_5$, with $T_1=M_{13}$. 

\noindent (1) The algebra $(\End  P)^{\mathrm{op}}$ is defined by the double cyclic quiver below  with relations:
 \begin{itemize}
 \item[(i)] starting from each vertex,  $xy=yx$ and $x^2=y^2$.
 Note that these are the defining relations for the algebra $C$ in \cite{JKS1} (see also \S\ref{sec:Grass}) when $k=2$ and $n=4$,
 where $\CM C$ is applied to give a categorification the cluster algebra $\CC[\gr(2, 4)]$.

\item[(ii)] $x_3x_2=0$ and $y_4y_5$=0.
\end{itemize}
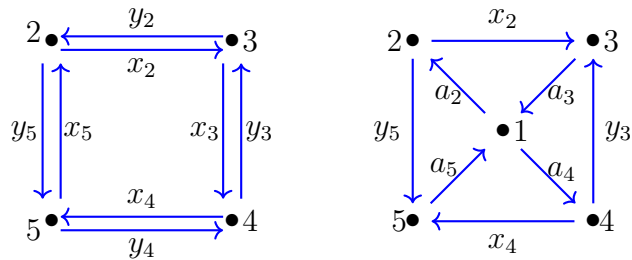
\[\begin{tikzpicture}[scale=0.6,
 arr/.style={thick, -angle 60}]
 
\path (4, 2) node (1) {$\bullet$};
\path (3.6, 1.8) node () {$5$};
\path (8, 2) node (2) {$\bullet$};
\path (8.4, 2) node (2) {$4$};
 
 \path (4, 6) node (3) {$\bullet$};
\path (8, 6) node (4) {$\bullet$};
 \path (3.6, 6.2) node (3) {$2$};
\path (8.4, 6) node (4) {$3$};
 
\draw[blue, thick, ->] (4.2, 5.8)  to  (7.8, 5.8);
\draw[blue, thick, ->]  (7.8, 6.1) to   (4.2, 6.1);
\path(6, 6.5) node  (){$y_2$};
\path(6, 5.4) node  (){$x_2$};

\draw[blue, thick, ->] (4.2, 1.8)  to (7.8, 1.8);
\draw[blue, thick, ->] (7.8, 2.1) to  (4.2, 2.1);
\path(6, 2.5) node(){$x_4$};
\path(6, 1.4) node(){$y_4$};

\draw[blue, thick,->]   (3.8, 5.5) to (3.8, 2.5);
\draw[blue, thick,->] (4.2, 2.5) to  (4.2, 5.5);
\path(3.4, 4) node(){$y_5$};
\path(4.6, 4) node(){$x_5$};

\draw[blue, thick,->]  (7.8, 5.5) to (7.8, 2.5);
\draw[blue, thick,->] (8.2, 2.5) to  (8.2, 5.5);
\path(8.6, 4) node(){$y_3$};
\path(7.4, 4) node(){$x_3$};

\path (12, 2) node (1) {$\bullet$};
\path (16, 2) node (2) {$\bullet$};
\path (11.7, 2) node () {$5$};
\path (16.3, 2) node () {$4$};
  
\path (14, 4) node (5) {$\bullet$};
\path (14.4, 4) node () {$1$};

\path (12.8, 4.80) node () {$a_2$};
\path (15.3, 4.80) node () {$a_3$};
\path (12.7, 3.2) node () {$a_5$};
\path (15.3, 3.1) node () {$a_4$};
 
 \path (12, 6) node (3) {$\bullet$};
\path (16, 6) node (4) {$\bullet$};
 \path (11.6, 6) node () {$2$};
\path (16.4, 6) node () {$3$};
 
\draw[blue, thick,->] (3) -- (4);
\path(14, 6.5) node(){$x_2$};
\draw[blue, thick,->] (4) -- (5);
\draw[blue, thick,->] (5) -- (3);

\draw[blue, thick,->]  (2) --  (1);
\path (14, 1.5) node (){$x_4$};

\draw[blue, thick,->] (1) -- (5);
\draw[blue, thick,->]  (5) -- (2);

\draw[blue, thick,->] (3) -- node[left]{{\color{black}$y_5$}}(1);
\draw[blue, thick,->]  (2) --node[right]{{\color{black}$y_3$}} (4);
 \end{tikzpicture}
 \]
(2) The algebra $A=(\End  T)^{\mathrm{op}}$ is defined by the quiver  on the right with relations: 
 \begin{itemize}
 \item[(i)] for each $i$, the two paths of length 2 from $i$
to $1$  are the same.
  
\item[(ii)] any path containing two of the arrows (which may be equal) $a_4$, $x_2$, $y_5$ is zero.
Note that these three named arrows 
correspond to irreducible maps induced by $t$ (see \eqref{eq:t}) when restricted to vertex 0, while 
the other arrows correspond to injective irreducible maps.
\end{itemize}
The algebra $A$ 
has the same Gabriel quiver as its
 counterpart constructed from $\CM C$ \cite{JKS1}, and so 
 the exchange matrices defined by the two Gabriel quivers are the same. 
This is only natural, as we will see in  \Cref{sec9} that $\Dfiltered(4, \{2\})$ 
also gives a  categorification of the cluster structure on $\CC[\Gr(2, 4)]$.
\end{example}

\section{Projective resolutions and the Grothendieck group $K(\proj A)$}
Let $T=\oplus_i T_i$ be a cluster tilting object in $\Dfiltered(\setJ)$,
 $A=(\End T)^{\text{op}}$ and $Q_T$ be as in \Cref{sec:clstructnew}.  
Let 
\[\bar A=(\End \projfun T)^{\text{op}}. \]
We compare the Grothendieck groups of $A$ and $\bar A$.
\subsection{Projective resolutions of $A$-modules}\label{sec:8.1}
Let  $\funF=\Hom(T, -).$ 
Then $\funF$ induces an equivalence 
\begin{equation}\label{eq:addTprojA}
\funF\colon \add T\to \proj A.\end{equation}
Denote by $S_i$ the simple top of the projective 
$A$-module $\funF T_i$. 
\begin{lemma} \label{Lem:resHom}
Let $M\in \Dfiltered(\setJ)$. An $\add T$-approximation $T'\to M$ can be completed 
to an admissible exact sequence,
\begin{equation}\label{Eq:newname}
\begin{tikzpicture}[scale=0.3,
 arr/.style={black, -angle 60}, baseline=(bb.base)]
  \coordinate (bb) at (0, 11.8);
 \path (6,12) node(a1) {$T''$};
 \path (12,12) node(a2) {$T'$};
\path (18,12) node(a3) {$M$};
 \path (0,12) node(a0) {$0$}; 
 \path (24,12) node(a4) {$0$.};

 \path[arr] (a0) edge  (a1);
 \path[arr] (a1) edge  (a2);
 \path[arr] (a2) edge  (a3);
 \path[arr] (a3) edge  (a4);
 \end{tikzpicture}
\end{equation}
where $T''\in \add T$. Consequently, 
$\projd_A \funF M\leq 1.$
\end{lemma}
\begin{proof}
Since $T$ contains the indecomposable projective-injective objects, 
an $\add T$-approximation $T'\to M$ is surjective. 
Then by \Cref{Lem:liftapprox} (1), the short exact sequence \eqref{Eq:newname} is admissible and $T''\in \Dfiltered(\setJ)$. Applying $\stHom(T,-)$ to the triangle 
corresponding to \eqref{Eq:newname} in the  stable category of $\Dfiltered(\setJ)$
and taking the long exact sequence we see that \[\ext^1(T,T'')=0.\] 
As $T$ is a cluster tilting object,  $T''$ is in $ \add T$.  
So we obtain a short exact sequence, 
\begin{equation*}\label{Eq:8.1}
\begin{tikzpicture}[scale=0.6,
 arr/.style={black, -angle 60}]
 \path (3,12) node(a1) {$\funF T''$};
 \path (6.2,12) node(a2) {$\funF T'$};
\path (9.4,12) node(a3) {$\funF M$};
 \path (0,12) node(a0) {$0$}; 
 \path (12.4,12) node(a4) {$0$};

 \path[arr] (a0) edge  (a1);
 \path[arr] (a1) edge  (a2);
 \path[arr] (a2) edge  (a3);
 \path[arr] (a3) edge  (a4);
 \end{tikzpicture}
\end{equation*}
and so $\projd_A \funF M\leq 1.$
\end{proof}
\begin{lemma}\label{Lem:resmutablesimple} 
Suppose that $T_k$ is  a mutable summand  of $T$. By splicing together the
two mutation sequences \eqref{Eq:exseq1} and \eqref{Eq:exseq2} for $T_k$ 
and applying $\funF$, we obtain a projective resolution
\[
\begin{tikzpicture}[scale=0.3,
 arr/.style={black,  -angle 60}]
  \path (0, 2) node(a0) {$0$}; 
 \path (6, 2) node(a1) {$\funF T_k$};
 \path (12, 2) node(a2) {$\funF F_k$};
\path (18, 2) node(a3) {$\funF E_k$};
 \path (24, 2) node(a4) {$\funF T_k$};
  \path (30, 2) node(a5) {$S_k$};
  \path (36, 2) node(a6) {$0$.};

 \path[arr] (a0) edge  (a1);
 \path[arr] (a1) edge  (a2);
 \path[arr] (a2) edge  (a3);
 \path[arr] (a3) edge  (a4);
 \path[arr] (a4) edge  (a5);
\path[arr] (a5) edge  (a6);
 \end{tikzpicture}
 \]
of the simple $A$-module $S_k$.
Consequently, 
$\projd_A S_k\leq 3. $
\end{lemma}
\begin{proof}
Applying $\funF$ to the mutation sequences in \eqref{Eq:exseq1} and 
\eqref{Eq:exseq2} produces the long exact sequence, 
\[
\begin{tikzpicture}[scale=0.3,
 arr/.style={black,  -angle 60}]
  \path (0, 2) node(a0) {$0$}; 
 \path (6, 2) node(a1) {$\funF T_k$};
 \path (12, 2) node(a2) {$\funF F_k$};
\path (18, 2) node(a3) {$\funF E_k$};
 \path (24, 2) node(a4) {$\funF T_k$};
  \path (30, 2) node(a5) {$\cok \funF f$};
  \path (36, 2) node(a6) {$0$.};

 \path[arr] (a0) edge  (a1);
 \path[arr] (a1) edge  (a2);
 \path[arr] (a2) edge  (a3);
 \path[arr] (a3) edge node [auto]{$\funF f$}  (a4);
 \path[arr] (a4) edge (a5);
\path[arr] (a5) edge  (a6);
 \end{tikzpicture}
 \]
where 
$f\colon E_k\to T_k$ is an $\add (T\minus \{T_k\})$-approximation. By 
\Cref{Prop:2-cycles} the quiver of $Q_T$ has no loops, 
 so the only map $T_k\to T_k$ (up to a scalar) that does not factor through $f$ is the identity map. Therefore, the cokernel of $\funF f$ is isomorphic to
$S_k$, and so $S_k$ has the projective resolution as stated. 
\end{proof}

 Recall that the indecomposable projective objects in $\Dfiltered(\setJ)$ are 
$P_j$ for $1\leq j \leq n-1$, 
  and $\charT_j$ for $j\in \setJ$. Each $P_j$
is the minimal  lift of the projective object $Q_{n-j}\in \sub Q_\setJ$.
\begin{lemma} \label{Lem:frozensimple}
Let $T_i=P_{j}$ for some $j$.
Then the simple $A$-module $S_i$ has  a projective resolution as follows, 
\[
\begin{tikzpicture}[scale=0.3,
 arr/.style={black, -angle 60}]
 \path (0, 2) node(a0) {$0$}; 
 \path (6, 2) node(a1) {$\funF T''$};
 \path (12, 2) node(a2) {$\funF T'$};
\path (18, 2) node(a3) {$\funF T_i$};
 \path (24, 2) node(a4) {$S_i$};
 \path (30, 2) node(a5) {$0$.};

 \path[arr] (a0) edge  (a1);
 \path[arr] (a1) edge  (a2);
 \path[arr] (a2) edge  (a3);
 \path[arr] (a3) edge  (a4);
  \path[arr] (a4) edge  (a5);
 \end{tikzpicture}
 \]
Consequently, $\projd_A S_i \leq 2.$
\end{lemma}

\begin{proof}
Let $X$ be the minimal  lift of $\rad Q_{n-j}$
to $\Dfiltered(\setJ)$. Then 
$X\in \Dfiltered(\setJ)$ and there is an embedding $l\colon X\to T_i$ with  $T_i/l(X)$ the simple module 
associated to vertex $j$ (of the quiver defining $D$).  
So $X$ is a maximal submodule of $T_i$.
Consider the admissible $\add T$-presentation in \eqref{Eq:newname} for $X$,
\begin{equation}\label{Eq:8.2}
\begin{tikzpicture}[scale=0.3,
 arr/.style={black, -angle 60}, baseline=(bb.base)]
  \coordinate (bb) at (0, 1.8);
 \path (-1.5, 2) node (a) {$\epsilon$:};
  \path (0, 2) node(a0) {$0$}; 
 \path (6, 2) node(a1) {$T''$};
 \path (12, 2) node(a2) {$ T'$};
\path (18, 2) node(a3) {$X$};
 \path (24, 2) node(a4) {$0$,};
 
 \path[arr] (a0) edge  (a1);
 \path[arr] (a1) edge  (a2);
 \path[arr] (a2) edge node[above] {$g$} (a3);
 \path[arr] (a3) edge  (a4); 
\end{tikzpicture} 
\end{equation}
where $g$ is an $\add T$-approximation.
Splicing \eqref{Eq:8.2}
with the embedding $l\colon X \to T_i$ and 
applying $\funF$ give the sequence
\begin{equation}\label{eq:projresfrozenS}
\begin{tikzpicture}[scale=0.3,
 arr/.style={black, -angle 60}]
  \path (0, 2) node(a0) {$0$}; 
 \path (6, 2) node(a1) {$\funF T''$};
 \path (12, 2) node(a2) {$\funF T'$};
\path (18, 2) node(a3) {$\funF T_i$};
 \path (25, 2) node(a4) {$\cok \funF f$}; 
 \path (31, 2) node(a5) {$0$,};

 \path[arr] (a0) edge  (a1);
 \path[arr] (a1) edge  (a2);
 \path[arr] (a2) edge node[auto]{$\funF f$}  (a3);
 \path[arr] (a3) edge  (a4);
  \path[arr] (a4) edge  (a5);
 \end{tikzpicture}
\end{equation}
where $f$ is the composition of $g$ with the embedding $l$. 

Let $T_l\in \add T$ be indecomposable, and let 
$h\colon T_l\to T_i$ be a non-isomorphism. Then 
$\im h\subseteq l(X)$. Indeed, suppose that this is not the case. Then 
$\pi h\colon \pi T_l\to Q_{n-j}$ is surjective, since 
$Q_{n-j}$ is a projective $\Pi$-module and $\top Q_{n-j}\in \im \pi h$. As $Q_{n-j}$ is projective, 
$\pi T_l$ splits. Thus $T_l=T_i$ and $h$ is an isomorphism, contrary to our assumption. Therefore, $\im h\subseteq l(X)$, as claimed. 
Consequently, $h$ factors through $g$, since 
$g$ is an $\add T$-approximation.
As $T_i$ is indecomposable, $\End T_i$ is a local ring. Therefore, 
$\cok \funF f=S_i$, the simple quotient of $\Hom(T, T_i)$ generated by $\mathrm{Id}_{T_i}$,  and so \eqref{eq:projresfrozenS} is a projective resolution of $S_i$ as claimed. 
\end{proof}

\begin{remark}
Note that the projective object $P_j\in \Dfiltered(\setJ)$ in \Cref{Lem:frozensimple} 
can be different from the projective D-module $De_j$ and $\top P_j$ may not be simple. 
In particular, the module $X$ in the proof may be larger than $\rad P_j$.
\end{remark}

Iyama--Kalck--Wemyss--Yang showed that Frobenius categories can be realised as  categories of 
Gorenstein projective $A$-modules for certain algebras $A$ of finite global dimension 
\cite[Thm. 1.1]{KIWY}. 
It is natural to ask  
whether $\Dfiltered(\setJ)$ admits such a description. 
Below we provide an example of $(\End T)^{\text{op}}$ that has infinite global dimension,  
which is a feature not seen in $\sub Q_\setJ$ and $\CM C$.
Consequently, the result in \cite{KIWY} does not apply in this setting.
Note that by \cite[Thm. 2.2]{GLS06},  
\begin{equation}\label{eq:gldimbarA}
\mathrm{gl.} \dim (\End \pi T)^{\text{op}}\leq 3.
\end{equation}

\begin{example} \label{Ex:GabrielQuiversEndT1} 
Continue from \Cref{Ex:GabrielQuiversEndT}, where $T_2=\charT_2=M_{12}$.
We compute the projective dimension of the simple module $S_2$ at the vertex associated to $T_2$. 
We have the following exact sequences, where  the two right maps are  $\add (T\minus \{M_{12}\})$- and 
$\add T$-approximations, respectively, 
\[
\begin{tikzpicture}[scale=0.3,
 arr/.style={black, -angle 60}]
 \path (4,12) node(a1) {$M_4\oplus M_{13}$};
 \path (14,12) node(a2) {$M_{23}\oplus M_{14}$};
\path (22,12) node(a3) {$M_{12}$};
 \path (-3,12) node(a0) {$0$}; 

 \path[arr] (a0) edge  (a1);
 \path[arr] (a1) edge  (a2);
 \path[arr] (a2) edge  (a3);
 \end{tikzpicture}
\]
and 
\[
\begin{tikzpicture}[scale=0.3,
 arr/.style={black, -angle 60}]
 \path (6,12) node(a1) {$M_4$};
 \path (12,12) node(a2) {$M_{34}$};
\path (18,12) node(a3) {$M_4$.};
 \path (0,12) node(a0) {$0$}; 

 \path[arr] (a0) edge  (a1);
 \path[arr] (a1) edge  (a2);
 \path[arr] (a2) edge  (a3);
 \end{tikzpicture}
\]
So the minimal projective resolution of $S_2$ is infinite as follows, 
\[
\begin{tikzpicture}[scale=0.6,
 arr/.style={black,  -angle 60}]
  \path (-8, 2) node(b) {$\dots$}; 
   \path (-5, 2) node(a0) {$\funF T_3$};
 \path (-2, 2) node(a1) {$\funF T_3$};
 \path (2, 2) node(a2) {$\funF T_3\oplus \funF T_1$};
\path (7, 2) node(a3) {$\funF T_3\oplus \funF T_5$};
 \path (11, 2) node(a4) {$\funF T_2$};
  \path (14, 2) node(a5) {$S_2$};
  \path (16, 2) node(a6) {$0$.};

 \path[arr] (b) edge  (a0);
 \path[arr] (a0) edge  (a1);
 \path[arr] (a1) edge  (a2);
 \path[arr] (a2) edge  (a3);
 \path[arr] (a3) edge   (a4);
 \path[arr] (a4) edge (a5);
\path[arr] (a5) edge  (a6);
 \end{tikzpicture}
 \]
Therefore, the projective dimension of $S_2$ is infinite and $\mathrm{gl.} \dim (\End T)^{\text{op}}=\infty.$
\end{example}
\subsection{Grothendieck groups of $A$-modules}\label{sec:8.2}
Let $K(A)$ be the Grothendieck group of  $A$-modules, which has a natural 
basis consisting of the classes $[S_i]$ of simple $A$-modules. 
Let $K(\proj A)$ be the Grothendieck group of projective $A$-modules which
has the basis consisting of classes $[\funF T_i]$. There is an inclusion 
$$
K(\proj A)\subseteq K(A)
$$ 
which maps each class in $K(\proj A)$ to its corresponding class in $K(A)$.  
Any $A$-module $X$ of finite projective dimension has a class $[X]$ in 
$K(\proj A)$.

Denote by $K(\bar A)$ and $K(\proj \bar A)$ the Grothendieck groups of $\bar A$-modules, and of projective $\bar A$-modules, respectively. 
Since $\bar A$ has finite global dimension (see \eqref{eq:gldimbarA}), 
\begin{equation}\label{eq:KAbar} K(\bar A)=K(\proj \bar A). 
\end{equation}
Furthermore, $\bar A$ is a quotient of $A$, we also have an inclusion 
\begin{equation}\label{eq:KAAbar}
 K(\bar A)\subseteq K(A)
\end{equation}
mapping each class in $K(\bar A)$ to its corresponding class in $K(A)$.
In fact, $K(\bar A)$ is the subgroup generated by the classes of simple $A$-modules 
corresponding to indecomposable summands of $T$ that are not in 
$\add \charT_\setJ$. All of these simples have finite projective dimension by
\Cref{Lem:resmutablesimple} and \Cref{Lem:frozensimple}. 
So the inclusion \eqref{eq:KAAbar}
induces an injection,
\[
\beta_A\colon K(\bar A)  \to K(\proj A), ~~
(\mathrm{or } ~K(\proj\bar A)  \to K(\proj A)).
\]
\begin{proposition} \label{Lem:Grothiso} 
\begin{itemize}
\item[]
\item[(1)] For any $M\in \Dfiltered(\setJ)$, we have $[\funF M]\in K(\proj A)$ and $[\funF \projfun M]\in K(\bar A)$.
\item[(2)] Let  $M\in \sub Q_\setJ$ and  $\hat M$ a lift of $M$ to $\Dfiltered(\setJ)$. Then
$$
\beta_A([\funF M])= [\funF \hat M] - [\funF \eta \hat M],
$$
\item[(3)] The map 
$
\nu\colon K(\proj A) \rightarrow K(\proj \bar A)\oplus K(\add \charT_{\setJ}),  ~~
[\funF U]\mapsto [\funF \projfun U]+[\eta U]$, for all   $U\in \Dfiltered(\setJ)
$, is an isomorphism.
\end{itemize}
\end{proposition}
\begin{proof}
(1) From \Cref{Lem:resHom} we know that $\projd_A \funF M \leq 1$, and
so $[\funF M]$ is equal to the class of its projective resolution in $K(\proj A)$.
By \Cref{lem:dimlemma}, 
\begin{equation}\label{eq:comphom}
\Hom(T,\projfun M)=
\Hom(\projfun T, \projfun M),
\end{equation}
and so
$$
\funF\projfun M  = \Hom(\projfun T, \projfun M).
$$
Therefore $[\funF \projfun M]\in K(\bar A)$ as required. This proves (1).

(2) Let $M\in \sub Q_\setJ$. By \Cref{Lem:dense}, we can lift 
$M$ to $\Dfiltered(\setJ)$ to obtain a short exact sequence
\[
\begin{tikzpicture}[scale=0.6,
 arr/.style={black, -angle 60}]
  \path (0,1) node(a0) {$0$}; 
 \path (3,1) node(a1) {$\eta\M$};
 \path (6,1) node(a2) {$\M$};
\path (9,1) node(a3) {$M$};
 \path (12,1) node(a4) {$0,$}; 
  
 \path[arr] (a0) edge  (a1);
 \path[arr] (a1) edge  (a2);
 \path[arr] (a2) edge  (a3);
\path[arr] (a3) edge  (a4);
 \end{tikzpicture}
\]
where $\eta \hat M\in \add \charT_\setJ$ and $\pi \M=M$. By, \Cref{Lem:tiltingchar}, $\Ext^1(T, \eta\hat M )=0$.
Applying  $\funF=\Hom(T,-)$ to the sequence, we have 
\[
\begin{tikzpicture}[scale=0.6,
 arr/.style={black, -angle 60}]
  \path (0,1) node(a0) {$0$}; 
 \path (3,1) node(a1) {$\funF \eta\M$};
 \path (6,1) node(a2) {$\funF\M$};
\path (9,1) node(a3) {$\funF M$};
 \path (12,1) node(a4) {$0.$}; 
  
 \path[arr] (a0) edge  (a1);
 \path[arr] (a1) edge  (a2);
 \path[arr] (a2) edge  (a3);
\path[arr] (a3) edge  (a4);
 \end{tikzpicture}
\]
Therefore, 
$$
\beta_A([\funF M])= [\funF \hat M] - [\funF\eta \hat M].
$$
Both $[\funF \hat M]$ and $[\funF\eta \hat M]$ belong to $K(\proj A)$ (see \Cref{eq:addTprojA}), and so (2) follows.

(3) 
Let
\[
\delta:K(\proj A)\to K(\proj \bar A), ~ [\funF T_i]\mapsto [\funF \pi T_i]
\]
We have 
\[
\delta\beta_A([\funF \pi T_i])=\delta([\funF T_i]-[\funF\eta T_i])=[\funF \pi T_i]
\]
and so $\delta$
splits  $\beta_A$ and has  kernel $K(\add \charT_\setJ)$, noting that
$\pi T_j=0$ if and only if $T_j\in \add \charT_\setJ$. 
So we have the isomorphism $\nu$ as claimed.
\end{proof}

%

Note that here is another map
$$
\beta_{\bar A} \colon K(\bar A)\rightarrow K(\proj A),
$$
induced by first taking projective resolutions of simple $\bar A$-modules in $\mod \bar A$,
and then composing with the inclusion $K(\bar A)\subseteq K(\proj A)$ 
from \Cref{Lem:Grothiso} (2).

\begin{lemma} \label{Lem:equalbeta}
The two maps  $\beta_A $ and $\beta_{\bar A}$ are the same. 
\end{lemma}
\begin{proof}
The projective resolution of a mutable simple is constructed by splicing together
the mutation sequences in \Cref{Lem:resmutablesimple}. By \Cref{Lem:liftexchange},  
these mutation sequences are lifts of mutation sequences in $\sub Q_\setJ$ 
and are admissible. 
Therefore, $\beta_A = \beta_{\bar A}$
on mutable simples.
Similarly, $\beta_A = \beta_{\bar A}$ holds for the other simple $\bar A$-modules, by the construction of the projective resolution in \Cref{Lem:frozensimple}.
\end{proof}

\begin{remark}
When restricted to mutable simples, it follows from \Cref{Lem:resmutablesimple} that the map $\beta_A$ is represented by the matrix $-\mat{B}_T$, with 
respect to the bases $[S_i]$ and $[\funF T_i]$,
where $\mat{B}_T$ is the exchange matrix constructed from $T$ in  \Cref{sec:clstructnew}.
 By \Cref{Ex:GabrielQuiversEndT}, 
the map $\beta_A$ can not be naturally extended to the whole group
$K(A)$, in general.  
\end{remark}

\section{Categorification of the cluster structure on flag varieties}\label{sec9}
We give an additive 
categorification of $\CC[\flag(\setJ)]$.
This categorification will be quantised to obtain a cluster structure
on the quantum coordinate ring $\pqflag$ (see \Cref{thm:qmain1} and \Cref{rem:finrem}).

\subsection{Cluster characters}
Let $\mathcal{F}$ be a Frobenius stably 2-CY category with
a cluster structure, for instance $\sub Q_{\setJ}$ or $\Dfiltered(\setJ)$. 

\begin{definition}\label{Def:char}
A {\it cluster character} on $\mathcal{F}$, with values in a commutative domain 
${R}$, is a map
$
\charC\colon \mathcal{F} \to {R}
$
such that for any $M, N\in \mathcal{F}$, 
\begin{itemize}
\item[(1)] $\charC(M)=\charC(N)$ when $M$ and $N$ are isomorphic;
\item[(2)] $\charC(M\oplus N)=\charC(M)\cdot \charC(N)$;
\item[(3)] if $\Ext^1_{\mathcal{F}}(M,N)\cong\mathbb{C}$ and the short exact sequences in 
$\mathcal{F}$,
\begin{equation*}
\begin{tikzpicture}[scale=0.6,
 arr/.style={black, -angle 60}]
 \path (3,12) node(a1) {$M$};
 \path (6,12) node(a2) {$E$};
\path (9,12) node(a3) {$N$};
 \path (0,12) node(a0) {$0$}; 
 \path (12,12) node(a4) {$0,$};
 \path[arr] (a0) edge  (a1);
 \path[arr] (a1) edge  (a2);
 \path[arr] (a2) edge  (a3);
 \path[arr] (a3) edge  (a4);
 \end{tikzpicture}
\end{equation*}
and
\begin{equation*}
\begin{tikzpicture}[scale=0.6,
 arr/.style={black, -angle 60}]
 \path (3,1) node(a1) {$N$};
 \path (6,1) node(a2) {$F$};
\path (9,1) node(a3) {$M$};
 \path (0, 1) node(a0) {$0$}; 
 \path (12,1) node(a4) {$0,$};
 \path[arr] (a0) edge  (a1);
 \path[arr] (a1) edge  (a2);
 \path[arr] (a2) edge  (a3);
 \path[arr] (a3) edge  (a4);
 \end{tikzpicture}
\end{equation*}
are non-split, 
then $\charC(M)\cdot \charC(N)=\charC(E) + \charC(F)$.
\end{itemize}
\end{definition}

Note that post-composing a cluster character with an algebra homomorphism
between commutative domains gives a 
new cluster character. Pre-composing with full exact functors can also produce 
new cluster characters.

Suppose ${R}$ is a cluster algebra with the initial seed $(B, X)$ \cite{FZ1}, where the exchange 
matrix $\mat B$ is the same as the mutation matrix defined by a cluster tilting object 
$T$ (see \eqref{eq:BfromT}) and 
$X=\{\kappa(T_i)\colon T_i \text{ is indecomposable and } T_i \in \add T\}$. 
Then $\kappa$ offers an \emph{additive categorification} (\cite{GLS11} Section 2.8) of  $R$
and it induces a bijection between isomorphism classes of reachable cluster tilting objects and clusters in ${R}$.

Geiss--Leclerc--Schr\"{o}er \cite{GLS06, GLS08} proved Lustztig's function  
$\phi$ on $\sub Q_{\setJ}$ is a cluster character and it categorifies the cluster algebra structure on the coordinate ring $\CC[N_{\setK}]$ (see  \Cref{sec:cat} for the definition of $\CC[N_{\setK}])$. Under this categorification, $\phi(M)$ is homogeneous of 
degree $[M]$, using the grading by the root lattice $\rootlat=K(\Pi)$.
 
Fu--Keller \cite{FuKeller} gave a general construction of cluster characters with
values in formal Laurent polynomial rings
(see also \cite{CC, Palu}), which depends on a chosen cluster tilting 
object $T$. Following \cite[(1.3)]{GLS12}, Fu--Keller's cluster character 
applied to $\sub Q_{\setJ}$ can be written as 
\begin{equation} \label{eq:fukellercluster}
\theta^T(M) = z^{[\funF M]}\sum_{d}\chi{(\Gr_d(\Ext^1(T,M))}z^{-\beta(d)},
\end{equation}
where $\beta=\beta_{\overline{A}}$ is as defined in \Cref{sec:8.2}. 

\begin{remark} Strictly speaking, 
the functor $\funF$ is defined as a $\Hom$-functor by a cluster tilting object in 
$\Dfiltered(\setJ)$ (see \Cref{sec:8.1}), but by \eqref{eq:comphom} for any 
$M\in \sub Q_\setJ$,
\[
\Hom(\T, M)=\Hom(T, M),
\]
where $\T$  is a cluster tilting object such that $\pi \T=T$. 
So $[\funF M]$ makes sense in \eqref{eq:fukellercluster}.
\end{remark}

By identifying $\phi(T_i)$ with $z^{[\funF T_i]}$ in \eqref{eq:fukellercluster}, 
Geiss--Leclerc--Schr\"{o}er proved the following.

\begin{theorem} \cite[Thm. 4]{GLS12} \label{thm:GLSEqual}
For all $M\in \sub (Q_{\setJ})$, 
$
\phi(M)=\theta^{T}(M). 
$
\end{theorem}

This theorem produces an explicit expression of  
$\phi(M)$ as a Laurent polynomial in any chosen cluster. 

\subsection{Cluster characters on $\Dfiltered(\setJ)$}
Define 
\begin{equation}\label{def:charPhi}
\Phi\colon \Dfiltered(\setJ) \rightarrow {\CC[N_{\setK}][K(\add \charT_{\setJ})]}, M\mapsto \phi (\projfun M)\cdot z^{[\eta M]},
\end{equation}
and
\begin{equation}\label{def:charTheta}
\Theta^T\colon \Dfiltered(\setJ) \to {\CC[K(\proj \bar{A})][K(\add \charT_{\setJ})]}, M\mapsto 
\theta^T(\projfun M)\cdot z^{[\eta M]}. 
\end{equation}
These two maps are well-defined, 
since for any $M\in \Dfiltered(\setJ)$, 
\[\eta M \in \add \charT_{\setJ}.\] 
Note that by $\CC[N_{\setK}][K(\add \charT_{\setJ})]$, we mean 
the {\em formal Laurent polynomial ring} of the Grothendieck group $K(\add\charT_{\setJ})$ with 
coefficients in $\CC[N_{\setK}]$. This ring is graded by $K(\sub Q_\setJ) \times K(\add \charT_{\setJ})$, which is isomorphic to $\grothEJ$ by the isomorphism 
$\sigma$ from \Cref{Lem:isom1}.

\begin{lemma}\label{Lem;charPhiTheta} The maps
$\Phi$ and $\Theta^T$ are cluster characters, and by 
identifying $\Phi(T_i)$ with $\Theta^T(T_i)$, for any $M\in \Dfiltered(\setJ)$, 
$$\Phi(M)=\Theta^T(M).$$
Moreover, $\Phi(M)$ has degree $\sigma([M])$.
\end{lemma}
\begin{proof}
By the definition of the exact structure  $\mathcal{E}$ on $\Dfiltered(\setJ)$, $\eta$ is additive 
on admissible short exact sequences. Therefore,  
for any $M, N\in \Dfiltered(\setJ)$, 
\begin{equation*}\label{Eq:property2}
\eta(M\oplus N)= \eta M \oplus \eta N 
\end{equation*}
and when $M, N, E, F$ are as in  \Cref{Def:char} (3), 
\begin{equation*}\label{Eq:property3}
\eta E=\eta F=\eta M \oplus \eta N.
\end{equation*}
So,  $\Phi$ and $\Theta^T$ are cluster characters, since $\phi$ and $\theta^T$ are. 
That $\Phi(M)=\Theta^T(M)$ follows from the equality 
$\phi(M)=\theta^{\projfun T}(M)$ from \Cref{thm:GLSEqual}.

Finally, that $\Phi(M)$ has degree $\sigma([M])$ follows using the isomorphism 
$\sigma$ from \Cref{Lem:isom1} and that $\phi(\projfun M)$ has degree 
$[\projfun M]$
\end{proof}

It is also possible to define a cluster character along the lines of \cite{FuKeller} by 
adapting the construction of $\theta^T$. We define  
\[
\Psi^T\colon \Dfiltered(\setJ) \to \CC[K(\proj A)],  
M\mapsto z^{[T, M]}\sum_{d}\chi(\Gr_d(\ext^1(T, M))z^{-\beta(d)},
\]
where $\beta=\beta_A$ is the map induced by taking
projective resolutions of mutable simple $A$-modules. Recall that $\ext^1(T, M)$ is only 
supported at mutable vertices, which 
all have finite projective dimension, by \Cref{Lem:resmutablesimple}.  
So $\beta$ and therefore $\Psi^T$ is well-defined.  

Recall the isomorphism  from \Cref{Lem:Grothiso} (3), 
\[
\nu\colon K(\proj A) \rightarrow K(\proj \bar A)\oplus K(\add \charT_{\setJ}).
\]
The inverse induces an isomorphism of Laurent polynomials, 
\begin{equation}\label{eq:Cnuinv}
\CC[\nu^{-1}]\colon \CC[K(\proj \bar A)\oplus K(\add \charT_{\setJ})] \rightarrow \CC[K(\proj A)].
\end{equation}
\begin{lemma}\label{Lem:Psi}
The map $\Psi^T\colon \Dfiltered(\setJ) \to \CC[K(\proj A)]$ is the composition of the cluster character
$\Theta^T$ with the algebra isomorphism $\CC[\nu^{-1}]$. As a consequence, it is a cluster character.
\end{lemma}
\begin{proof}
Following the definition of $\nu$, 
we have \[\CC[\nu] (z^{[\funF T_i]}) = z^{[\funF \projfun T_i] + [\eta T_i]}.\]
By \Cref{Lem:equalbeta}, $\beta_{\bar A}(d)=\beta_{A}(d)$ for any $d\leq \ext^1(T, M)$,  since $d$
is only supported on mutable vertices.  Also, $\ext^1(T,M)=\Ext^1(\projfun T, \projfun M)$ by \Cref{Thm:fullexact}.
So $\Psi^T$ is the composition of 
the cluster character $\Theta^T$ in \eqref{def:charTheta} and
the algebra isomorphism $\CC[\nu^{-1}]$, and thus a cluster character.
\end{proof}

We remark that it is also possible to prove the lemma, using similar arguments to
 \cite{Palu, FuKeller} and the fact that $\ext^1(T, M)=\Ext^1(\projfun T, \projfun M)$. 

\subsection{Categorification of $\CC[\flag(\setJ)]$}\label{sec:cat}
Let $\setJ\subseteq [n-1]$, $\setK=[n-1]\backslash \setJ$, and let 
$$
N_{\setK}\subseteq P_{\setK}\subseteq \gl_n
$$
be the unipotent and parabolic subgroups of $\gl_n$ determined by $\setK$, respectively. 
Let 
\[\flag(\setJ) = \gl_n/P^-_{\setK}\] 
be the flag variety associated to $\setJ$,
where  $P^-_{\setK}$ is the opposite parabolic of $P_{\setK}$.
Let $\wtl_{\setJ}$ be the sublattice of the weight lattice $\wtl$ generated by the fundamental weights $\omega_j$, $j\in \setJ$, 
 \[
\wtl_{\setJ}=\{\lambda=\sum_{j\in \setJ}a_j\omega_j\colon a_j \in  \ZZ\}
\] 
and let $\wtl_\setJ^{+}=\wtl^+\cap \wtl_\setJ$ be the monoid consisting of the dominant weights in $\wtl_\setJ$. 
The homogeneous coordinate ring $\CC[\flag(\setJ)]$ is multigraded by
 $\wtl_{\setJ}$.  More precisely, by the  Borel-Weil Theorem,
\begin{equation}\label{eq:BWthm}
\CC[\flag(\setJ)]=\bigoplus _{\lambda \in \wtl_{\setJ}^+} L(\lambda),
\end{equation}
where $L(\lambda)$ is the irreducible representation with highest 
weight $\lambda$. Furthermore, this coordinate ring is generated by the 
subspace $\oplus _{j \in \setJ} L(\omega_j)$.

Recall the isomorphism in \eqref{Eq:gtowt1},
$\zeta\colon \wtl \to K(\Dfiltered), ~~\epsilon_i\mapsto [\Delta_{i}],$
which restricts to an isomorphism between $\wtl^+$ and $K^+(\Dfiltered)$ (see \Cref{sec:5.1}
for definitions). Note also
the isomorphism from \Cref{Lem:isom1},
 $\sigma\colon \grothE \rightarrow K(\Pi) \oplus K(\add \charT)$.
Therefore,  there is a finer decomposition of \eqref{eq:BWthm} into a direct sum of weight-spaces,
\begin{equation}\label{eq:gradingE}
\CC[\flag(\setJ)]=\bigoplus_{(d, \lambda)} L(\lambda)_{d+\lambda}
=\bigoplus_{\omega}\CC[\flag(\setJ)]_{\omega}, 
\end{equation}
where the first sum is over  $(d,  \lambda)\in  K^{+}(\sub Q_{\setJ})\times K^{+}(\add \charT_{\setJ})$, the second sum is over $\omega\in K^{+}(\Dfiltered(\setJ))_{\mathcal{E}}$ and the
summands are related by $\sigma(\omega)=(d,\lambda)$.
By convention 
$L(\lambda)_{d+\lambda}$ is zero if $d+\lambda$ is not a weight of the representation $L(\lambda)$.

For any $I=\{i_1, \dots, i_k\}$,
denote by $\Delta_I$ the minor of the submatrix consisting of the columns $i_1, \dots, i_k$ in the 
first $k$ rows. 
Note that in this case $\Delta_I$ is contained in the weight 
space $ \CC[\flag(\setJ)]_{\omega_k}$ of the fundamental weight $\omega_k$. 

Recall the quotient map, 
$$
\rho\colon\CC[{\flag(\setJ)}]\mapsto \CC[N_{\setK}], ~ f\mapsto f|_{N_K}.
$$
In particular,  $\rho(\Delta_{[j]})=1$, for any $j\in \setJ$.
The restriction of $\rho$ to a representation $L(\lambda)$ is injective  and so we have 
an $\grothEJ$-graded injection
\begin{equation}\label{eq:rhohat}
\fln\colon \CC[{\flag(\setJ)}]\rightarrow \CC[N_{\setK}][K(\add \charT_{\setJ})], ~~x\mapsto  \rho(x)\cdot z^{\lambda},
\end{equation}
where $x \in \CC[{\flag(\setJ)}]_{\omega}$ and $\sigma(\omega)=(d, \lambda)$, 
remembering the homogeneous multi-degree $\lambda$.  
By \cite[Prop. 9.1]{GLS08} and \cite[Thm. 3]{GLS06},
\begin{equation}\label{eq:spanning}  
 \fln(L(\lambda)) \text{ is spanned by }
\phi(M)\cdot z^\lambda,
\text{ where } M\subseteq \oplus_i Q_i^{\lambda_i} \text{ and } 
\lambda = \sum_i \lambda_i [\charT_i]. 
\end{equation} 
By \Cref{Lem:dense}, every 
such $M$ is of the form $\projfun (\hat M)$
for some $\hat M\in \Dfiltered(\setJ)$, which is unique up to isomorphism when $[\eta \hat M]=\lambda$. We can therefore
define the map
\[
\charfl \colon \Dfiltered(\setJ)\to \CC[\flag(\setJ)],  ~~M\mapsto  \fln^{-1}(\Phi(M)).
\]

\begin{theorem} \label{Thm:classiccat}
The map $\charfl \colon\Dfiltered(\setJ) \to \CC[\flag(\setJ)]$ is a cluster 
character which gives an additive categorification of the cluster algebra structure
on $\CC[\flag(\setJ)]$. Moreover, 
\begin{itemize}
\item[(1)] $\charfl(M)$ has degree $[M]\in \grothEJ$,
\item[(2)] $\charfl(M)$ for all $M\in \Dfiltered(\setJ)$ span $\CC[\flag(\setJ)]$,
\item[(3)] $\charfl(M_I) = \Delta_I$.
\end{itemize}
Finally, by identifying $\charfl(T_i)$ with $\Psi^T(T_i)$,  for any $M\in \Dfiltered(\setJ)$, we have
$$\charfl(M)=\Psi^T(M).$$
\end{theorem}
\begin{proof}
The map $\charfl$ is the composition of the cluster character $\Phi(M)$ from \Cref{Lem;charPhiTheta} and the algebra homomorphism $\fln^{-1}$ defined on
the image of $\fln$. So $\charfl$ is a cluster character. As $(\sub Q_\setJ, \phi)$ provides a categorification of
 $\CC[N_K]$ (\cite{GLS08}) and the cluster structure on 
$\Dfiltered(\setJ)$ is lifted from that of 
$\sub Q_J$ (see \S \ref{sec:7.2} and \S \ref{sec:7.3}),
$(\Dfiltered(\setJ), \Phi)$ gives a 
 categorification of $\CC[\flag(\setJ)]$.

The representation $\fln(L(\lambda))$ is spanned by elements $\phi(\projfun M)\cdot z^\lambda$ with 
$\eta(M)=\lambda$ (see \eqref{eq:spanning}), and every such element has a unique preimage $\charfl(M)$
(see \Cref{Lem:dense}).
This, together with the isomorphism $\sigma$ from \Cref{Lem:isom1} implies (1), and, 
in conjunction with \eqref{eq:BWthm},  implies (2).

By \cite[\S 6.2]{GLS08}, $\phi({\projfun (M_I)})$ is the restriction 
of the minor $\Delta_I$ to $N_{\setK}$, and so \[\charfl(M_I)=\Delta_I. \]
This proves (3).

Finally, $\charfl(M)=\Psi^T(M)$, by the equality $\Phi(M)=\Theta^T(M)$ from \Cref{Lem;charPhiTheta} and the description of $\Psi^T$ using $\Theta^T$ in 
\Cref{Lem:Psi}. 
\end{proof}
\begin{remark}
Geiss--Leclerc--Schr\"{o}er \cite[\S 10]{GLS08} proved that the coordinate ring 
$\CC[\flag(\setJ)]$ is a cluster algebra by homogenisation of the cluster structure on $\CC[N_K]$, 
and Jensen--King--Su \cite{JKS1} gave a categorification of the cluster algebra $\CC[\Gr(k, n)]$ by 
homogenising Geiss--Leclerc--Schr\"{o}er's cluster character $\phi$.
\Cref{Thm:classiccat} gives a categorification of the cluster algebra on $\CC[\flag(\setJ)]$ and a
new categorification of $\CC[\Gr(k, n)]$.
The approach is different from \cite{JKS1}, by introducing the map $\fln$, which has a transparent
quantisation (see \Cref{sec16}). 
\end{remark}

\begin{remark} \label{rem:classicalcase} \qquad
\begin{itemize}

\item[(1)] 
Localising $ \CC[\flag(\setJ)]$ at the minors $\Delta_{[j]}$ for $j\in \setJ$ turns $\fln$ into 
 an isomorphism, 
\begin{equation*} \label{eq:rhohat1}
\fln\colon \CC[\flag(\setJ)]_{\{\Delta_{[j]}\colon j\in \setJ\}} \rightarrow \CC[N_{\setK}][K(\add \charT_{\setJ})].
\end{equation*}
 
\item[(2)] The construction of $\charfl$ produces a basis of functions of the form $\charfl(M)$ for \emph{generic} $M$, using the corresponding basis for $\CC[N_K]$ (see e.g. \cite{GLS12}). 

\item[(3)] As explained in Remarks \ref{Rem:missingobject} and \ref{Rem:missingobject2}, 
it makes sense  to include $n$ in $\setJ$. In this case, we obtain a larger coordinate ring $\CC[\flag(\setJ)]$ 
by adding the determinant $\Delta_{[n]}$ as a generator, and \Cref{Thm:classiccat} is still true.
\end{itemize}
\end{remark}

\begin{example}
Consider the cluster structure on $\CC[\flag(\setJ)]$ with $n=4$ and $\setJ=\{1, 3\}$. 
In this case, the AR-quiver of 
$\Dfiltered(\setJ)$ is described in  \Cref{Ex:ARquivers} (2). Let 
\[
T=M_{3}\oplus M_{134}\oplus X\oplus M_4\oplus M_{234}\oplus M_1\oplus M_{12}.
\]
Then $T$ is a cluster tilting object and 
the algebra $\CC[\flag(\setJ)]$ is a cluster algebra of type 
$\mathbb{A}_1\times \mathbb{A}_1$.
The first two summands are mutable, and the others are projective-injective 
in $\Dfiltered(\{1, 3\})$. 
Note that $X$ is not a submodule of $D_0$ and so not of the form $M_I$ for some $I\subseteq [4]$.
However, following the mutation sequences for pairs $(M_3,  M_{124})$ and
$(M_2, M_{134})$,
and the fact that $\hat{\Phi}$
is a cluster character, we have 
\begin{eqnarray*}
 \hat{\Phi}(X)&=& \Delta_{3}\Delta_{124} - \Delta_4 \Delta_{123}\\
                    &=&\Delta_{2}\Delta_{134} - \Delta_1 \Delta_{234}.
\end{eqnarray*}
\end{example}

\section{Separation, rigidity and quasi-commutation rules}\label{sec10}
We investigate the connections between the combinatorial notion of  
weak and strong separation and 
the extension groups $\ext^1(-, -)$ and $\Ext^1(-, -)$ in $\Dfiltered(\setJ)$. In particular, 
maximal collections of weakly separated sets 
are in one-to-one correspondence with (basic) maximal rigid modules 
with respect to $\ext^1(-, -)$ whose indecomposable summands all have rank one, strong separation can be detected by $\Ext^1(-, -)$. Furthermore, we prove that 
Quasi-commutation rules for quantum 
minors can be computed using dimensions of $\Hom(-, -)$-spaces or $\Ext^1(-, -)$, and that 
$\Ext^1(-, -)$ detects when the product of two quantum minors is invariant under the bar involution.

\subsection{The Grassmannian cluster category $\CM C$}\label{sec:Grass}
Recall the Grassmannian cluster category $\CM C$ constructed in \cite{JKS1}. Let  $C=C_{k, n}$ be the complete 
path algebra of a double cyclic quiver with relations. 
 For instance, when $n=5, k=2$, the algebra $C$ is 
defined by the  quiver, 
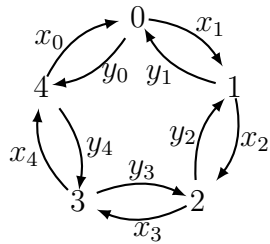
\[
\begin{tikzpicture}[scale=0.9,
 Earrow/.style={thick,-latex},
 Barrow/.style={thick,-latex, black},
 Rline/.style={thick,teal},
 Vline/.style={thick,teal},
 Bline/.style={thick,blue} ]
\newcommand{\alp}{10}
\newcommand{\centerarc}[4]{($(#1)+(#2:#4)$) arc (#2:#3:#4)}
\newcommand{\qvertex}{\small $\bullet$}

\path (0.5, 1) node (bb) {};
\newcommand{\radius}{1.8cm}
\foreach \j in {0,...,4}{
  \path (90-72*\j:\radius) node (w\j) {\color{black}\j};
  \path (162-72*\j:\radius) node[black] (v\j) {};
}
\path[Earrow] (v1) edge[bend left=25,thick] node [right, xshift=-0.1cm, yshift=0.1cm] {$x_1$} (w1);
\path[Earrow] (v2) edge[bend left=25,thick] node [right, xshift=-0.1cm, yshift=0cm] {$x_2$} (w2);
\path[Earrow] (v3) edge[bend left=25,thick] node [below, xshift=0.08cm, yshift=0.1cm]  {$x_3$} (w3);
\path[Earrow] (v4) edge[bend left=25,thick] node[left, xshift=0.1cm, yshift=-0.05cm]  {$x_4$} (w4);
\path[Earrow] (v0) edge[bend left=25,thick] node [left, xshift=-0cm, yshift=0cm]  {$x_0$} (w0);

\path[Earrow] (w1) edge[bend left=23,thick] node [left, xshift=0.15cm, yshift=-0.15cm] {$y_1$} (v1);
\path[Earrow] (w2) edge[bend left=23,thick] node [left, xshift=0.1cm, yshift=0cm] {$y_2$} (v2);
\path[Earrow] (w3) edge[bend left=23,thick] node [above, xshift=0cm, yshift=-0.1cm]  {$y_3$} (v3);
\path[Earrow] (w4) edge[bend left=23,thick] node[right, xshift=-0.05cm, yshift=-0cm]  {$y_4$} (v4);
\path[Earrow] (w0) edge[bend left=23,thick] node [right, xshift=-0.05cm, yshift=-0.05cm] {$y_0$} (v0);
\end{tikzpicture}
\]
with relations \[x_ix_{i-1}=y_{i+1}y_{i+2}y_{i+3},  ~~ x_iy_i=y_{i+1}x_{i+1},  ~~\forall ~ i \in \ZZ_5.\] The centre of $C$ is  
the complete polynomial ring $R=\CC[[t]]$ with $t=\sum_i x_iy_i$. 
Now the category $\CM C$ consists of Cohen-Macaulay $C$-modules, that is, $C$-modules 
that are free over the centre $R$.

There is a full exact quotient functor 
\[
\tilde{\projfun} \colon \CM C \ra \sub Q_{\setJ}, ~ M\mapsto M/Ce_0M.
\]
In particular, this induces isomorphisms on extension groups \cite[Cor. 4.6]{JKS1} and an equivalence 
of categories \[ \CM C/\add Ce_0 \ra \sub Q_{\setJ}. \] 

\emph{The index set $[n]$ for the vertices and the arrows of the quiver defining $C$ is identified with 
$\ZZ_n$. In particular, $0\equiv n\bmod n.$}

The rank one  modules 
in $\CM C$ are classified by subsets of $[n]$ with $k$ elements, which are referred to 
as {\em $k$-sets}.
More precisely, any $k$-set $I$ defines a rank one module $N_I$ in $\CM C$, where the 
vector space at each vertex is $R$, the arrow 
$x_i$ for any $i\in I$ acts as $1$, while the other $x$-arrows act as $t$, and the actions 
of $y$-arrows are determined by the generating relations of $C$. Any
rank one module is isomorphic to some $N_I$, and $N_I\not\cong N_J$ for any 
distinct $k$-sets $I$ and $J$. By construction, 
we have the following.
\begin{lemma} \label{Eq:twoprojections} For any $k$-set $I$,  
$\projfun M_I=\tilde{\projfun}N_I.$
\end{lemma}
\begin{example}
Let $k=2, ~n=5$ and $I=\{3, 5\}$. Then the modules $N_I$, $M_I$ and $\projfun M_I=\tilde{\projfun} N_I$ are as follows, 
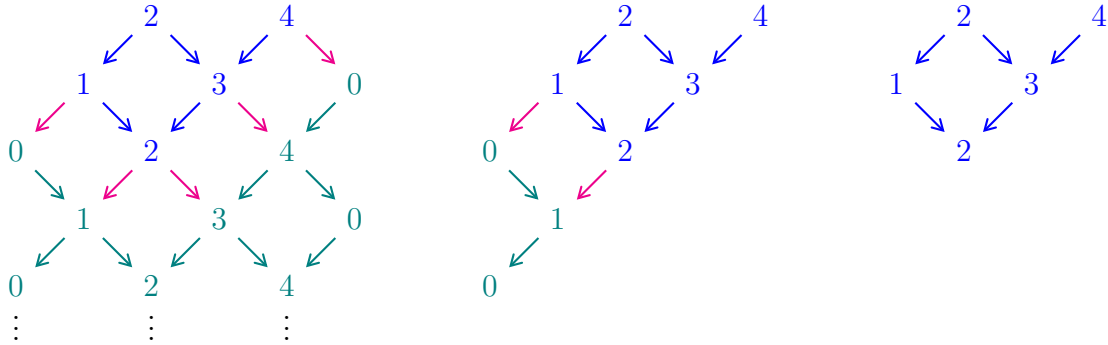
\[
\begin{tikzpicture}[scale=0.9,
teal arr/.style={magenta, -angle 60},
blue arr/.style={blue, -angle 60},
magenta arr/.style={teal, -angle 60}, thick] 
\path (2, 5) node (e2) {\color{blue}$2$};
\path (4, 5) node (e4) {\color{blue}$4$};
\path (1, 4) node (e1) {\color{blue}$1$};
\path (3, 4) node (e3) {\color{blue}$3$};
\path (5, 4) node (e0) {\color{teal}$0$};

\path[blue arr] (e2) edge (e1);
\draw[blue arr] (e2) edge (e3);
\draw[blue arr] (e4) edge (e3);
\draw[teal arr] (e4) edge (e0);

\path (2, 3) node (e22) {\color{blue}$2$};
\path (4, 3) node (e24) {\color{teal}$4$};
\path (0, 3) node (e20) {\color{teal}$0$};

\draw[blue arr] (e1) edge (e22);
\draw[teal arr] (e1) edge (e20);
\draw[blue arr] (e3) edge (e22);
\draw[teal arr] (e3) edge (e24);
\draw[magenta arr] (e0) edge (e24);

\path (1, 2) node (e21) {\color{teal}$1$};
\path (3, 2) node (e23) {\color{teal}$3$};
\path (5, 2) node (e30) {\color{teal}$0$};

\draw[magenta arr] (e20) edge (e21);
\draw[teal arr] (e22) edge (e21);
\draw[teal arr] (e22) edge (e23);
\draw[magenta arr] (e24) edge (e30);
\draw[magenta arr] (e24) edge (e23);

\path (2, 1) node (e32) {\color{teal}$2$};
\path (4, 1) node (e34) {\color{teal}$4$};
\path (0, 1) node (e40) {\color{teal}$0$};

\draw[magenta arr] (e21) edge (e40);
\draw[magenta arr] (e21) edge (e32);
\draw[magenta arr] (e23) edge (e32);
\draw[magenta arr] (e23) edge (e34);
\draw[magenta arr] (e30) edge (e34);

\path (2, 0.5) node (e32) {$\vdots$};
\path (4, 0.5) node (e34) {$\vdots$};
\path (0, 0.5) node (e40) {$\vdots$};
\path (9, 5) node (f2) {\color{blue}$2$};
\path (11, 5) node (f4) {\color{blue}$4$};
\path (8, 4) node (f1) {\color{blue}$1$};
\path (10, 4) node (f3) {\color{blue}$3$};

\draw[blue arr] (f2) edge (f1);
\draw[blue arr] (f2) edge (f3);
\draw[blue arr] (f4) edge (f3);

\path (9, 3) node (f22) {\color{blue}$2$};
\path (7, 3) node (f20) {\color{teal}$0$};

\draw[blue arr] (f1) edge (f22);
\draw[teal arr] (f1) edge (f20);
\draw[blue arr] (f3) edge (f22);
\path (8, 2) node (f21) {\color{teal}$1$};
\draw[magenta arr] (f20) edge (f21);
\draw[teal arr] (f22) edge (f21);
\path (7, 1) node (f40) {\color{teal}$0$};
\draw[magenta arr] (f21) edge (f40);
\path (14, 5) node (g2) {\color{blue}$2$};
\path (16, 5) node (g4) {\color{blue}$4$};
\path (13, 4) node (g1) {\color{blue}$1$};
\path (15, 4) node (g3) {\color{blue}$3$};
\draw[blue arr] (g2) edge (g1);
\draw[blue arr] (g2) edge (g3);
\draw[blue arr] (g4) edge (g3);
\path (14, 3) node (g22) {\color{blue}$2$};
\draw[blue arr] (g1) edge (g22);
\draw[blue arr] (g3) edge (g22);
\end{tikzpicture}
\]
In $N_I$, the two columns of zeros are identified, and the teal coloured part is the projective module  
$Ce_0$, while the teal part in $M_I$ is $\charT_2$.
\end{example}

\subsection{Weak separation and rigidity}
\begin{definition} \cite[\S 1]{LZ} \label{def:non-crossing}
Two  subsets $I$ and $J$ of $[n]$  are  
{\em weakly separated}  if 
one of the following two conditions holds: 
\begin{itemize}
\item[(i)]  $|I|\geq |J|$ and $J\setminus {I}$ can be written as a disjoint union $J'\cup J''$ so that \[J'< (I\setminus J) <J'';\]
\item[(ii)]  $|J|\geq |I|$ and $I\setminus {J}$ can be written as a disjoint union $I'\cup I''$ so that \[I'< (J\setminus I) <I'',\]
\end{itemize}
where $I<J$ means that $i<j$ for all $i\in I$ and $j\in J$. 
A collection of subsets is said to be {\em weakly separated} if 
its elements are pairwise weakly separated.
\end{definition}

We remark that when the cardinalities of $I$ and $J$ are equal, then $I$ and $J$ are 
weakly separated if and only if they  are \emph{non-crossing}, that is, 
there do not exist $a,c\in I\setminors J$ and $b,d\in J\setminors I$ such
that $a,b,c,d$ are cyclically ordered.

We will prove that $I$ and $J$ are weakly separated
if and only if $\ext^1(M_I, M_J)=0$, or equivalently $M_I\oplus M_J$ is 
rigid in $\Dfiltered(\setJ)$.  

\begin{lemma} \label{Lem:rankoneext}
Let $I,J\subseteq [n]$.
Then 
$$
\dim \ext^1(M_I, M_J) = \dim \Hom(M_I, M_J) + \dim \Hom(M_J, M_J) - |I \cap J| - \min\{|I|, |J|\}.
$$
\end{lemma}
\begin{proof} By the definition of the bilinear form $\langle -, -\rangle$ in \eqref{Eq:bilin0}, 
\[
\langle [M_I], [M_J] \rangle = |I\cap J| \text{ and }
\langle [\eta M_I], [\eta M_J] \rangle = \min\{|I|, |J|\}.
\]
The lemma then follows from   \Cref{Lem:Symbil}.
\end{proof}

Below, we will give a combinatorial formula for computing $\dim \Hom(M_I, M_J)$. 
Let $I = \{i_1<\cdots< i_k\}$. For $1\leq d\leq k$, let $S(I, d)$ and $E(I, d)$ be 
the $d$ smallest elements and the $d$ largest elements in $I$, respectively. That is, 
\[
S(I, d)=\{i_s\colon 1\leq s\leq d\} \text{ and }E(I, d)=\{i_s\colon k-d+1\leq s\leq k\}. 
\] 
We have $S(I,|I|)=E(I,|I|)=I$.  Let 
$S(I,0)=\emptyset=E(I,0)$. Let
$J=\{j_1<\dots < j_k\}$ be another $k$-set.
We define 
\[
I\preceq J \text{ if } i_s\leq j_s, \text{ for all } 1\leq s\leq k.
\]
By computation, using the construction of the modules $M_I$, we have the following straightforward lemma.
\begin{lemma} \label{Lem:quotient} 
Let $I, J\subseteq [n]$. 
\begin{itemize}
\item[(1)] $M_I\subseteq M_J$ if and only if $I\preceq E(J,|I|) $.
\item[(2)] Suppose that $M_I\subseteq M_J$. Then 
 $M_J/M_I$ is a good module if and only if $I=E(J,|I|)$.
Moreover, in this case $M_J/M_{E(J,d)}=M_{S(J,|J|-d)}$.
\end{itemize}
\end{lemma}
\begin{lemma} \label{Lem:rankonehom}
Let $I, J\subseteq [n]$. Then  
\begin{equation}\label{dimhom}
\dim \Hom(M_I, M_J)=\max\{d \colon  S(I, d)\preceq E(J, d)\}.
\end{equation}
Equivalently, 
\begin{equation}\label{dimhom1}
 \dim \Hom(M_I, M_J)=d, \text{ where }  S(I, d)\preceq E(J, d)
\text{ but } S(I, d+1)\not\preceq E(J, d+1).
\end{equation}
\end{lemma}
\begin{proof}
If $f\colon M_I\to M_J$, then $\im (f)\leq M_J$ is a good module. 
Thus,
\[
\im (f)=M_{S(I,x)} \text{ with } 
S(I, x) \preceq E(J, x)
\] 
for some $0 \leq x \leq |I|$, by \Cref{Lem:quotient}.

We recall the faithful functor from \Cref{Lem:charTequivalence},
$$
\e_0\colon \Dfiltered\to \mod \CC[t]/(t^n), ~~M\mapsto e_0M.
$$
We have
$$
\End  M_I= \End (e_0M_I) = \CC[t]/(t^{|I|})
$$
and 
$$
\Hom(M_I, M_J)\leq \Hom (e_0M_I, e_0M_J).
$$ 
Therefore, $\Hom(M_I,M_J) = \CC[t]/(t^{d})$,
where 
\[
d=\dim \Hom(M_I, M_J)
\] 
is the largest $x$ so that $M_{S(I, x)}=M_I/M_{E(I,|I|-x)}\leq M_J$.
This proves \eqref{dimhom}. 

Next, by the definition of the order $\preceq$, the left-hand sides of the 
equalities in \eqref{dimhom} and \eqref{dimhom1} are equal. 
So \eqref{dimhom1} follows from \eqref{dimhom}.
\end{proof}

For any subset $I$ of $[n]$ with $s=|I|$, Oh--Postnikov--Speyer \cite{OPS} defined 
the set
\[
\pad(I)=I\cup [n+s+1, 2n],
\] 
which is an $n$-subset of $[2n]$.
Thus $M_{\pad(I)}$ is the submodule of the 
projective module ${D_{2n}}e_0$ whose $\Delta$-support is
$\pad(I)$, where $D_{2n}$ is the Auslander algebra of {$\CC[t]/(t^{2n})$}.
Similarly, $N_{\pad(I)}$ denote the rank one module in $\CM C_{n, 2n}$ 
corresponding to the $n$-set $\pad(I)$.

\begin{lemma} \cite[Lemma 12.7]{OPS} \label{Lem:padseparated} 
Let $I, J\subseteq [n]$. Then $I$ and $J$ are weakly separated
if and only if $\pad(I)$ and $\pad(J)$ are weakly separated. 
\end{lemma}
\begin{lemma} \label{Lem:padhom}
Let $I, J\subseteq [n]$ with  $|J|=t$. Then 
\begin{equation}\label{Eq:dimhom}
\dim \Hom(M_{\pad(I)}, M_{\pad(J)}) =\dim \Hom(M_I, M_J)+ (n-t).
\end{equation}
\end{lemma}
\begin{proof}
Let $I=\{i_1 < \dots < i_s\}$ and $J=\{j_1 < \dots < j_t\}$. 
Let $d=\dim \Hom(M_I, M_J)$. By  \Cref{Lem:rankonehom},
\[d=\max\{c \colon  S(I, c)\preceq E(J, c)\}.\] 
In particular,  $d \leq t \text{ and } d\leq s$. 
Denote the elements in $\pad(I)$ by 
$i_1< \dots<i_s<\dots< i_n$, where the last $(n-s)$ elements are $[n+s+1, 2n]$. 
We have, 
\[ 
S(\pad(I), d+(n-t))=S(I, d)\cup \{i_{d+1}, \dots, i_{d+(n-t)}\} 
\]
and 
\[ 
E(\pad(J), d+(n-t))=E(J, d)\cup [n+t+1, 2n]. 
\]
Clearly, 
\[
S(I, d) \preceq E(I, d) ~~\text{ and }~~ \{i_{d+1}, \dots, i_{d+(n-t)} \}  \preceq [n+t+1, 2n].
\]
So 
\begin{equation}\label{Eq:ineq}
S(\pad(I), d+(n-t))\preceq E(\pad(J), d+(n-t)).
\end{equation}
We show that $d+(n-t)$ is maximal such that \eqref{Eq:ineq} is true. Indeed, when $d=t$, then $d+(n-t)=n$, 
which is clearly maximal. Now, we assume that $d<t$. We have
\begin{equation}\label{eq:newES1}
S(\pad(I), d+(n-t)+1)=S(\pad(I), d+1)\cup \{i_{d+2}, \dots, i_{d+(n-t)+1}\}, 
\end{equation}
and 
\begin{equation}\label{eq:newES2}
E(\pad(J), d+(n-t)+1)=E(J, d+1)\cup [n+t+1, 2n]. 
\end{equation}
Note that 
\[
S(\pad(I), d+1)=\left\{\begin{tabular}{ll} $S(I, d+1)$ &  if  $d<s$; \\ 
$S(I, d)\cup \{n+d+1\}$  &  if $d=s$.\end{tabular}
 \right. 
\]
 By the maximality of $d$ for the first case, while by the fact that  $j<n+d+1$ for any $j\in J$ for the second case, 
 we have  
\[
S(\pad(I), d+1)\not\preceq E(J, d+1).
\]   
Consequently, 
\[
S(\pad(I), d+1)\not\preceq E(\pad(J), d+1).
\]
Therefore, by \eqref{eq:newES1} and \eqref{eq:newES2}, 
\[
S(\pad(I), d+ (n-t)+1)\not\preceq E(\pad(J), d+(n-t)+1). 
\]
Hence, by \Cref{Lem:rankonehom},  $d+(n-t)$ is maximal and 
\eqref{Eq:dimhom} holds. 
\end{proof}
\begin{proposition} \cite[Prop. 5.6]{JKS1} \label{Prop:WS1} 
Let $I, J\subseteq [n]$ be $k$-sets. Then  $\Ext^1(N_I, N_J)=0$ 
if and only if $I$ and $J$ are weakly separated.
\end{proposition}
Below, we have a similar characterisation of weak 
separation for any two subsets of $[n]$ using the exact structure $\mathcal{E}$ on  
$\Dfiltered$.
\begin{theorem}\label{thm:extWS}
Let $I,J\subseteq [n]$. Then $\ext^1(M_I, M_J)=0$ if and only if 
$I$ and $J$ are weakly separated.
\end{theorem}
\begin{proof}
Let $s=|I|$ and $t=|J|$. As $\ext^1(-, -)$ is symmetric, 
without loss of generality, we may assume $s\leq t$. 
By \Cref{Lem:padhom}, 
\begin{align*}
&\dim \Hom(M_{\pad(I)}, M_{\pad(J)}) +\dim \Hom(M_{\pad(J)}, M_{\pad(I)}) \\
=& \dim \Hom(M_I, M_J)+ (n-t) + \dim \Hom(M_J, M_I)+ (n-s). 
\end{align*}
Therefore, 
\[
\dim \Hom(M_I, M_J)+\dim \Hom(M_J, M_I)=|I\cap J|+s
\]
if and only if 
\[
\dim \Hom(M_{\pad(I)}, M_{\pad(J)}) +\dim \Hom(M_{\pad(J)}, M_{\pad(I)})
= |\pad(I)\cap \pad(J)| +n.
\]
Hence, by \Cref{Lem:rankoneext}, 
\[
 \ext^1(M_I,M_J)=0 \text{ if and only if } \ext^1(M_{\pad(I)}, M_{\pad(J)})=0.
\]
By \Cref{Eq:twoprojections},
\[
\projfun M_{\pad(I)} = \tilde{\projfun} N_{{\pad(I)}}  \mbox{ and } \projfun M_{\pad(J)} = \tilde{\projfun} N_{{\pad(J)}}.
\]
As both $\projfun$ and $\tilde{\projfun}$ are full exact (see \Cref{Thm:fullexact} and \cite[Cor. 4.6]{JKS1}),
we have,
\begin{eqnarray*}
\ext^1(M_{\pad(I)}, M_{\pad(J)}) 
&\cong & \Ext^1( {\projfun} M_{{\pad(I)}},  {\projfun} M_{{\pad(J)}})\\
&\cong & \Ext^1( \tilde{\projfun} N_{{\pad(I)}},  \tilde{\projfun} N_{{\pad(J)}})\\
&\cong & \Ext^1(N_{{\pad(I)}},  N_{{\pad(J)}})
\end{eqnarray*}
Therefore, by \Cref{Prop:WS1} and \Cref{Lem:padseparated},
  $\ext^1(M_I, M_J)=0$
if and only if $I$ and $J$ are weakly separated.
\end{proof}
\subsection{Quasi-commuting quantum minors}
Recall that {\em quantum matrix algebra} $\CC_q[M_{n\times n}]$ is  the $\Cq$-algebra generated by the variables $x_{ij}$, 
for $1\leq i, j \leq n$,  subject to the following homogeneous relations (cf. \cite{FRT})

\begin{equation}\label{eq:matqrel} 
x_{ij}x_{st}=\begin{cases}
qx_{st}x_{ij} & \text{if $i=s$ and $j<t$;}\\
qx_{st}x_{ij} & \text{if $i<s$ and $j=t$;}\\
x_{st}x_{ij} & \text{if $i<s$ and $j>t$;}\\
x_{st}x_{ij}+(q-q^{-1})x_{sj}x_{it} & \text{if $i<s$ and $j<t$,}
\end{cases} 
\end{equation}
and the quantum minor for $I=\{i_1, \dots, i_r\}$ is 
\begin{equation}\label{eq:qminor}
\qminor{I}=\sum_{\sigma} (-q)^{l(\sigma)} x_{1 i_{\sigma(1)}}\dots x_{1 i_{\sigma(r)}},
\end{equation}
where $l(\sigma)$ is the length of the permutation $\sigma$.

\begin{definition}\label{def:CIJ}
When $I$ and $J$ are weakly separated, let 
\[
  c(I,J) = \begin{cases} 
 |J''| - |J'| & \text{in Case (i)},\\
 |I'| - |I''| & \text{in Case (ii)},
\end{cases}
\]
where Cases (i) and (ii) are as in \Cref{def:non-crossing}.
\end{definition}

\begin{theorem}\label{Thm:LZ} \cite[Thm 1.1]{LZ}
Two quantum minors $\qminor{I}$ and $\qminor{J}$ quasi-commute if and only if $I$ and $J$ are weakly separated. 
In this case, the quasi-commutation rule is
\[
 \qminor{I}\qminor{J}=q^{c(I,J)}\qminor{J}\qminor{I}.
\]
\end{theorem}

\begin{remark}\label{rem:signprob}
Some authors use an alternative presentation of the quantum 
matrix algebra with $q$ and $q^{-1}$ swapped; see, for example, \cite{LZ}.
The statement of \Cref{Thm:LZ} differs by a sign from its original formulation 
in \cite{LZ}.
\end{remark}

Let $a\in I=\{i_1<\dots< i_k\}\subseteq [n]$, we denote 
$I\setminors \{a\}$ by $I_{\hat{a}}$.
Below we will give a categorical interpretation of the quasi-commutation rules. We start with two technical lemmas. 

\begin{lemma}\label{Lem:homInd} Let $I, J \subseteq [n]$ and 
$a\in I\cap J$. Then 
\begin{equation} \label{Eq:dimhom2}
\dim \Hom(M_I, M_J)=\dim \Hom(M_{I_{\hat{a}}}, M_{J_{\hat{a}}})+1.
\end{equation}
\end{lemma}
\begin{proof}
Let $I=\{i_1<\dots < i_k\}$, $J=\{j_1<\dots < j_l\}$,
 $a=i_x=j_y$ and 
 \[d=\dim \Hom(M_I, M_J). \]
We will show that 
\begin{equation}\label{Eq:ineq2}
S(I_{\hat{a}}, d-1)\preceq E(J_{\hat{a}}, d-1)
\end{equation}
and $d-1$ is maximal such that \eqref{Eq:ineq2} is true. 
Then the equality in \eqref{Eq:dimhom2} follows from \Cref{Lem:rankonehom}.
We have the following three cases to consider, depending on the relation between $x, y$ and $d$. 
 
\begin{itemize}\item[(i)]
 $d<x$ and $l-d+1\leq y \leq l$.

\item[(ii)] $x\leq d$ and $ l-d+1\leq y\leq l$.
 
\item[(iii)] $ x\leq d$ and $y\leq l-d$. 
\end{itemize}
\noindent
As $d$ is maximal such that $S(I, d)\preceq E(J, d)$, the situation that both 
$y\leq l-d$ and $x> d$ hold does not happen.
We first consider Case (i). 
We have
\[
S(I_{\hat{a}}, d-1 )=S(I, d-1) ~\text{ and }~ E(J_{\hat{a}}, d-1)=\{j_{l-d+1}, \dots, j_{y-1}, j_{y+1}, \dots, j_l \}.
\] 
Therefore, the inequality \eqref{Eq:ineq2} is indeed true. 

Next we show that $d-1$ is maximal as stated. Note by the maximality of $d$, there exists $1\leq t\leq d+1$
such that 
\begin{equation}\label{Eq:ineq3}
i_t>j_{l-(d+1) +t}.
\end{equation}
As we are in Case (i), 
\[i_{y-(l-d)+1}<\dots <i_{d+1}\leq i_x=a=j_y<j_{y+1}<\dots <j_l .\] 
Therefore,
\[t\leq y-(l-d).\]
By definition,  the $s$th element in $E(J_{\hat{a}}, d)$ is
\[
E(J_{\hat{a}}, d)_s=\left\{\begin{tabular}{ll} $j_{l-d+s-1}$ & if $s\leq y-(l-d)$; \\ 
$j_{l-d+s}$ & otherwise.
\end{tabular}\right.
\]
and $S(I_{\hat{a}}, d)= S(I, d)$.
Therefore, the $t$th elements in the sets $E(J_{\hat{a}}, d)$ and $S(I_{\hat{a}}, d)$ satisfy 
the following, 

\[
E(J_{\hat{a}}, d)_{t}= j_{l-d+t-1}< i_t=S(I_{\hat{a}}, d)_t,
\] 
which is exactly the inequality in \eqref{Eq:ineq3}.
Consequently, 
\[
S(I_{\hat{a}}, d)\not\preceq E(J_{\hat{a}}, d).
\]
So, by \Cref{Lem:rankonehom},
$d-1$ is indeed maximal and \eqref{Eq:dimhom2} holds for Case (i).

The proof for the other cases is similar. We skip the details.
\end{proof}

\begin{theorem}\label{Thm:quasicomm}
Suppose that $I$ and $J$ are two weakly separated subsets of $[n]$ with  $|I|\geq |J|$ and 
$J\setminus {I}=J'\cup J''$ such that $J'< (I\setminus J) <J''$. Then
\begin{itemize}
\item[(1)] $\dim\Hom(M_I, M_J)=|I\cap J|+ |J''|$. 
\item[(2)] $\dim\Hom(M_J, M_I)=|I\cap J|+ |J'|$.
\item[(3)] $\dim\Ext^1(M_I, M_J)=|J''|$. 
\item[(4)] $\dim\Ext^1(M_J, M_I)=|J'|$.
\end{itemize} 
Consequently,
\begin{eqnarray*}
c(I, J) &=&\dim\Hom(M_I, M_J)-\dim\Hom(M_J, M_I) \\
&=&\dim\Ext^1(M_I, M_J)-\dim\Ext^1(M_J, M_I). 
\end{eqnarray*}
\end{theorem}

\begin{proof}
We prove (1) and (2) by induction on $|I\cap J|$.
When $|I\cap J|=0$, that is, $I\cap J=\emptyset$, 
\[
J=J' \cup J'' \text{ and } J'< I< J''.
\]
So by \Cref{Lem:rankonehom}, 
\[
\dim \Hom(M_I, M_J)=|J''| \text{ and } \dim \Hom(M_J, M_I)=|J'|. 
\]

Now consider the general case where $I\cap J\not=\emptyset$. Assume that 
(1) and (2) are true for any weakly separated sets $K$ and $L$ with $|K\cap L|<|I\cap J|$.
Take $a\in I\cap J$. 
Note that $I_{\hat{a}}$ and $J_{\hat{a}}$ are weakly separated with 
$|I_{\hat{a}}|\geq |J_{\hat{a}}|$ and $|I_{\hat{a}}\cap J_{\hat{a}}|=|I\cap J|-1$. We have 
\[
I_{\hat{a}}\setminors J_{\hat{a}}=I\setminors J \text{ and }
J_{\hat{a}}\setminors I_{\hat{a}}=J\setminors I=J'\cup J''.
\]
So by induction, 
\[
\dim\Hom(M_{I_{\hat{a}}}, M_{J_{\hat{a}}})=|I_{\hat{a}} \cap J_{\hat{a}}| +|J''|=|I\cap J|-1 +|J''|
\]
and 
\[
\dim\Hom(M_{J_{\hat{a}}}, M_{I_{\hat{a}}})=|I_{\hat{a}} \cap J_{\hat{a}}| +|J'|=|I\cap J|-1 +|J'|.
\]
On the other hand by \Cref{Lem:homInd}, 
\[
\dim\Hom(M_{I}, M_{J})=\dim\Hom(M_{I_{\hat{a}}}, M_{J_{\hat{a}}})+1
\]
and 
\[
\dim\Hom(M_{J}, M_{I})=\dim\Hom(M_{J_{\hat{a}}}, M_{I_{\hat{a}}})+1.
\]
Therefore, we have (1) and (2), which together with \eqref{Eq:bilin0} imply  (3) and (4). 
Finally, the descriptions of $c(I, J)$ follows from \Cref{def:CIJ} and  (1)-(4).
\end{proof}
\subsection{Strong separation and $\Ext^1(-. -)$}\label{Sec:strongExt}
\begin{definition}\cite[\S 1]{LZ}
Let $I, J\subseteq [n]$. Then $I$ and $J$ are
\em{ strongly separated } if  either $J\minus I<I\minus J$ 
or $I\minus J <J\minus I$. 
\end{definition}

We remark that when one of $I$ and $J$ contains the other one, then they are 
strongly separated, and strong separation implies weak separation.

Recall the bar involution on the quantum matrix algebra 
\[\CC_q[M_{n\times n}]\to \CC_q[M_{n\times n}]\colon 
 q\mapsto q^{-1} \text{ and } x_{ij}\mapsto x_{ij}. 
 \]
By \cite[(2.16)]{LZ}, we have  
\[
{\color{blue}\overline{\qminor{I}\qminor{J}}=q^{-d(I, J)} \qminor{J}\qminor{I}},
\]
where
\[
d(I, J)= \min\{|I|, |J|\}-|I\cap J|.
\]
Again, note the discrepancy in signs compared with the original version 
(see remark \Cref{rem:signprob})

\begin{proposition} \cite[Cor.2.10]{LZ}
Let $I, J\subseteq [n]$. Then $I$ and $J$ are strongly separated if and only if 
one of the products $\qminor{I}\qminor{J}$ and  $\qminor{J}\qminor{I}$ is invariant under the bar involution.
\end{proposition}

\begin{theorem}\label{thm:strongseparation}
Let $I, J\subseteq [n]$ with $|I|\geq |J|$. Then $I$ and $J$ are strongly separated 
if and only if either $\Ext^1(M_J, M_I)=0$ or $\Ext^1(M_I, M_J)=0$. 
\end{theorem}
\begin{proof} Assume that 
$I$ and $J$ are strongly separated. Then they are weakly separated, and (we may assume that) 
$J\minus I=J'\cup J''$ such that
\[
J'<I\minus J<J''.
\]
Since $I$ and $J$ are strongly separated,  either $J'$ or $J''$ is empty.  Now \Cref{Thm:quasicomm} 
implies that either $\Ext^1(M_J, M_I)=0$ or $\Ext^1(M_I, M_J)=0$.

Conversely, assume that one of the extension groups is zero. Then 
by the symmetry of $\ext^1$,  we have 
\[
\ext^1(M_I, M_J)=0, 
\] 
which implies that $I$ and $J$ are weakly separated, following \Cref{thm:extWS}.
Now by \Cref{Thm:quasicomm} and the assumption that one of the two extension 
groups vanishes, either $J'$ or $J''$ is empty. Therefore, $I$ and $J$ are 
strongly separated. 
\end{proof}
\begin{corollary}\label{cor:barinv}
Let $I, J\subseteq [n]$ with $|I|\geq |J|$ be weakly separated. {\color{blue}Then }
\begin{itemize}
\item[(1)]  $\qminor{I}\qminor{J}$ is invariant under the bar involution if and only if 
$\Ext^1(M_I, M_J)=0$.
\item[(2)] $\qminor{J}\qminor{I}$ is invariant under the bar involution if and only if 
$\Ext^1(M_J, M_I)=0$.
\end{itemize}
\end{corollary}
\begin{proof}
By assumption, $\qminor{I}$ and $\qminor{J}$ quasi-commute. By \Cref{Thm:quasicomm} and by comparing $-d(I, J)$ and $c(I, J)$, the corollary follows.
\end{proof}
\section{Cluster tilting objects of rank one modules} \label{sec11}
We discuss when a maximal weakly separated collection defines a cluster tilting object.
When $\setJ$ is an interval, we extend rectangle clusters for $\CC[\Gr(k, n)]$ with $k\in \setJ$
to clusters (of minors) for $\CC[\flag(\setJ)]$. 

\subsection{Extending rectangle clusters.}
A set $I\subseteq [n]$ is called 
a {\em $\setJ$-set} if $|I|\in \setJ$, 
and a  {\em cointerval}
if it is the complement of an interval in $[n]$. 
Let $\maxNC_{k,\square}$ be the collection of interval $k$-sets, and 
$k$-sets that contain $1$ and are unions of two intervals.  
Note that  $\maxNC_{k,\square}$
is a collection of weakly separated sets of size  
$k(n-k)+1$, and so it is maximal by \cite[Thm. 3.5]{OPS} (see also \cite{JKS3}). 

The collection $\maxNC_{k,\square}$ determines a cluster in $\CC[\Gr(k, n)]$, 
referred to as a \emph{rectangle cluster} (\cite{RW}). The
$k$-sets in $\maxNC_{k,\square}$ except $[k]$ naturally label
the boxes in the $k\times (n-k)$-grid. More precisely, 
a box $b$ is the upper-right corner of the rectangle whose lower left corner is 
the lower corner of the $k\times (n-k)$ grid. 
The label assigned to box $b$ is given by the vertical steps in the lattice path going from the upper left corner to 
the lower right corner of the $k\times (n-k)$ grid that runs along the top and right boundary of the rectangle determined by  box $b$.
We place the $k$-set $[k]$ on the left of the grid in the first row. 

\begin{example} \label{ex:73J}Let $n=7$ and $k=3$. The collection $\maxNC_{k,\square}$ 
has the following $k$-sets.
\[
\begin{tikzpicture}[scale=1.0,
 arr/.style={black, -angle 60}, thick]
\foreach \i in {1,...,5}{ \path (\i, 4.15) node (w\i) {};}
\foreach \i in {1,...,5}{\path (\i, 0.85) node (v\i) {};}

\draw[blue] (1, 4) --(1, 3);
\draw[blue] (1, 3) --(3, 3);
\draw[blue] (3, 3) --(3, 1.1);
\draw[blue] (3, 1.1) --(5, 1.1);

\draw[red] (1, 4) --(2, 4);
\draw[red] (w2) --(v2);
\draw[red] (2, 1) --(5, 1);

\path (1.5, 3.5) node  {\color{red}$234$};
\path (2.5, 3.5) node  {$345$};
\path (3.5, 3.5) node  {$456$};
\path (4.5, 3.5) node  {$567$};

\path (1.5, 2.5) node  {$134$};
\path (2.5, 2.5) node  {\color{blue}$145$};
\path (3.5, 2.5) node  {$156 $};
\path (4.5, 2.5) node  {$167$};

\path (1.5, 1.5) node  {$124$};
\path (2.5, 1.5) node  {$125$};
\path (3.5, 1.5) node  {$126$};
\path (4.5, 1.5) node  {$127$};

\path (0.5, 1.5) node  {\color{magenta}$123$};
\end{tikzpicture}
\]
Here the red path corresponds to the box labelled by $234$, as
step numbers $2, 3, 4$ are the vertical steps. Similarly, the blue path corresponds 
to the box labelled by $145$. 
\end{example}
\begin{proposition} \label{prop:MWS} Let $k\in \setJ$. Then $\maxNC_{k, \square}$ can be uniquely extended 
to a maximal collection $\maxNC_{\setJ, k, \square}$ of weakly separated sets, by adding all the cointerval $t$-sets for $t\in \setJ$ with 
$ t<k$ and all the interval $t$-sets for $t\in \setJ$ with $t>k$.  
\end{proposition}
\begin{proof}
It is straightforward to see that $\maxNC_{\setJ, k, \square}$ is a collection of weakly separated sets. 
Let $I$ be a $\setJ$-set, and 
$I\not\in \maxNC_{\setJ, k, \square}$. 
We claim that $I$ is not weakly separated from a set in 
$\maxNC_{k, \square}$. 
Consequently, 
$\maxNC_{\setJ, k, \square}$ is the unique maximal 
collection of weakly separated sets containing $\maxNC_{k, \square}$.

We prove the claim. Let $t=|I|$. If $t=k$, 
clearly, the collection $I\cup \maxNC_{k, \square}$ is not weakly separated, 
as $\maxNC_{k, \square}$ is a maximal such collection. So it remains 
to consider the following two cases.
\begin{itemize}
\item[(i)]  $t< k$ and $I$ is not a cointerval. 
\item[(ii)] $t>k$ and $I$ is not an interval. 
\end{itemize} 
In Case (i),  there exist $a<i<b$ in $[n]$ such that 
\begin{equation}\label{eq:iab}
i\in I, \text{ but } a, b \not\in I.
\end{equation}
We first choose $i$ minimal, and for this $i$ choose $a$ and $b$ minimal, 
satisfying \eqref{eq:iab}.
Note that this implies that $[1, a-1]\subseteq I$ and $[i, b-1]\subseteq I$. 
Therefore 
\begin{equation} \label{eq:abi}
a< a+(b-i)\leq t<k.
\end{equation}
Let $L\in \maxNC_{k, \square}$ be the set defined by
\[
L=\left\{ \begin{array}{ll}  [1, a]\cup [ i+1, i+(k-a)], & \text{if } i+(k-a)\leq n  \\ \\   
  {} [1, k-(n-i)] \cup  [i+1, n],    & \text{otherwise}.
  \end{array}\right. 
\]
By construction and \eqref{eq:abi}, we see that $i\not\in L$ and 
$a,b\in L$. That is, we have 
\[ 
|I|=t<k=|L|, \;
a<i<b, \; a, b\in L\setminus I \text{ and } i \in I\setminus L.
\]
So $I$ and $L$ are not weakly separated. \\

Now consider Case (ii).   There exist $a<i<b$ in $[n]$ such that 
\begin{equation} \label{eq:iab2}
i\not\in I, \text{ but } a, b \in I.
\end{equation}
Note also in this case, 
\begin{equation} \label{eq:kt}
k<t\leq n.
\end{equation}
Choose $i$ minimal, and then for this $i$, choose $a$ and $b$ maximal, satisfying \eqref{eq:iab2}.
Then 
\begin{equation}\label{eq:ia}
i=a+1
\end{equation} 
 and so by \eqref{eq:kt}, 
\begin{equation}\label{eq:ika}   
i+(k-a)=k+1\leq n,
\end{equation} 
which implies 
\[
[ i, i+(k-a)]\subseteq [n].
\]
Let 
\[
L=\left\{ \begin{tabular}{ll}  
$[i, i+k-1]$, & if $a=1$;\\ \\
$[1, a-1]\cup [ i, i+(k-a)]$, & if  $1< a< k$;  \\ \\
$[1, k-1] \cup  [i, i]$,    &  \text{otherwise}.
\end{tabular}\right. 
\]
In either case, $L\in \maxNC_{k, \square}$ and 
\[
i\in L \text{ and } a \not\in L.
\]
We show that $b\not\in L$. When $a\geq k$, then $b\not\in L$ since $b>i$. 
Now assume that $a< k$. Since $b$ was chosen to be maximal, 
we have $[b+1, n]\cap I=\emptyset$ and so $I$ is a proper subset of 
$[1, b]$, since $i<b$ and $i\not\in I$.
Therefore 
\[  
k<|I|=t < b.
\]
Note  that $i=a+1$ from \eqref{eq:ia} and that $i+(k-a)=k+1<b$
from \eqref{eq:kt}.
Therefore,
\[
b\not\in L.
\] 
So we have $L\in \maxNC_{k, \square}$ such that   
\[ |I|=t > k=|L|, \;
a<i<b, \; a, b\in I\setminus L \text{ and } i \in L\setminus I.
\]
Therefore, $L$ and $I$ are not weakly separated. 
This completes the proof. 
\end{proof}

\subsection{Cluster tilting objects}\label{Sec:extendedrectangle}
Let $\maxNC$ be a maximal collection of weakly separated $\setJ$-sets and define
\[
T_{\maxNC}=\bigoplus_{I\in \maxNC} M_I,
\]   
This module
is rigid in $\Dfiltered(\setJ)$ since weak separation for sets translates to 
vanishing of $\ext^1(-, -)$ for modules, by \Cref{thm:extWS}.

\begin{theorem} \label{Thm:cardinality} 
There exists a cluster tilting object in $\Dfiltered(\setJ)$ of the form $T_{\maxNC}$ 
if and only if $\setJ$ is an interval. 
Moreover, in this case, every such module $T_{\maxNC}$ is a cluster tilting object in $\Dfiltered(\setJ)$. In particular, this applies when 
$\maxNC=\maxNC_{\setJ, k, \square}$ with $k\in \setJ$.
\end{theorem}
\begin{proof}
If $\setJ$ is not an interval,  by \Cref{Lem:charTsub} (2),  there is a rank two projective-injective object in
$\Dfiltered(\setJ)$.  Note that any indecomposable projective-injective object must be 
a summand of any cluster tilting object. Therefore, $\Dfiltered(\setJ)$ does not have
cluster tilting objects of the form $T_\maxNC$,  consisting solely of rank one modules.

Now assume that $\setJ=[r, s]$ is an interval.  Danilov--Karzanov--Koshevoy proved
that all the maximal collections of weakly separated $\setJ$-sets $\maxNC$ 
have the same size \cite[Thm. 1.1]{DKK1} (see also \cite{DKK2}),
\[
|\maxNC| = \left(\begin{matrix} n+1\\ 2 \end{matrix}\right)
-\left(\begin{matrix} n-s+1\\ 2 \end{matrix}\right)-
\left(\begin{matrix} r+1\\ 2 \end{matrix}\right)+1.
\]
This number is equal to $\dim\flag(\setJ) + |\setJ|$, by a 
straightforward computation. By \Cref{Thm:classiccat}, $\Dfiltered(\setJ)$ 
provides a categorification of the cluster algebra $\CC[\flag(\setJ)]$. 
Consequently, the number of indecomposable summands (up to
isomorphism) of a cluster tilting object equals  $\dim\flag(\setJ) + |\setJ|$. 
Hence $T_{\maxNC}$  is a maximal rigid object in 
$\Dfiltered(\setJ)$ and therefore a cluster tilting object, by \Cref{Lem:liftcto1}. 
This completes the proof.
\end{proof}

\begin{remark} 
 When $\setJ=\{k\}$, we call $T_{\maxNC_{k, \square}}$ a {\em rectangle cluster tilting object}, 
 and in the general case 
where $\setJ$ is an interval, we call the cluster tilting object 
$T_{\maxNC_{\setJ, k, \square}}$ an {\em extended rectangle cluster tilting object}. 
When $\setJ$ is not an interval, the number $|\maxNC_{\setJ, k, \square}|$ varies depending 
on $k$. See the example below.
\end{remark}
\begin{example}\label{ex:higherones}
Let $n=5$ and $\setJ=\{1, 3\}$, then $\maxNC_{\setJ,1, \square}$ has 
$8$ elements,
\[
  \{1\},  \{2\}, \{3\}, \{4\}, \{5\}, \{1,  2,  3\}, \{2, 3, 4\}, \{3, 4, 5\}
\] 
while $\maxNC_{\setJ,3, \square}$ has $9$ elements,
\[
\{1\},  \{5\}, \{1, 2, 3 \}, \{1, 2,  4\},   \{ 1, 2, 5\},  \{1, 3, 4\}, \{1, 4, 5\}, \{ 2, 3, 4\} , \{ 3, 4, 5\}.
\] 
Both are maximal collections of weakly separated $\setJ$-sets. 
The dimension of the flag variety
is $8$ in this example. So the number of summands in a
cluster tilting object is $8+|\setJ|=10$. The module $T_{\maxNC_{\setJ, k, \square}}$ in either case is
 not a cluster tilting object.  
\end{example}

When $\setJ$ is an interval, we arrange the sets in $\maxNC_{\setJ, k, \square}$
in a grid. We start with the rectangle containing the $k$-sets in $\maxNC_{k, \square}$ and then add interval $\setJ$-sets 
of larger size upward and cointerval $\setJ$-sets of smaller size rightward.
More precisely, the interval $[r, s]$ is placed 
above $[r+1, s]$ for $s-r\geq k$, and the cointerval $(t-1)$-sets are placed in the column on the right of cointerval $t$-sets with $[n-t+1, n]$ and $[n-t+2, n]$ in the same row for $t\leq k$.

\begin{example} \label{ex:grid} 
Let $n=7$, $k=3$ and $\setJ=[4]$. The sets in $\maxNC_{\setJ, k, \square}$
are arranged as follows, where the red coloured part is $\maxNC_{k, \square}$, 
the teal part consists of the intervals of size bigger than $k$, while the blue part consists of the cointervals of size smaller than $k$. 
\[
\begin{tikzpicture}[scale=1.0,
 arr/.style={black, -angle 60}, thick]
\path (1.5, 1) node {\color{red} 123};
\path (2.5, 1) node {\color{red} 124};
\path (3.5, 1) node {\color{red} 125};
\path (4.5, 1) node {\color{red} 126};
\path (5.5, 1) node {\color{red} 127};
\path (6.5, 1) node {\color{blue} 12};
\path (2.5, 1.8) node {\color{red} 134};
\path (3.5, 1.8) node {\color{red} 145};
\path (4.5, 1.8) node {\color{red} 156};
\path (5.5, 1.8) node {\color{red}  167};
\path (6.5, 1.8) node {\color{blue} 17};
\path (7.5, 1.8) node {\color{blue} 1};

\path (2.5, 2.6) node {\color{red} 234};
\path (3.5, 2.6) node {\color{red} 345};
\path (4.5, 2.6) node {\color{red} 456};
\path (5.5, 2.6) node {\color{red} 567};
\path (6.5, 2.6) node {\color{blue} 67};
\path (7.5, 2.6) node {\color{blue} 7};

\path (2.5, 3.4) node {\color{teal} 1234};
\path (3.5, 3.4) node {\color{teal} 2345};
\path (4.5, 3.4) node {\color{teal} 3456};
\path (5.5, 3.4) node {\color{teal} 4567};
\end{tikzpicture}
\]
\end{example}

\subsection{The Gabriel quiver for extended rectangle clusters}
We call an indecomposable summand $T_i$ of $T_{\maxNC_{\setJ, k, \square}}$ a 
\emph{level-$l$ summand} if the $\Delta$-support of $T_i$ is an $l$-set.
A  mutable summand of $T_{\maxNC_{k, \square}}$ is mutable for 
$T_{\maxNC_{\setJ, k, \square}}$  and we call such a summand a 
{\em mutable level-$k$} summand for   $T_{\maxNC_{\setJ, k, \square}}$. 
Note that some
frozen summands of $T_{\maxNC_{k, \square}}$
can be mutable for $T_{\maxNC_{\setJ, k, \square}}$.
The vertex corresponding to a summand $T_i$ is described in 
the same way, that is, \emph{mutable or frozen level-$l$ vertices}.

\begin{proposition} \label{Prop:gabriel}
Let $k\in \setJ=[r,s]$ and let $T=T_{\maxNC_{\setJ,k, \square}}$.  
Then the quiver $Q_T$ has no arrows between level-$k$-mutable vertices and vertices in the other levels.
\end{proposition}
Note that the module structure of rank one modules is explicit, and so are their homomorphisms. 
By carefully analysing the homomorphisms between the indecomposable summands 
of $T=T_{\maxNC_{\setJ,k, \square}}$, we have the proposition as claimed. We skip the details of the proof. 
In fact, it is not hard to compute the quiver $Q_T$ for   
$ T=T_{\maxNC_{\setJ,k, \square}}$. 
We give an example to illustrate.

\begin{example} Let $n=7$ and $\setJ=[4]$. 
Below is the quiver of 
$Q_{T_{\maxNC_{\setJ,3, \square}}}$.
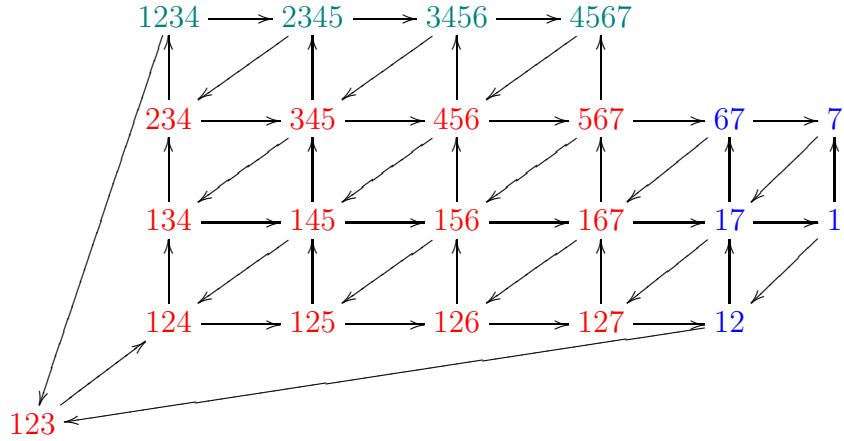
\[
\xymatrix{
&\color{teal} 1234 \ar[r]  \ar[ddddl] & \color{teal}2345  \ar[dl]  \ar[r] & \color{teal}3456 \ar[dl]   \ar[r] & \color{teal}4567 \ar[dl]  \\
 & {\color{red}234 \ar[u]}   \ar[r]    & {\color{red}345\ar[u]   \ar[r]} \ar[dl]& {\color{red}456\ar[u]  \ar[r]}  \ar[dl]&  {\color{red}567  \ar[u]  \ar[r]\ar[dl] } & \color{blue}67 \ar[dl]  \ar[r]   & \color{blue}7 \ar[dl]   \\
 & {\color{red} 134}\ar[u]  \ar[r]    & {\color{red} 145} \ar[u]   \ar[r] \ar[dl] & {\color{red}156} \ar[u]  \ar[r] \ar[dl]& {\color{red} 167}  \ar[u]  \ar[r] \ar[dl]& \color{blue}17 \ar[u]   \ar[r]  \ar[dl]&  \color{blue}1 \ar[u]\ar[dl] \\
  &{\color{red} 124}    \ar[u]\ar[r] & {\color{red} 125} \ar[r]  \ar[u]& {\color{red} 126} \ar[r] \ar[u] 
  & {\color{red} 127}  \ar[r]  \ar[u]&  \color{blue}12 \ar[u] \ar[dlllll]\\
{\color{red}123}  \ar[ur] & & & & &
}\]

\noindent The red part is the rectangle cluster $T_{\maxNC_{3, \square}}$, 
in which the mutable summands are 124, 125, 126, 134, 145, 156, but 
only 123 is not mutable in $T_{\maxNC_{\setJ, 3, \square}}$.
\end{example}

\section{Reachability of rank one modules}\label{sec12}
When $\setJ$ is an interval, 
we show that geometric exchanges and flips on collections of weakly separated sets 
induce mutations of cluster tilting objects. 
Consequently,  the extended rectangle cluster tilting objects 
$T_{\maxWS_{\setJ, k, \square}}\in \Dfiltered(\setJ)$ are shown to be mutation equivalent, and all the rank one modules $M_I$ with $|I|\in \setJ$ are {\em reachable}, that is,  the modules $M_I$ can 
be obtained from some $T_{\maxWS_{\setJ, k, \square}}$ by mutations.
\subsection{Geometric exchanges and mutations of cluster tilting objects}
 For $L\subseteq [n]$ and $a, b, c\in  [n]$, 
we denote by $La$, $Lab$ and $Labc$ the subsets $L\cup \{a\}$, $L\cup \{a, b\}$ and $L\cup \{a, b, c\}$, respectively. 
 Let $\maxNC$ be a collection of $\setJ$-sets and let 
\[\pad(\maxNC)=\{\pad(I)\colon I\in \maxNC\}.\]
We adapt the definition of geometric exchange for $k$-sets  \cite{Scott, JKS2}  to $\setJ$-sets. 
Suppose that $\maxNC$ is weakly separated and contains $Lab$, $Lbc$, $Lcd$, $Lad$, $Lac$, where $a, b, c, d$ are 
 four cyclically ordered numbers in $[n]\backslash L$. 
 A \emph{geometric exchange} replaces $Lac$ by $Lbd$ to obtain a new 
 collection $\maxNC'$ of $\setJ$-sets, that is, 
 \[
 \maxNC'=(\maxNC\backslash\{Lac\})\cup \{Lbd\}.
 \]
In the case of $k$-sets, geometric exchanges preserve weak separation \cite{Scott}. 
We will see that this is also true for $\setJ$-sets.
We remark that in the setting of $k$-sets, maximal collections of weakly separated sets are the face labels of plabic graphs, 
which are related by a sequence of square moves  \cite{OPS}. A square move changes the face labels by a geometric exchange. It would be interesting to investigate whether analogous 
connections extend to $\setJ$-sets. 

\begin{proposition}\label{prop:WSGeomEx}
Geometric exchanges preserve weak separation of $\setJ$-sets. 
\end{proposition} 
 \begin{proof} Let $L, a, b, c, d, \maxNC$ and $\maxNC'$ be as above. 
Note that, by \Cref{Lem:padseparated}, $\maxNC'$ is weakly separated if and only if 
$\pad(\maxNC')$ is  weakly separated. 
Recall that  \[\pad(Lxy)=Lxy\cup[n+l+1, 2n] \]
for all distinct $x, y\in \{a, b, c, d\}$, where $l=|L|+2$.  Extending 
$\pad(\maxNC)$ to a maximal weakly 
separated collection $\hat{\maxNC}$ of $n$-sets in $[2n]$ and applying the geometric exchange to $\hat{\maxNC}$ by replacing $\pad(Lac)$ by $\pad(Lbd)$, we obtain 
a maximal weakly separated collection of $n$-sets containing $\pad(\maxNC')$. In particular 
 $\pad(\maxNC')$ is weakly separated. Therefore, $\maxNC'$ is weakly separated as claimed. 
 \end{proof}

%
\begin{proposition}\label{prop:GeomExMutation}
Let $L, \{a, b, c, d\}$ be disjoint subsets of $[n]$. 
\begin{itemize}
\item[(1)] 
$\dim\ext^1(Lac, Lbd)=1$. 
\item[(2)] We have the following short exact sequences,
\begin{equation}\label{eq:mutseq1}
\begin{tikzpicture}[scale=0.3,
 arr/.style={black, -angle 60}, baseline=(bb.base)]
 \coordinate (bb) at (0, 11.8);
 
 \path (6,12) node(a1) {$ M_{Lac}   $};
 \path (15,12) node(a2) {$M_{Lad}\oplus M_{Lbc}$};
\path (24,12) node(a3) {$M_{Lbd}$};
 \path (0,12) node(a0) {$0$}; 
 \path (30,12) node(a4) {$0$,};
 \path[arr] (a0) edge  (a1);
 \path[arr] (a1) edge  (a2);
 \path[arr] (a2) edge  (a3);
 \path[arr] (a3) edge  (a4);
 \end{tikzpicture}
\end{equation}
\begin{equation}\label{eq:mutseq2}
\begin{tikzpicture}[scale=0.3,
 arr/.style={black, -angle 60}, baseline=(bb.base)]
 \coordinate (bb) at (0, 11.8);
 
 \path (6,12) node(a1) {$ M_{Lbd}  $};
 \path (15,12) node(a2) {$ M_{Lab}\oplus M_{Lcd}$};
\path (24,12) node(a3) {$M_{Lac}$};
 \path (0,12) node(a0) {$0$}; 
 \path (30,12) node(a4) {$0$.};
 \path[arr] (a0) edge  (a1);
 \path[arr] (a1) edge  (a2);
 \path[arr] (a2) edge  (a3);
 \path[arr] (a3) edge  (a4);
 \end{tikzpicture}
\end{equation}

\item[(3)] Let $\setJ$ be an interval. For any  cluster tilting object $T_\maxNC$ in $\Dfiltered(\setJ)$ that has $M_{Lab}$, $M_{Lbc}$, $M_{Lcd}$, $M_{Lad}$ and one of $M_{Lac}, M_{Lbd}$ as summands, the sequences in (2) are 
the mutation sequences for the pair $M_{Lac}$ and $M_{Lbd}$. 
\end{itemize} 
\end{proposition}
\begin{proof}
(1) Without loss of generality, we may assume that $a<b<c<d$.
By \Cref{Lem:homInd}, 
\[
\dim \Hom(M_{Lac}, M_{Lbd}) =|L|+\dim \Hom(M_{ac}, M_{bd})=|L|+2
\]
and similarly 
\[
\dim \Hom(M_{Lbd}, M_{Lac}) =|L|+1
\]
Therefore by \Cref{Lem:rankoneext}, 
\[
\dim \ext^1(M_{Lac}, M_{Lbd}) 
= (|L|+2)+(|L|+1)-|L|-(|L|+2)=1
\]

(2) Note that both $M_{Lad}$ and
$M_{Lbc}$  have a unique submodule isomorphic to $M_{Lac}$ and identifying these two submodules we obtain the quotient module $M_{Lbd}$ and so we have the exact sequence
 \eqref{eq:mutseq1}.
The proof for \eqref{eq:mutseq2} is similar.

(3) 
We may assume that $T_{\maxNC}$ has $M_{Lac}$ as a summand, that is, $\maxNC$ contains $Lac$. 
We need to prove that the maps in the sequences are $\add (T\backslash\{M_{Lac}\})$-approximations.
By \Cref{prop:WSGeomEx}, $\maxNC'=(\maxNC \backslash Lac)\cup \{Lbd\}$ is weakly separated and 
so by \Cref{thm:extWS},
$$
\ext^1(M_{L_{bd}}, T_\maxNC \backslash M_{Lac}) =\ext^1(T_\maxNC \backslash M_{Lac}, M_{L_{bd}})=0.
$$ 
Therefore, 
$$
M_{Lac} \to M_{Lad}\oplus M_{Lbc} \text{ and } M_{Lab}\oplus M_{Lcd}\to M_{Lac}
$$
are both $\add (T_\maxNC \backslash \{ M_{L ac} \})$-approximations, as required. 
\end{proof}

\begin{remark}
 If a mutation of a cluster tilting object is induced by a geometric exchange as in \Cref{prop:GeomExMutation}, we  
 say that it is a {\em mutation by geometric exchange}. 
\end{remark}

\subsection{Raising and lowering flips}
We recall the two types of flips from \cite{LZ}. 
Let $\WS$ be a collection of weakly separated $\setJ$-sets. Assume that 
$\WS$ contains $Lj$ and its four \emph{witnesses} $Li$, $Lk$, $Lij$ and $Ljk$, where $i<j<k$
are contained in $[n]\backslash L$. Then the operation on $\WS$ that
replaces $Lj$ by $Lik$ is called a \emph{weak raising flip}.
The inverse operation is called a \emph{weak lowering flip}. 
We say that two collections of weakly separated sets are \emph{flip equivalent} if they can be 
obtained from each other by weakly raising or lowering flips.

\begin{theorem} [\cite{LZ}, Theorem 5.1] \label{thm:LZ}
Weak raising and lowering flips preserve weak separations.
\end{theorem}

\begin{remark}
\Cref{thm:LZ} can be deduced from \Cref{prop:WSGeomEx}. 
Let $s=|L|$, $t=n+s+2$,  $L'=L\cup[t+1, 2n]$, and let $\WS'$ be obtained from 
$\WS$ by a flip replacing $Lj$ with  $Lik$ as defined above. 
Then $\pad(\WS)$ contains $L'it, L'jt, L'kt, L'ij, L'jk$, and $i<j<k<t$.
Now the flip on $\WS$ corresponds to the geometric exchange on $\pad(\WS)$ 
replacing $L'jt$ by $L'ik$. Hence by \Cref{prop:WSGeomEx}, $\WS'$ is weakly separated.
\end{remark}

By similar arguments as those used in the proof of \Cref{prop:GeomExMutation}, 
we have the following.

\begin{proposition}\label{prop:FlipMutation}
Let $L$ and $\{i< j< k\}$ be disjoint subsets of $[n]$. 
\begin{itemize}
\item[(1)] 
$\dim\ext^1(Lj, Lik)=1$. 
\item[(2)] We have the following short exact sequences,
\begin{equation}\label{eq:flipseq1}
\begin{tikzpicture}[scale=0.3,
 arr/.style={black, -angle 60}, baseline=(bb.base)]
 \coordinate (bb) at (0, 11.8);
 
 \path (6,12) node(a1) {$ M_{Lik}   $};
 \path (15,12) node(a2) {$M_{Li}\oplus M_{Ljk}$};
\path (24,12) node(a3) {$M_{Lj}$};
 \path (0,12) node(a0) {$0$}; 
 \path (30,12) node(a4) {$0$,};
 \path[arr] (a0) edge  (a1);
 \path[arr] (a1) edge  (a2);
 \path[arr] (a2) edge  (a3);
 \path[arr] (a3) edge  (a4);
 \end{tikzpicture}
\end{equation}
\begin{equation}\label{eq:flipseq2}
\begin{tikzpicture}[scale=0.3,
 arr/.style={black, -angle 60}, baseline=(bb.base)]
 \coordinate (bb) at (0, 11.8);
 
 \path (6,12) node(a1) {$ M_{Lj}  $};
 \path (15,12) node(a2) {$ M_{Lij}\oplus M_{Lk}$};
\path (24,12) node(a3) {$M_{Lik}$};
 \path (0,12) node(a0) {$0$}; 
 \path (30,12) node(a4) {$0$.};
 \path[arr] (a0) edge  (a1);
 \path[arr] (a1) edge  (a2);
 \path[arr] (a2) edge  (a3);
 \path[arr] (a3) edge  (a4);
 \end{tikzpicture}
\end{equation}

\item[(3)] Let $\setJ$ be an interval. For any cluster tilting object $T_\maxNC$ in $\Dfiltered(\setJ)$ that has $M_{Li}$, $M_{Lk}$, $M_{Lij}$, $M_{Ljk}$ and one of $M_{Lj}, M_{Lik}$ as summands, the sequences in (2) are 
the mutation sequences for the pair $M_{Lj}$ and $M_{Lik}$.
\end{itemize} 
\end{proposition}

\begin{remark}
 If a mutation of a cluster tilting object is induced by a flip, we  
 say that it is a {\em mutation by flip}. 
\end{remark}

Let $k\in \setJ=[r,s]$. Let $T'$ be a cluster tilting object in 
$\Dfiltered(\{k\})$ that is reachable from $T_{\maxNC_{k, \square}}$ and
\[T=T'\oplus (\bigoplus_I M_I)\]
where the second direct sum is over 
the $\setJ$-sets that are intervals of size bigger than $k$ or
cointervals of size smaller than $k$.

\begin{proposition}\label{Prop:reachability} The module $T$
is a cluster tilting object in $\Dfiltered(\setJ)$ that is reachable from 
$T_{\maxNC_{\setJ, k, \square}}$, via mutations by  geometric 
exchanges and flips. 
\end{proposition}
\begin{proof} 
By construction, the cluster tilting object $T_{\maxNC_{\setJ,k, \square}}$ is obtained from 
$T_{\maxNC_{k, \square}}$ by  adding the interval $\setJ$-sets of size bigger than $k$ and 
the cointerval  interval $\setJ$-sets of size smaller than $k$.
By \Cref{Prop:gabriel}, the quiver $Q_{T_{\maxNC_{\setJ,k, \square}}}$ has the property that
there are no arrows between mutable level-$k$ vertices and vertices in the other levels. So
mutations of $T_{\maxNC_{k, \square}}$ 
in $\Dfiltered(\{k\})$ are   mutations of 
$T_{\maxNC_{\setJ, k, \square}}$ in $\Dfiltered(\setJ)$. Moreover, 
such a mutation does not introduce new arrows between mutable level-$k$ vertices and vertices 
in the other levels. So the property of $Q_{T_{\maxNC_{\setJ,k, \square}}}$ is preserved under mutations 
 in $\Dfiltered(\{k\})$ on cluster tilting objects that are reachable from 
 $T_{\maxNC_{k, \square}}$.
 Following \cite[Thm. 1.4]{OPS}, $T'$ is reachable from 
 $T_{\maxNC_{k, \square}}$ by geometric exchanges, which by \Cref{prop:GeomExMutation} (3)
are also mutations in $\Dfiltered(\setJ)$. 
 Therefore, $T$ is a reachable cluster tilting object in $\Dfiltered(\setJ)$ as claimed. 
\end{proof}

In \Cref{Sec:extendedrectangle}, 
we introduced an arrangement of the $\setJ$-sets in
$\maxNC_{\setJ, k, \square}$.  
Denote the label at box in row $i$ and column $j$ by $I_{i, j}$, 
where row 1 is at the bottom and column 1 is on the left. We have 
$$I_{i, j}=[1, k-i] \cup[k-i+j+1, k+j].$$
Denote by $\row^k_i$ 
the set of $k$-sets in $\maxNC_{k, \square}$ that are  in row $i$. In particular, 
$\row^k_k$ consists of interval $k$-sets.

\begin{proposition} \label{Thm:flipequiv} Let $\setJ=[r, s]$ and $k-1, k\in \setJ$. 
The maximal collections of $\maxNC_{\setJ, k, \square}$ and $\maxNC_{\setJ, k-1, \square}$ are flip equivalent. 
Consequently, all such collections $\maxNC_{\setJ, l, \square}$ with $l\in \setJ$ are flip equivalent. 
\end{proposition}
\begin{proof} 
We will apply weakly lowering flips at each box in the first $k-1$ rows in the $k\times (n-k)$-grid,  
proceeding row by row from bottom to top, and within each row from right to left. 

We claim that after applying the flips at box $(i, j)$, 
\begin{equation}\label{eq:newlabel}
\text{the new set $I'_{i, j}$ at 
the $(i, j)$-box is  $[1, k-i-1]\cup [k-i+j, k+j-1]$,} 
\end{equation}
replacing  the original $k$-set $I_{i, j}$.
Consequently, after completing the flipping of the first $k-1$ rows, namely, 
all but the top row in the $k\times (n-k)$-grid, for any $1\leq s\leq k-1$, 
\begin{equation}\label{eq:rowi}
 \row^k_s \text{ is changed into } \row^{k-1}_s \minus \{[1, k-s-1]\cup[n-s+1, n]\},
\end{equation}
where however the $(k-1)$-set $[1, k-s-1]\cup[n-s+1, n]$ is a 
cointerval and is already an element in $S_{\setJ, k, \square}$. Therefore, 
$\maxNC_{k-1, \square}$ is contained in the new collection of $\setJ$-sets.
As the interval $l$-sets for $l\geq k$ and cointerval 
$t$-sets for $t\leq k-1$ remain unchanged, after flipping at the final box at 
$(k-1, 1)$, we obtain $\maxNC_{\setJ, k-1, \square}$.   
Therefore, after flipping all 
boxes in the first $(k-1)$ rows, we obtained $\maxNC_{\setJ, k-1, \square}$.

The claim can be proved by induction together with a careful examination of the witnesses. 
Note in particular that the new $(k-1)$-set $I'_{i, j}$ provides a witness for 
both $I_{i, j-1}$ when $j>1$ and $I_{i+1, j}$.
We skip the details of the proof and illustrate the main steps through 
the example below.
\end{proof}

\begin{example} We continue from  \Cref{ex:grid}.
The set $127$ has four witnesses, 126, 167, 17 and 12, while  $126$ only has 3 witnesses, $125$, $156$, $12$. 
The lowering flip replaces $127$ by $16$, providing the fourth witness for $126$. The grid arrangement of  $\maxNC_{\setJ, 3, \square}$,
after flipping at $127$, is as follows,   
\[
\begin{tikzpicture}[scale=1.0,
 arr/.style={black, -angle 60}, thick]
\path (1.5, 1) node {123};
\path (2.5, 1) node {124};
\path (3.5, 1) node { \color{blue} 125};
\path (4.5, 1) node {\color{red} 126};
\path (5.5, 1) node {\color{blue} 16};
\path (6.5, 1) node {\color{blue} 12};
\path (2.5, 1.8) node {134};
\path (3.5, 1.8) node {145};
\path (4.5, 1.8) node {\color{blue} 156};
\path (5.5, 1.8) node {167};
\path (6.5, 1.8) node {17};
\path (7.5, 1.8) node {1};

\path (2.5, 2.6) node {234};
\path (3.5, 2.6) node {345};
\path (4.5, 2.6) node {456};
\path (5.5, 2.6) node {567};
\path (6.5, 2.6) node {67};
\path (7.5, 2.6) node {7};

\path (2.5, 3.4) node { 1234};
\path (3.5, 3.4) node {2345};
\path (4.5, 3.4) node {3456};
\path (5.5, 3.4) node {4567};
\end{tikzpicture}
\]
where the four witnesses of $126$ are in blue. We flip at $126$ and continue  leftwards until 
$125$, $124$ are flipped. Then the new 
collection of weakly separated sets is, 
\[
\begin{tikzpicture}[scale=1.0,
 arr/.style={black, -angle 60}, thick]
\path (1.5, 1) node {123};
\path (2.5, 1) node {\color{teal} 13};
\path (3.5, 1) node {\color{teal} 14};
\path (4.5, 1) node {\color{teal} 15};
\path (5.5, 1) node {\color{teal} 16};
\path (6.5, 1) node { 12};
\path (2.5, 1.8) node {134};
\path (3.5, 1.8) node {145};
\path (4.5, 1.8) node {\color{blue} 156};
\path (5.5, 1.8) node {\color{red} 167};
\path (6.5, 1.8) node { 17};
\path (7.5, 1.8) node {1};

\path (2.5, 2.6) node {234};
\path (3.5, 2.6) node {345};
\path (4.5, 2.6) node {456};
\path (5.5, 2.6) node {\color{blue} 567};
\path (6.5, 2.6) node {\color{blue} 67};
\path (7.5, 2.6) node {7};

\path (2.5, 3.4) node { 1234};
\path (3.5, 3.4) node {2345};
\path (4.5, 3.4) node {3456};
\path (5.5, 3.4) node {4567};
\end{tikzpicture}
\]
Now, in row 2,  $167$ has four witnesses, the blue entries together with $16$. So we can flip at $167$ and continue leftward along row $2$ until the end. 
Then the new collection of weakly separated sets is as follows.
\[
\begin{tikzpicture}[scale=1.0,
 arr/.style={black, -angle 60}, thick]
\path (1.5, 1) node { \color{brown}  123};
\path (2.5, 1) node {\color{teal} 13};
\path (3.5, 1) node {\color{teal} 14};
\path (4.5, 1) node {\color{teal} 15};
\path (5.5, 1) node {\color{teal} 16};
\path (6.5, 1) node {\color{red} 12};
\path (2.5, 1.8) node {\color{teal}23};
\path (3.5, 1.8) node {\color{teal}34};
\path (4.5, 1.8) node {\color{teal} 45};
\path (5.5, 1.8) node {\color{teal} 56};
\path (6.5, 1.8) node {\color{blue} 17};
\path (7.5, 1.8) node {\color{blue} 1};

\path (2.5, 2.6) node {\color{magenta}234};
\path (3.5, 2.6) node {\color{magenta}345};
\path (4.5, 2.6) node {\color{magenta}456};
\path (5.5, 2.6) node {\color{magenta}567};
\path (6.5, 2.6) node {\color{blue}67};
\path (7.5, 2.6) node {\color{blue}7};

\path (2.5, 3.4) node { \color{magenta}1234};
\path (3.5, 3.4) node {\color{magenta}2345};
\path (4.5, 3.4) node {\color{magenta}3456};
\path (5.5, 3.4) node {\color{magenta}4567};
\end{tikzpicture}
\]
The $2$-sets form the collection $\maxNC_{2,\square}$. 
By rotating the boundary entries beyond the teal coloured part clockwise, we obtain the grid arrangement for $\maxNC_{\setJ,2, \square}$ as 
 described in  \Cref{Sec:extendedrectangle}.
\[
\begin{tikzpicture}[scale=1.0,
 arr/.style={black, -angle 60}, thick]
\path (1.5, 1) node {\color{red} 12};
\path (2.5, 1) node { \color{teal} 13};
\path (3.5, 1) node {\color{teal} 14};
\path (4.5, 1) node {\color{teal} 15};
\path (5.5, 1) node {\color{teal} 16};
\path (6.5, 1) node { \color{blue} 17};
\path (7.5, 1) node {  \color{blue} 1};
\path (2.5, 1.8) node {\color{teal}23};
\path (3.5, 1.8) node {\color{teal}34};
\path (4.5, 1.8) node {\color{teal}45};
\path (5.5, 1.8) node {\color{teal}56};
\path (6.5, 1.8) node {\color{blue}67};
\path (7.5, 1.8) node { \color{blue} 7};

\path (2.5, 2.6) node {\color{brown}123};
\path (3.5, 2.6) node {\color{magenta}234};
\path (4.5, 2.6) node {\color{magenta}345};
\path (5.5, 2.6) node {\color{magenta}456};
\path (6.5, 2.6) node {\color{magenta}567};

\path (3.5, 3.4) node {\color{magenta}1234};
\path (4.5, 3.4) node {\color{magenta}2345};
\path (5.5, 3.4) node {\color{magenta}3456};
\path (6.5, 3.4) node {\color{magenta}4567};
\end{tikzpicture}
\]
\end{example}

\subsection{Reachability of rank one modules for $\setJ=[r,s]$.}
%

For the purpose of categorifying the quantum cluster structure on 
flag varieties in \Cref{sec:16.2}, we will need to know that we can reach all the rank one
modules by mutations from a given cluster tilting object. 
A natural candidate for an initial cluster tilting object is 
$T_{\maxNC_{\setJ, k, \square}}$. 

\begin{theorem} \label{Thm:reachability}
Let $\setJ\subseteq [n-1]$ be an interval.
\begin{itemize}
\item[(1)] The  cluster tilting objects  $T_{\maxNC_{\setJ,k, \square}}$ 
and $T_{\maxNC_{\setJ,k',  \square}}$, for any $k, k'\in \setJ$,  are mutation equivalent 
through mutations by flips. 

\item[(2)] Any module $M_I$ with $I$ a $\setJ$-set can be extended to a reachable  
cluster tilting object of the form $T_{\maxWS}$ that can be obtained from some 
$T_{\maxNC_{\setJ,k, \square}}$ through mutations by geometric exchanges. 
Consequently, $M_I$ is reachable.
\end{itemize}
\end{theorem}
\begin{proof}
(1) follows from \Cref{prop:FlipMutation} and \Cref{Thm:flipequiv}.

(2) Let $k=|I|$. Extend $I$ to a maximal collection $\maxWS'$ of weakly separated 
$k$-sets. Now let
\[
\maxWS=\maxWS'\cup \{J\colon J \text{ is an interval of size}>k \text{ or a cointerval of size}<k\}.
\]
By \Cref{Prop:reachability}, 
$T_{\maxWS}$ is reachable from $T_{\maxNC_{\setJ, k, \square}}$, using mutations 
by geometric 
exchanges.
Therefore, $M_I$ is reachable as claimed.
\end{proof}

\begin{remark}
Let $\maxNC$ and $\maxNC'$ be two maximal collections of 
weakly separated $\setJ$-sets, where $\maxNC$ (resp. $\maxNC'$) contains a maximal collection 
of weakly separated $k$-sets (resp. $l$-sets). Assume that $\setJ$ is an interval. 
By \Cref{Prop:reachability} and \Cref{Thm:reachability},
$T_{\maxNC}$ and $T_{\maxNC'}$ are mutation equivalent. 
 It would be interesting to determine whether this remains true for 
arbitrary maximal collections of weakly separated $\setJ$-sets.
\end{remark}

\section{The quantum cluster algebras $\qcl{\mat B_T, \mat L_T}$ and 
$\qcl{\mat B_{\projfun T}, \mat L_{\projfun T}}$}\label{sec:13}
Let $T$ be a cluster tilting object in $\Dfiltered(\setJ)$. In this section, we construct 
a compatible pair of matrices $(\mat{B}_T, \mat{L}_T)$ and the quantum cluster algebra 
$\qcl{\mat B_T, \mat L_T}$ associated to $T$. 
 We then explain how $\mat L_T$ is related to the matrix $\mat{L}_{\pi{T}}$ constructed in 
\cite{GLS13}. Understanding the relationship between these two matrices will lead to a connection between 
 $\qcl{\mat B_T, \mat L_T}$ and the corresponding algebra 
 $\qcl{\mat B_{\projfun T}, \mat L_{\projfun T}}$ in \cite{GLS13},  and ultimately  
between the quantum coordinate rings $\pqflag$ and $\CC_q[N_\setK]$
(see \Cref{sec16}). 

\subsection{Construction of quantum seeds} 
\begin{definition} \cite[Def.~3.1]{BZ}
Let $\mat B$ be a $\cn\times \cm$ matrix and $\mat L$ a skew-symmetric 
$\cn\times \cn$ matrix. The pair of matrices $(\mat{B}, \mat{L})$ is \emph{compatible} if 
\begin{equation}
\mat{B}^{\mathrm{tr}}\mat{L} = \begin{pmatrix} \mat{D} & \mat{0}\end{pmatrix}
\end{equation}
where $\mat{D}$ is a diagonal matrix with positive entries.
\end{definition}

Note that, if a pair $(\mat{B}, \mat{L})$ is compatible, then $\mat{B}$ 
has full rank $\cm$ and its principal part is skew-symmetrisable 
\cite[Prop.~3.3]{BZ}. 

Given a compatible pair $(\mat{B}, \mat{L})$ and $k$ in the 
mutable range $1\leq k\leq \cm$, let
\[\mu_k(\mat{B}, \mat{L})=(\mu_k\mat{B}, \mu_k\mat{L}),\]
be the mutation of $(\mat{B}, \mat{L})$ in the $k$-direction, which 
is again compatible \cite[Prop.~3.4]{BZ}.

Denote the indecomposable summands of $T$ by $T_i$. 
Let 
$\mat B=\mat{B}_T$ be the matrix defined in \eqref{eq:BfromT}. 
We define the matrix  $\mat{L}=\mat{L}_T=(\lambda_{ij})$ by 
\[
\lambda_{ij}=\dim \Hom(T_i, T_j)-\dim\Hom(T_j, T_i).
\]

The following result follows from similar arguments as the proof of  
\cite[Prop. 10.1 and 10.2]{GLS13} (see also \cite[Thm. 6.3]{JKS2}) and we skip the details. 
We remark that this can be also put into the setting of extriangulated categories, and then 
the quantum cluster structure also follows from Grabowski-Pressland's construction \cite{GP}.

\begin{theorem} \label{prop:main}
The pair $(\mat{B}_T, \mat{L}_T)$ is compatible with 
$\mat{D} = 2\mat{Id}$. Furthermore,  
\[
\mu_k(\mat{B}_T, \mat{L}_T)=(\mat{B}_{\mu_k T}, \mat{L}_{\mu_k T}).
\]
\end{theorem}

\begin{remark}\label{rem:13.3}
By \Cref{Thm:quasicomm}, when $T=\oplus M_{I_i}$, then 
\[
\lambda_{ij}=c(I_i, I_j)
\]
and so $\mat{L}_T$ encodes the quasi-commutation rules for the quantum minors $\qminor{I_i}$.
\end{remark}

\subsection{Properties of $\mat{L}$}
The matrix $\mat L=\mat L_T$ induces  
a skew-symmetric bilinear form on $K(\add T)$ 
such that for any $X,Y\in \add T$, 
\begin{equation}\label{eq:Lbil}
\mat{L}( [X], [Y]) =\dim \Hom(X,Y) -\dim \Hom(Y,X).
\end{equation}

We explain how a module $M$ in $\Dfiltered(\setJ)$ or $\add \pi T$ determines a class 
in $K(\add T)$. Note that \[\add \pi T\subseteq \mod \Pi\subseteq \mod D.\]
We say that a $\aus$-module 
$M$ has an \emph{$\add T$-presentation} if there is a short exact sequence as follows,
\begin{equation}\label{eq:Tresol}
\begin{tikzpicture}[scale=0.6,
 arr/.style={black, -angle 60}, baseline=(bb.base)]
\coordinate (bb) at (0, 0.9);
 
\path (9,1) node (c3) {$M$};
\path (6,1) node (c2) {$T_M'$};
\path (3,1) node (c1) {$T_M''$};
\path (12, 1) node (c4) {$0$,};
\path (-1,1) node  {};
\path (0, 1) node (c0) {$0$};

\path[arr] (c1) edge node[auto]{} (c2);
\path[arr] (c2) edge node{}  (c3);
\path[arr] (c0) edge (c1);
\path[arr] (c3) edge (c4);
\end{tikzpicture}
\end{equation}
where $T_M'$, $T_M''\in \add T$. Such a presentation exists for any $M\in \Dfiltered(\setJ)$ 
(see \Cref{Lem:resHom}) and for $M\in \add \pi T$. For $\pi T$, we use the sequence, 
\[
\begin{tikzpicture}[scale=0.6,
 arr/.style={black, -angle 60}]
\path (9,1) node (c3) {$\projfun T$};
\path (6,1) node (c2) {$T$};
\path (3,1) node (c1) {$\eta T$};
\path (12, 1) node (c4) {$0$.};
\path (-1,1) node  {};
\path (0, 1) node (c0) {$0$};

\path[arr] (c1) edge node[auto]{} (c2);
\path[arr] (c2) edge node{}  (c3);
\path[arr] (c0) edge (c1);
\path[arr] (c3) edge (c4);
\end{tikzpicture}
\]

For $M\in \mod D$ with an $\add T$-presentation as in  \eqref{eq:Tresol},  
we can define the class of $M$ in $K(\add T)$ as 
\begin{equation}\label{eq:classinaddT}
[M]=[T_M']-[T_M''].
\end{equation}
This is well-defined, because minimal approximation exists and is unique (see \cite[\S I.2]{ARS}).

Note that for objects $X$  in $\add T$ and its class $[X]\in K(\add T)$
are uniquely determined by each other. 
Define a bilinear form $\matadj$ on $K(\add T)$ by 
$$
\matadj( [X], [Y]) = \dim\Hom(X,\eta Y) - \dim\Hom(Y, \eta X),
$$
which by definition is skew-symmetric. 
We will show that $\matadj$ quantifies the discrepancy of the quasi-commutation rules for the 
corresponding cluster variables in 
$\CC[N_\setK]$ and $\pqflag$ (see \Cref{rem:rellattice} for further details).

\begin{proposition}\label{prop:Lad}
The form $\matadj(-, -)$ induces a skew-symmetric bilinear form on 
$K(\Dfiltered(\setJ))_{\mathcal{E}}$ as follows, 
 \begin{equation}\label{eq:lad}
\matadj( [X], [Y]) = \langle  [X], [\eta Y]\rangle - \langle [Y], [\eta X]\rangle, 
~\forall X, Y\in \Dfiltered(\setJ).
\end{equation}
\end{proposition}
\begin{proof}
First, by definition and  \eqref{eq:bilhom}, we know that \eqref{eq:lad} is true for any $X, Y\in \add T$. 
For the general case, consider $\add T$-approximations for $X, Y$ as in 
\eqref{eq:Tresol}. We have 
\[ [X]=[T_X']-[T_X''], ~  [Y]=[T_Y']-[T_Y'']\]
and 
\[
\matadj( [X], [Y]) = \matadj( [T'_X]-[T_X''], [T_Y']-[T_Y'']).
\]
Let $T'\in \add T$ and $Z\in \add \charT_\setJ$.
Applying $\Hom(T', -)$ to the trace sequence of the $\add T$-presentation sequence for $X$ (which 
is admissible), we have 
\[\dim\Hom(T', \eta X)=\dim\Hom(T', \eta T'_X)-\dim\Hom(T', \eta T''_X).\] 
Applying $\Hom(-, Z)$ to the approximation sequence for $X$  gives, 
\[
\dim\Hom(X,  Z)=\dim\Hom(T_X', Z)-\dim\Hom(T_X'', Z),
\]
where we also use $\Ext^1(X, Z)=0$ by \Cref{Lem:tiltingchar}.  
The same is true when $X$ is replaced by $Y$.
The desired equality \eqref{eq:lad} now follows from the bilinearity of $\matadj$ and $\langle-, -\rangle$, 
together with  \eqref{eq:bilhom}. 
\end{proof}

\begin{proposition} \label{lem:qrules1}
Let  $X,Y\in \add T$.
\begin{itemize}
\item[(1)] $\mat{L}([\projfun X], [\projfun Y]) = \dim \Hom(\projfun X, \projfun Y) - \dim \Hom(\projfun Y, \projfun X)$.

\item[(2)] If $X,Y \in \add \charT_{\setJ}$, then $\mat{L}( [X], [Y])=0$.
\end{itemize}
\end{proposition}
\begin{proof}
(1) First note that 
\begin{equation}\label{eq:tem}
\dim \Hom(\eta X, Y) = \dim \Hom(\eta X, \eta Y), 
\end{equation}
since $\eta$ computes the trace of $\charT$. By definition, 
\begin{equation}
\mat{L}([\projfun X], [\projfun Y]) = \mat{L}([X]-[\eta X], [Y]-[\eta Y]). 
\end{equation}
Now (1) follows from the bilinearity of $\mat{L}(-, -)$,  \Cref{lem:dimlemma} and \eqref{eq:tem}.

(2) Note that for any $\charT_i, \charT_j$ with $i\leq j$, 
\[
\dim \Hom(\charT_i, \charT_j)=\dim \Hom(\charT_j, \charT_i)=i.
\]
Therefore, for any $X, Y\in \add \charT_\setJ$,
\[\dim \Hom(X,Y) = \dim \Hom(Y,X)\] and so $\mat{L}( [X], [Y])=0$.
\end{proof}

\begin{corollary}
The bilinear form $\mat L_{\pi T}$ constructed in \cite{GLS13} is the restriction of 
$\mat{L}$ to $K(\add \projfun T)$.
\end{corollary}
\begin{proof}
Note that any object in $ \add \pi T$ is of the form $\pi X$ for some $X\in \add T$.
So the  bilinear form $\mat{L}_{\pi T}$ in \cite{GLS13} is 
\[
\mat{L}_{\pi T}([\pi X], [\pi Y])=\dim \Hom(\projfun X, \projfun Y) - \dim \Hom(\projfun Y, \projfun X), 
\]
for any $\pi X, \pi Y\in \add \pi T$. 
Now the corollary follows from \Cref{lem:qrules1} (1).
\end{proof}

\begin{proposition} \label{lem:qrules}
Let  $X,Y\in \Dfiltered(\setJ)$.
\begin{itemize}
\item[(1)] $\mat{L}( [X], [Y]) - \mat{L}([\projfun X],[\projfun Y]) = \matadj( [X], [Y]).$

\item[(2)] If $Y\in \add \charT_{\setJ}$, then $\mat{L}( [X], [Y]) = 
\matadj( [X], [Y])=\mat{L}([\projfun X], [Y])$.
\end{itemize}
\end{proposition}
\begin{proof}
(1) Note that $[\pi X]=[X]-[\eta X]$ and $[\pi Y]=[Y]-[\eta Y]$. So (1) follows from the 
bilinearity of the forms, \Cref{prop:Lad} and \eqref{eq:bilhom}.

(2) When $Y\in \add\charT_{\setJ}$, $\projfun Y=0$ and so the first equality follows from (1).
For the second equality, we have 
\[
\mat{L}( [\projfun X], [Y]) = \mat{L}( [X], [Y]) - \mat{L}( [\eta X], [Y])=\mat{L}( [X], [Y]), 
\]
by the linearity of $\mat{L}(-, -)$ and \Cref{lem:qrules1}.
\end{proof}

\begin{remark}\label{rem:rellattice} 
The  forms $\mat L_T$ and $\mat L_{\pi T}$ encode   
the quasi-commutation rules for cluster variables in the  clusters (see \Cref{subsec:qclalg})
corresponding to $T$ and $\pi T$, respectively.
Although both forms change under the mutation of $T$, their difference remains constant, 
following \Cref{lem:qrules} (1).
That is, for   $X, Y\in \Dfiltered(\setJ)$, 
\[
\mat L_T([X], [Y])- \mat L_{\pi T} ([\pi X], [\pi Y])=\mat L_{T'}([X], [Y])- \mat L_{\pi T'} ([\pi X], [\pi Y]),
\]
where $T$ and $T'$ are mutation equivalent.
\end{remark}

\begin{remark}\label{rem:inducedbil}
By the equivalence between $\add T$ and $\proj A$ in \eqref{eq:addTprojA}, 
the isomorphism in   \Cref{Lem:Grothiso}
induces an isomorphism, which we will also denote by $\nu$,
\begin{equation}\label{eq:xiGrothAddT}
\nu\colon K(\add T)\to K(\add \projfun T) \oplus K(\add \charT_{\setJ}), ~
[X] \mapsto [\projfun X]+ [\eta  X].
\end{equation}
Note that both $\projfun T$ and $\charT_\setJ$ are $D$-modules. It
makes sense to take the direct sum  $\projfun T\oplus\charT_\setJ$, 
and we have
\[ 
K(\add \projfun T) \oplus K(\add \charT_{\setJ})=K(\add \projfun T\oplus \charT_\setJ).
\] So 
$\nu$ identifies $K(\add T)$ with $K(\add \projfun T \oplus \charT_{\setJ})$, and 
 $\mat{L}$ induces a bilinear form on the latter. 
\end{remark}

\subsection{Quantum tori}
Let $\CC_q[K(\add T)]$ be the quantum torus (see \cite{BZ})
associated with the Grothendieck group $K(\add T)$. That is, $\CC_q[K(\add T)]$ is a 
$\CC[q^{\frac{1}{2}},q^{-\frac{1}{2}}]$-algebra with basis
$\{z^{h}\colon h\in K(\add T)\}$ 
and multiplication defined by 
\begin{equation}
z^{g}\cdot z^{h}=q^{\frac{1}{2}\mat L(g,  h)} z^{g+ h}
\end{equation}
for any $g, h \in K(\add T)$, where $\mat L=\mat L_T$. 
Consequently, 
\begin{equation}
z^{g} \cdot z^{h} = q^{\mat L(g,h)} z^{h} \cdot z^{g}
\end{equation}
and 
\begin{equation} \label{eq:qsum} 
z^{\sum_i a_i[T_i]} = q^{\frac{1}{2}\sum_{i>j} a_ia_j\mat{L}(T_i, T_j)} \prod (z^{[T_i]})^{a_i}.
\end{equation}
The algebra $\CC_q[K(\add \projfun T \oplus \charT_{\setJ})]$  is similarly defined, 
using the induced bilinear form in \Cref{rem:inducedbil}.

\begin{proposition} \label{eq:toriisom} 
The two quantum tori $ \CC_q[K(\add T)]$ and  $\CC_q[K(\add \projfun T \oplus \charT_{\setJ})]$ 
are $\grothEJ$-graded. Moreover, 
the following map is a $\grothEJ$-graded isomorphism, 
\begin{equation}\label{eq:xi} 
\qnu\colon \CC_q[K(\add T)]\rightarrow \CC_q[K(\add \projfun T \oplus \charT_{\setJ})], ~
z^{[V]}\mapsto z^{[\projfun V] + [\eta  V]}.\end{equation}
\end{proposition}

\begin{proof}
Note that $\add T\subseteq \Dfiltered(\setJ)$. So any $X\in \add T$ naturally determines a 
class $[X]\in \grothEJ$. Now, using the isomorphism 
$\nu\colon K(\add T)\to K(\add \projfun T \oplus \charT_{\setJ})$ in \eqref{eq:xiGrothAddT}, 
any class in the latter naturally corresponds to a class in $\grothEJ$.
Hence, 
the quantum tori $ \CC_q[K(\add T)]$ and  $\CC_q[K(\add \projfun T \oplus \charT_{\setJ})]$ 
are $\grothEJ$-graded. 

Next, note that the quasi-commutation rules of the generators of $\CC_q[K(\add \projfun T \oplus \charT_{\setJ})]$
are defined by the bilinear form on $K(\add \projfun T \oplus \charT_{\setJ})$, which is induced from the bilinear form on $K(\add T)$ via the isomorphism $\nu$. So the map $\qnu$ naturally preserves the grading and is thus a $\grothEJ$-graded isomorphism.
\end{proof}

\subsection{Quantum cluster algebras}\label{subsec:qclalg}
The compatible pair $(\mat B_T, \mat L_T)$  defines a
quantum cluster algebra $\qclusalgT$ \cite{BZ}. Moreover, 
by the quantum Laurent phenomenon \cite[Cor. 5.2]{BZ}, 
$$
\qclusalgT\subseteq \CC_q[K(\add T)].
$$ 
As $(\mat{B}_T, \mat{L}_T)$ is compatible 
with mutations (see \Cref{prop:main}), by \cite[Prop. 4.9]{BZ} and 
following the construction in \cite{GLS13}, 
for every reachable cluster tilting object $U$ and for $V\in \add U$, there
is a corresponding element $\Upsilon^U([V]) \in \CC_q[K(\add T)]$.
In fact, we define the map
\[
\Upsilon^U\colon ~ K(\add U)\to \CC_q[K(\add T)],
\]
by induction. 
First, define $\Upsilon^T(u) = z^{u}$ for any $u\in K(\add T)$. 
Now assume
that $\Upsilon^U$ has been defined and let $U_k$ be a mutable summand of $U$. 
Then let 
\begin{equation}\label{eq:13.10} 
\Upsilon^{\mu_kU}([U_j])
=\Upsilon^U([U_j]) \text{ for } j\neq k,
\end{equation}
and 
\begin{equation} \label{eq:qexchange}
\Upsilon^{\mu_kU}([\mu_kU_k]) %
=\Upsilon^{U}([E_k]-[U_k]) + \Upsilon^U([F_k]-[U_k]),
\end{equation}
where $E_k$ and $F_k$ are the middle terms of the two mutation sequences of $U$ 
in the direction $k$ (see \Cref{Lem:liftexchange}), 
\[
\Upsilon^{U}([E_k]-[U_k]) =q^{\frac{1}{2} \mat{L}_U([E_k], [U_k])}\Upsilon^{U}([E_k]) (\Upsilon^U([U_k]))^{-1}
\]
and
\[
\Upsilon^U([F_k]-[U_k]) =q^{\frac{1}{2} \mat{L}_U([F_k], [U_k])}\Upsilon^{U}([F_k]) (\Upsilon^U([U_k]))^{-1}.
\]
In general, for  $u=\sum_i a_i[U_i]\in K(\add U)$, 
\[\Upsilon^{U}(u) = q^{\frac{1}{2}\sum_{i>j}a_ia_j\mat{L}_U([U_i], [U_j])}\prod_i\Upsilon^{U}([U_i])^{a_i}.\] 
The Laurent polynomials $\Upsilon^U([U_i])$ for all reachable cluster 
tilting objects $U$ and indecomposable summands $U_i$ of $U$ are the 
quantum cluster variables in $\qclusalgT$, by construction.

\begin{remark}\label{rem:getridofU}
Let  $V$ be another cluster tilting object reachable from $T$ with $U_i$ as a summand.
Both $\Upsilon^{U}([U_i])$ and $\Upsilon^{V}([U_i])$ are quantum cluster variables,
and are therefore identical, by \cite[Thm. 6.1]{BZ}. So we can \emph{denote them by $
\Upsilon([U_i])$, unless we need to emphasise the cluster tilting objects}. 
\end{remark}

\begin{lemma}\label{lem:grading}
We have a decomposition 
$$
\qclusalgT = \bigoplus_{w\in \grothEJ} \qclusalgT_{w},
$$
which gives $\qclusalgT$ the structure of a $\grothEJ$-graded algebra. 
In particular, under this grading, $\Upsilon(V)$ has degree $[V]$. 
\end{lemma}
\begin{proof} 
First note that for any $V\in \add T$, $z^{[V]}\in \qclusalgT_{[V]} $, 
where the subscript $[V]$ is viewed as a class in $\grothEJ$. 
Now by induction, the  mutated variable \eqref{eq:qexchange} is 
homogeneous, since \[ [E_k]-[U_k]=[F_k]-[U_k]=[U^*_k]\]
in $\grothEJ$, 
and so the lemma follows. 
\end{proof}

\begin{remark}
Any cluster tilting object in 
$\Dfiltered(\setJ)$ has $\charT_\setJ$ as a summand (see \Cref{Lem:liftcto}), and  $z^{[\charT_j]}=\Upsilon(\charT_j)$ is 
a frozen cluster variable for any $j\in \setJ$.
\end{remark}

\begin{lemma} The cluster variables $z^{[\charT_j]}$ for $j\in \setJ$  quasi-commute with homogeneous elements $x \in \qclusalgT_w$, 
\begin{equation}\label{eq:quasicomm}
z^{[\charT_j]}\cdot x=q^{\matadj(w, [\charT_j])}x\cdot z^{[\charT_j]}.
\end{equation}
Consequently, the cluster variables $z^{[\charT_j]}$ commute with each other, and
 $\qclusalgT$ can be localised by inverting $z^{[\charT_j]}$, $j\in \setJ$.
\end{lemma}

\begin{proof}
Note that $\qclusalgT$ is generated by cluster variables and $\grothE$-graded by \Cref{lem:grading}. 
Furthermore, by \Cref{lem:qrules} (2), 
\[
\mat{L}(-, [\charT_j])=\matadj(-, [\charT_j])
\]
and by \Cref{prop:Lad}, $\matadj(-, -)$ is a bilinear form on $\grothE$.
Therefore, \eqref{eq:quasicomm} follows. Furthermore, by \Cref{lem:qrules1} (2), any two 
cluster variables $z^{[\charT_i]}$, $z^{[\charT_j]}$ commute.
As a consequence of \eqref{eq:quasicomm},  the localisation by inverting $z^{[\charT_{\setJ}]}$ exists.
\end{proof}

Similar to the construction of $\qclusalgT$, define inductively
$\upsilon^{\projfun U}([V])$ for a cluster tilting object $\projfun U$ reachable from $\projfun T$ and $V\in \add \projfun U$
(see \cite{GLS13} for more details), 
and obtain a quantum cluster algebra, 
$$
\qcl{\mat{B}_{\projfun T}, \mat{L}_{\projfun T}}\subseteq \CC_q[K(\add \projfun T)].
$$ 
Following  \Cref{rem:getridofU}, we can denote $\upsilon^{\projfun U}([V])$  by $\upsilon([V])$ 
for any $V\in \add \projfun U$.

Note that $U$ is reachable from $T$ if and only if $\projfun U$ is reachable from $\projfun T$ 
(see  \Cref{rem:reach}).
The algebra $\qcl{\mat{B}_{\projfun T}, \mat{L}_{\projfun T}}$ is $K(\Pi)$-graded 
(cf. the proof of \Cref{lem:grading}) with 
\[\deg \upsilon([V])=\dvector  V,\] 
for any reachable rigid module $V\in \sub Q_{\setJ}$.

We know that $\matadj$ defines a bilinear form on $\grothEJ$, which 
can be identified with $K(\Pi)\oplus K(\add \charT_\setJ)$ (\Cref{Lem:isom1}). So we can define 
\[
\qcl{\mat{B}_{\projfun T}, \mat{L}_{\projfun T}}[K(\add \charT_{\setJ})]
\subseteq \CC_q[K(\add \pi T)][K(\add \charT_\setJ)]
\] 
to be  the {\em twisted}
formal Laurent polynomials with the 
 quasi-commutation rules, 
\begin{equation}\label{eq:subCharT1}
x\cdot z^{u} = q^{\matadj(d, u)}  z^{u}\cdot x
\end{equation}
 and
\begin{equation} \label{eq:subCharT2}
z^{u}\cdot z^{w}=z^{u+w}=z^{w}\cdot z^{u}, 
\end{equation}
where 
$x$ is homogeneous of degree $d\in K(\Pi)$,
and $u, w\in K(\add \charT_{\setJ})$. 

We extend $\upsilon^{\projfun U}$ to $K(\add \projfun U \oplus \charT_{\setJ})$,
\begin{equation}\label{eq:upsilonexp}
\upsilon^{\projfun U}(u +  t) = q^{\frac{1}{2}\matadj(t, u)} \upsilon^{\projfun U}(u)\cdot z^{t},
\end{equation}
for any homogeneous $u\in K(\add \projfun U)$ and $t\in  K(\add \charT_{\setJ})$. 

Note that $\projfun$ and $\eta$ are exact functors and so they induce maps between Grothendieck groups. More precisely, 
for any $v=\sum_ia_i[T_i]\in K(\add T)$,  
\[\projfun(v)=\sum_ia_i [\projfun T_i] ~\text{ and } ~  \eta(v)= \sum_{i} a_i [\eta T_i].\]

Recall the isomorphism 
$\nu\colon K(\add T)\to K(\add \projfun T \oplus \charT_{\setJ})$ 
and that the bilinear form $\mat L$ induces a bilinear form 
on $K(\add \projfun T\oplus \charT_{\setJ})$
(see \Cref{rem:inducedbil}). 
By  \Cref{lem:qrules1}, the induced bilinear form restricted to $K(\add \projfun T)$ 
is consistent with $\mat L_{\projfun T}$.
For any $v\in K(\add T)$ and $t\in K(\add \charT_{\setJ})$, by \Cref{lem:qrules},
\begin{equation}\label{eq:exponents}
\mat L_{T}(t, \projfun (v))=\mat L_{T}(t,  v)=\matadj (t, v).
\end{equation}
Hence,  
\[
\CC_q[K(\add \pi T)][K(\add \charT_\setJ)]= 
\CC_q[K(\add \pi T \oplus \add \charT_\setJ)].
\]

\begin{theorem} \label{thm:qmain1}
The map $\qnu$ \eqref{eq:xi}
restricts to a $\grothEJ$-graded injection,
$$
\qnu\colon \qclusalgT\rightarrow 
\qcl{\mat B_{\projfun T}, \mat L_{\projfun T}}[K(\add \charT_{\setJ})], ~
\Upsilon(M) \mapsto q^{\frac{1}{2}\matadj([\projfun M], [\eta M])} \upsilon([\projfun M]) z^{[\eta M]}.
$$
Moreover, the restriction $\qnu$ becomes an isomorphism after inverting  
$\Upsilon(\charT_\setJ)=z^{[\charT_\setJ]}$.
\end{theorem}
We will prove \Cref{thm:qmain1} by induction. As noted in \Cref{rem:getridofU}, for cluster variables, the superscripts on $\Upsilon$ and $\upsilon$ can be omitted; however, we will use them in the proof to make the inductive steps clearer.

\begin{proof}[Proof of \Cref{thm:qmain1}]
By the definition of $\qnu$ in \eqref{eq:xi}, 
\[
\Upsilon^T(w)\mapsto \upsilon^{\projfun T}(\projfun(w) + \eta(w)),
\]
for $w\in K(\add T)$, and by induction we assume this is also true for a cluster 
tilting object $U$ reachable from $T$ and any $u\in K(\add U)$. 

Let $U_k$ be a mutable summand of $U$, $U'=\mu_k U$, 
$U_k^*=\mu_k U_k$ and let $E_k$, $F_k$ be the middle terms of the two mutation sequences.
 For any $X\in \Dfiltered(\setJ)$, let \[\ol X=\projfun X.\] 
By induction, we need to show that
\begin{equation}\label{eq:whattoprove}
\Upsilon^{U'}(u)\mapsto \upsilon^{\ol{U'}}(\projfun(u) + \eta (u)), ~\forall u\in K(\add U').
\end{equation}
Note that the quasi-commutation rules among $\upsilon^{\ol{U'}}(\projfun(u))$, $z^{\eta(u)}$
are defined by the bilinear form $\matadj$ (see \eqref{eq:exponents}), which is independent of $U'$. So  
to prove \eqref{eq:whattoprove}, it suffices to prove it for $u=[U'_i]$, the class of an 
indecomposable summand $U'_i$ of $U'$. 

First, consider the case: $i\not=k$.  By the definition of $\Upsilon^{U'}$ \eqref{eq:13.10} and by induction, 
\[
\qnu(\Upsilon^{U'}(u))= \qnu(\Upsilon^{U}(u))
= \upsilon^{\ol U}(\pi(u)+\eta(u)). 
\]
Note that $\ol U_i'\in \add \ol U$. So  following from the definition of $\upsilon^{\ol{U'}}$ and \eqref{eq:exponents}, we have the required equality, 
\[
\qnu(\Upsilon^{U'}(u))= \upsilon^{\ol {U'}}(\pi(u)+\eta(u)).
\]

Next, consider the case: $i=k$. By the definition of  $\Upsilon^{U'}$ (see \eqref{eq:qexchange}) and by induction, we have
\begin{eqnarray*}
&&\qnu(\Upsilon^{U'}([U_k^*]))   \\
&=& \qnu(\Upsilon^{U}([E_k]-[U_k])+\Upsilon^{U}([F_k]-[U_k]))\\
&=& \upsilon^{\ol U}([\ol E_k]-[\ol U_k] + \eta([E_k] - [U_k]))+
\upsilon^{\ol U}([\ol F_k]-[\ol U_k] + \eta([E_k]- [U_k])). 
\end{eqnarray*}
Since the mutation sequences are admissible, 
$$
\eta [U^*_k] = \eta ([E_k] - [U_k]) = \eta([F_k] - [U_k]).
$$
Moreover, in $K(\Dfiltered(\setJ))_{\mathcal{E}}$, which contains $K(\Pi)$ as a subgroup (see \Cref{Lem:isom1}), 
\[ [\ol E_k]-[\ol U_k]=[\ol F_k]-[\ol U_k]=[\ol{U_k^*}].\] 
By \eqref{eq:upsilonexp}  and \eqref{eq:exponents}, 
we have  
\[
\upsilon^{\ol U}([\ol E_k]-[\ol U_k] + \eta([E_k] - [U_k])) 
=q^{\frac{1}{2} \matadj( [\eta U_k^*], [\ol {U^*_k}])}  \upsilon^{\ol U} ([\ol E_k]-[\ol U_k]) 
z^{[\eta U^*_k]}, 
\]
and similarly,
\[
\upsilon^{\ol U}([\ol F_k]-[\ol U_k] + \eta([F_k] - [U_k])) =
q^{\frac{1}{2} \matadj( [\eta U_k^*], [\ol {U^*_k]})}  \upsilon^{\ol{U}} ([\ol F_k]-[\ol U_k]) z^{[\eta U^*_k]}.
\]
Therefore, 
\begin{eqnarray*}
\qnu(\Upsilon^{U'}([U^*_k]))&=& q^{\frac{1}{2} \matadj([\eta U_k^*], [\ol U^*_k])}  (\upsilon^{\ol{U}} ([\ol E_k]-[\ol U_k]) 
+  \upsilon^{\ol{U}} ([\ol F_k]-[\ol U_k])  ) z^{[\eta U^*_k]} \\
&=& q^{\frac{1}{2} \matadj( [\eta U_k^*], [\ol U^*_k])}  (\upsilon^{\ol{U'}} ( \ol{U^*_k})) z^{[\eta U^*_k]} \\
&=& \upsilon^{\ol{U'}} ( [\ol{U^*_k}] +[\eta U_k^*]), 
\end{eqnarray*}
as required. For the second equality, we use the fact  that
$
\ol {U^*_k} = \mu_k (\ol U_k),
$
following from \Cref{Lem:liftexchange}. 
This completes the proof.
\end{proof}

\begin{remark}
Let $\mathcal{A}(\mat B_T)$ and $\mathcal{A}(\mat B_{\projfun T})$ be the cluster 
algebras defined by the matrices $\mat B_T$ and $\mat B_{\pi T}$, respectively. 
We have $\mathcal{A}(\mat B_T)\cong \CC[\flag(\setJ)]$ (see \Cref{Thm:classiccat}), 
and $\mathcal{A}(\mat B_{\projfun T})\cong \CC[N_\setK]$. Therefore, 
\begin{equation}\label{eq:classiciso}
\mathcal{A}(\mat B_T)/\mathcal{I} \cong \mathcal{A}(\mat B_{\projfun T}),
\end{equation}
where $\mathcal{I}$ is the ideal generated by $z^{[\charT_j]}-1$ with $z^{[\charT_j]}$ the cluster 
variable corresponding to $\charT_j$ for $j\in \setJ$.
The homomorphism from \Cref{thm:qmain1},
\[\qnu\colon \qclusalgT\rightarrow \qcl{\mat B_{\projfun T}, \mat L_{\projfun T}}[K(\add \charT_{\setJ})]\] 
 is not  a quantisation of the classical isomorphism \eqref{eq:classiciso}.
Notably,  the quasi-commutation rules for the corresponding cluster variables are different.
\end{remark}

\section{Dual PBW generators for the cluster algebra $\qnilw$}\label{sec14}
Geiss--Leclerc--Schr\"{o}er constructed a set of modules 
$M_{\bfi, k}$  associated to a reduced word of a Weyl group element $w$
such that the cluster variables in $A_q(\mathfrak{n}(w))$ defined by modules $M_{\bfi, k}$
are the PBW generators  \cite{GLS13}. 
For certain $w$, we construct a reduced word $\bfi$ and prove
that the associated modules $M_{\bfi, k}$ all have a simple top. 
This result will be used to prove that $\pqflag$ 
is a quantum cluster algebra in \Cref{sec16}.

\subsection{Construction of cluster tilting modules and $M_{\bfi, k}$} 
Let $W$ be the Weyl group of $\gl_n$. Note that $W$ is isomorphic to 
the symmetric group on $n$ letters.
Denote the simple reflections by $\sigma_i$ for $i\in [n-1]$. 
Let $w\in W$ and 
\[
\mathbf{i}=(i_r,  \cdots,  i_1)
\]
be a reduced word for  $\omega$. 

Recall that 
$\Pi$ is the preprojective algebra of type $\mathbb{A}_{n-1}$, 
$\Pi_j$ is the projective $\Pi$-module at vertex $j$. Let   $R_j=\Pi_{n-j}$ 
be the injective $\Pi$-module with socle $S_j$.
Let $\soc_{(i_1)} M$ be the trace of $S_{i_1}$ in $M$ and let
$\soc_{(i_1, \cdots, i_k)}M$ be the preimage of 
the trace of $S_{i_k}$ in $M/\soc_{(i_{1},  \cdots,  i_{k-1})} M$ under
the quotient map. Let 
\begin{equation}\label{eq:constVC} V_{\bfi, k} = \soc_{(i_k,  \dots,  i_{1} )} R_{i_k}, ~ V_{\mathbf{i}}=\oplus_k V_{\mathbf{i}, k}   ~\text{ and }
\mathcal{C}_{\bfi}=\fac(V_{\bfi}).
\end{equation}
Then $V_\bfi$ is a cluster tilting object in $\fac(V_{\bfi})$ (see \cite[\S 2.4]{GLS11}).

By convention, $V_{\bfi, 0}=0$.
For $1\leq j\leq n-1$, let 
\begin{equation}\label{eq:kj}k_j=\max  \{0\}\cup \{k\colon i_k=j\}.\end{equation}
Define 
\[ \FIng_{\mathbf{i}, j}=V_{\mathbf{i}, k_j} ~\text{ and } ~ \FIng_{\mathbf{i}}= \oplus_{j} \FIng_{\mathbf{i}, j}.\]

\begin{remark}\label{rem:ViVj}
\emph{The category $\mathcal{C}_{\bfi}$ and the module $\FIng_{\bfi}$ are independent of the choice of the reduced word $\mathbf{i}$
for $w$ (see \cite{BIRS, GLS11, GLS13}), and so  
they are also denoted by $\mathcal{C}_\omega$ and $\FIng_\omega$, respectively.}
In contrast, $V_\bfi$ depends on $\bfi$. However,  any two such modules $V_{\bfi}$ and $V_{\bf j}$ are mutation equivalent 
(see \cite[Prop III.4.3]{BIRS} and \cite[\S 3.1]{GLS13}).
The construction of $V_\bfi$ used here follows
\cite{GLS11, GLS13}, the formulation in \cite{BIRS} is different.
An object in $\cat_w$ is \emph{reachable} if it is a summand 
of a cluster tilting object that is mutation equivalent to some $V_{\bfi}$.
\end{remark}

For any $1\leq k\leq r$, let 
\begin{equation}\label{eq:kminus}k^-=\max  \{0\}\cup  \{s\colon s<k, i_s=i_k\}. \end{equation}
There is an embedding $ V_{\bfi, k^-}\to V_{\bfi, k}$
and let 
\begin{equation}\label{eq:Mk} M_{\bfi, k}=V_{\bfi, k}/V_{\bfi, k^-}
\text{ and } \beta_k=\dvector M_{\bfi, k}.
\end{equation}
Then  
\begin{equation}\label{eq:dvMikroot}
\beta_k=
\begin{cases}
\alpha_{i_1} &\text{ if } k=1, \\
\sigma_{i_1}\dots \sigma_{i_{k-1}} (\alpha_{i_k}), & \text{otherwise}.
\end{cases}
\end{equation} 
 is a root \cite[Cor. 9.3]{GLS11}.
Moreover, the modules $M_{\bfi, k}$ are reachable from $V_{\bfi}$ \cite[\S 13.1]{GLS11}.

Denote
by $W_\setK$ the subgroup generated by $\sigma_i$ for $i\in \setK$, by $w_0$ and 
$w_0^{\setK}$  the longest elements in $W$ and $W_\setK$, respectively.
Let  $\Pi_{\setJ}=\oplus_{j\in J} \Pi_j$. 

\begin{example}\label{ex:Mik} The longest element $w_0$ has a reduced word as follows,  
\begin{equation}\label{eq:w00} \mathbf{i}=(i_r, \dots, i_1)=J_{n-1}\dots J_{1},\end{equation}
where $J_i=(1,\dots, i)$.  
Observe that the first time $n-1$ appears in $\bfi$ is in $J_{n-1}$. 
Let $\bfi_{n-1}=J_{n-2}\dots J_1$, which is a reduced word for the longest element in the 
symmetric group on  $n-1$ letters. 
By construction, when $1\leq k\leq r-n+1$,
\[
V_{\mathbf{i}, k}= V_{\mathbf{i}_{n-1}, k} \text{ and } M_{\mathbf{i}, k}= M_{\mathbf{i}_{n-1}, k}.  
\]
As $J_{n-1}=(1, \dots, n-1)$, when $r-n+2\leq k\leq r$,  
\[V_{\bfi, k} = R_{i_k}.\]
Therefore \[\mathcal{C}_{w_0}=\fac \FIng_{\omega_0}=\mod \Pi.\] 
\end{example}

\begin{remark}\label{rem:simtop}
Observe that the modules $M_{\bfi, k}$ in \Cref{ex:Mik} all have a simple top, which is not true 
for an arbitrary reduced word $\bfi$, as illustrated by the example below. 
\end{remark}

\begin{example}\label{ex:Mik2} Let $n=5$ and let $\mathbf{i}$ be 
the reduced word $(3,2,1,4,3)$. 
We have $V_{\bfi, 1}=M_{\bfi, 1}=V_{\bfi, 5^{-}}=S_3$, and $V_{\bfi, 5}$
and $M_{\bfi, 5}=V_{\bfi, 5}/V_{\bfi, 5^-}$ as follows.
\[
\begin{tikzpicture}[scale=0.9,
teal arr/.style={magenta, -angle 60},
blue arr/.style={blue, -angle 60},
magenta arr/.style={teal, -angle 60}, thick] 

\path(10, 4) node (a) {$V_{\bfi, 5}:$};
\path (11, 5) node (g1) {\color{blue}$1$};
\path (13, 5) node (g3) {\color{blue}$3$};
\path (12, 4) node (g2) {\color{blue}$2$};
\path (14, 4) node (g4) {\color{blue}$4$};
\path (13, 3) node (g33) {\color{blue}$3$};

\draw[blue arr] (g1) edge (g2);
\draw[blue arr] (g3) edge (g2);
\draw[blue arr] (g4) edge (g33);

\draw[blue arr] (g2) edge (g33);
\draw[blue arr] (g3) edge (g4);


\path(17, 4) node (b) {$M_{\bfi, 5}:$};
\path (18, 5) node (h1) {\color{blue}$1$};
\path (20, 5) node (h3) {\color{blue}$3$};
\path (19, 4) node (h2) {\color{blue}$2$};
\path (21, 4) node (h4) {\color{blue}$4$};

\draw[blue arr] (h1) edge (h2);
\draw[blue arr] (h3) edge (h2);
\draw[blue arr] (h3) edge (h4);
\end{tikzpicture}
\]
 In particular,  $\top M_{\bfi, 5}=S_1\oplus S_3$, is not simple.
\end{example}

\begin{theorem}\cite[Thm. 2.8, Thm. 2.10, Lem. 17.2]{GLS11}  \label{thm:gls}
Let $w\in W$.
\begin{itemize}
\item[(1)] $\mathcal{C}_{w}=\fac \FIng_{w}$ is Frobenius stably 2-CY. 
\item[(2)] The indecomposable projective-injective objects in $\mathcal{C}_{w}$ are  
$\{\FIng_{\bfi, j}\colon \FIng_{\bfi, j}\not=0\}$. 
\item[(3)] When $w=w_0w_0^\setK$, $\mathcal{C}_{w}=\fac \Pi_\setJ$.
\end{itemize}
\end{theorem}

\subsection{The quantum cluster algebra $\qnilw$ and dual PBW-generators} \label{sec:14.2}
Let   $\mathfrak{n}$
and $\mathfrak{n}_\setK$ be  the  Lie algebras of 
the unipotent subgroups $N$ and $N_\setK$ as defined in \Cref{sec:cat}.
Let $U_q(\n)$ be the positive part of the quantum group $U_q(\mathfrak{g})$. Then 
\[U_q(\n)=\bigoplus_{\alpha\in \rootlat} U_q(\n)_{\alpha},\] where each graded 
component $U_q(\n)_{\alpha}$ is finite dimensional. Let 
\[
\qnil=\bigoplus_{\alpha\in \rootlat}  \Hom_{\CC(q)}(U_q(\n)_\alpha, \CC(q))=
\bigoplus_{\alpha\in \rootlat} \qnil_{\alpha}\leq U_q(\n)^*,
\]
which can be endowed with a multiplication so as to be identified with 
a subalgebra of $U_q(b)^*$ (see \cite[\S 4.1]{GLS13}).
By \cite[Prop 4.1]{GLS13}, there is an isomorphism of algebras
\[\psi\colon U_q(\n)\to \qnil.\]
The algebra $\qnil$ has a set of dual PBW-generators $E^*(\beta)$, one for each 
positive root $\beta$  \cite[\S 7.2]{GLS13}, constructed from the PBW-generators 
for $U_q(\n)$.
Let 
\[\qnilw=\psi(U_q(\n_\setK)).\] 
This is a subalgebra of $\qnil$ generated by $E^*(\beta_k)$, where  $\beta_k$ is 
as defined in \eqref{eq:dvMikroot}. 
The quantum algebra  $\qnilw$ is viewed as the quantum 
deformation of the coordinate ring $\CC[N_\setK]$ as in \cite{GLS13, Kim}. 

We are also interested in the integral form of $\qnilw$ over $\Cq$. 
By a slight abuse of notation, we will continue to denote it by $\qnilw$, 
with the ground ring specified explicitly when needed.

Let $\bfi$ be a reduced word of $w$ and $\qcl{\mat{B}_{{\bfi}}, \mat{L}_{{\bfi}}}$ the quantum cluster algebra defined by  the compatible pair 
$(\mat{B}_{{\bfi}}, \mat{L}_{{\bfi}})=
(\mat{B}_{V_{\bfi}}, \mat{L}_{V_{\bfi}})$  constructed from $V_{\bfi}$ over {$\Cq$}. 
Note that $\qcl{\mat{B}_{{\bfi}}, \mat{L}_{{\bfi}}}$ is independent of the choice 
of the reduced word $\bfi$ and so it can  be denoted by $\qcl{\mathcal{C}_{w}}$.

Let $\mathcal{F}_q(\n(w))$ be  the skew field of fractions of $\qnil$.
There is an embedding
\[\qcl{\mathcal{C}_w} \to \mathcal{F}_q(\n(w))\] 
following \cite[\S 12.1]{GLS13}, 
such that 
\begin{equation}\label{eq:glsmap}
\upsilon(M_{\bfi, k})\mapsto E^*(\beta_k),
\end{equation}
where $\upsilon(M_{\bfi, k})$ is the cluster variable corresponding to $M_{\bfi, k}$.
This embedding  induces an isomorphism of $\Cq$-algebras 
\cite[Cor.11.2.8]{KKKO} (see also \cite[Thm. B]{GY}), 
\begin{equation*}
 \qcl{\mathcal{C}_{w}} \cong A_q(\mathfrak{n}(w)),
\end{equation*}
which improves the original isomorphism over $\CC(q)$ \cite[Thm. 12.3]{GLS13}.
In particular, when $w=w_0w_0^K$, the embedding induces an isomorphism of $\Cq$-algebras,
\begin{equation}\label{eq:isow0w0K}
 \qcl{\fac \Pi_{\setJ}}\to \qnilw.
\end{equation}

\begin{remark}\label{rem:swapfacsub}
Let $\Dual =\Hom_\CC(-, \CC)$. Note that 
the preprojective algebra $\Pi$ is self-dual, that is $\Pi\cong \Pi^{\text{op}}$, and 
\[
\Dual\colon \fac \Pi_\setJ \to \sub Q_\setJ
\]
is a duality. So $\Dual V_{\bfi}$
is a cluster tilting object in $\sub Q_{\setJ}$.  We have
\[\mat B_{V_{\bfi}}=-\mat B_{\Dual V_{\bfi}} \text{ and } \mat L_{V_{\bfi}}=-\mat L_{\Dual V_{\bfi}}.\]
So the map 
\[
\qcl{\fac P_{\setJ}} \to \qcl{\sub Q_{\setJ}}, ~~ \upsilon([V_{\bfi, k}])\mapsto \upsilon([\Dual V_{\bfi, k}]) 
\text{ and } q\mapsto q^{-1},
\]
defines an isomorphism of quantum cluster algebras. Consequently, 
\begin{equation}\label{eq:KKKO}
 \qcl{\sub Q_{\setJ}}\cong \qnilw.
\end{equation}
\end{remark}

\subsection{Construction of modules $M_{\bfi, k}$ with a simple top}\label{sec14.3}
Recall that the length of an element $w\in W$ is denoted by $l(w)$. 
For a word $\bfi$, let $\sigma(\bfi)$ be the corresponding permutation and define the 
length $l(\bfi)$ of $\bfi$ to be the number of letters in $\bfi$. We have 
$l(\sigma(\bfi))\leq l(\bfi)$ and when 
$\bfi$ is reduced, $l(\bfi)=l(\sigma(\bfi))$. 

We define the support of a reduced word $\bfi=(i_r, \dots, i_1)$, denoted by $\supp (\bfi)$, 
the full subgraph of type $\mathbb{A}_{n-1}$ with vertices 
$\{i_j\colon 1\leq j\leq r \}$.
A word $\bfi=(i_s, \dots, i_1)$ is called an {\em interval word} if 
$i_j=i_{j-1}-1$ for all $j$, or $i_j=i_{j-1}+1$ for all $j$. In either case, 
we also write $\bfi=\intw{i_s}{i_1}$. For instance, as interval words 
$\intw{1}{3}=(1,2,3)$, while $\intw{3}{1}=(3, 2, 1)$.
An interval of the form $[a, b]$ with $a>b$ is the empty set, and should not be 
confused with the interval word $\intw{a}{b}$.

Recall the reduced word for $w_0=w_0w_0^{\emptyset}$ in \eqref{eq:w00}. It 
satisfies the following,
\begin{equation}\label{eq:redw0} i\not\in \cup_{j<i} J_j, ~
J_i\backslash\{i\}\subseteq \cup_{j<i} J_i \text{ and } \supp (\cup_{j\leq i} J_j )
\text{ is connected},
\end{equation}
where, by abuse of notation, we view each word $J_i$ as the set consisting 
of the letters appearing in $J_i$.
We will show that such a reduced word exists for $w_0w_0^\setK$ 
 for any $\setK$.
We introduce some notation necessary for the proof. 
Suppose that $\emptyset\not=\setK$. We say that an interval $[a, b]\subseteq \setK$ 
is {\em maximal} if it is not properly contained in any interval subset of $\setK$.  
Decompose $\setK$ into a union of 
maximal interval subsets,  
 \begin{equation}\label{eq:intK}\setK=K_1\cup \dots \cup K_r,\end{equation}
where $K_i<K_{i+1}$ (see \Cref{def:non-crossing}). Suppose that $K_i=[c_i, d_i]$. Then 
$d_i < c_{i+1} -1$,  for all $i$. Let 
\[
t_r=n-d_r-1,~ t_{i}=c_{i}-d_{i-1}-2 \text{ for } 1< i \leq r-1 \text{ and } t_1=c_1-1.
\]
Note that $\sum_i t_i$ is the number of elements in $[n]$ that are fixed by $w_0^\setK$. 
Let \[Y_0=[n-1], K_0=\emptyset \text{ and } L_0=\setK.\]
For any $i\geq 1$, let 
\[Y_i=[c_i, n-1] \text{ and } L_i=\cup_{j\geq i}K_j.\] 

\begin{lemma} \label{prop:w0w0K}
The element $w_0w_0^\setK$ has a reduced word of the form 
\begin{equation}\label{eq:w0} I_{m} \dots I_{1},\end{equation}
where  $I_i=\intw{a_i}{b_i}$ and when $i>1$,  
\begin{equation}\label{eq:redw} b_i\not\in  \cup_{j<i} I_j, \text{ while }
I_i\backslash\{b_i\}\subseteq \cup_{j<i} I_j.
\end{equation}
Consequently, $\supp (\cup_{j\leq i} I_j)$  is connected for any $i$.
\end{lemma}
\begin{proof} 
We will proceed by induction on the number of intervals in \eqref{eq:intK}. That is, for each $i$, a reduced word  as in \eqref{eq:w0} exists for $w_0^{Y_i}w_0^{L_i}$. The lemma follows from the case $i=0$. Note that by definition, $t_i>1$ for each $1<i<r$, however $t_1$ and $t_r$
may be zero, in which case the corresponding step in construction of reduced words 
is skipped.

First, consider the case where $i=r$. 
We have $Y_r=[c_r, n-1]$, $L_r=K_r$.
Define the reduced word $I_j$ for each  $1\leq j\leq t_r$ as follows,  
\begin{equation}\label{eq:intw1}I_j=\intw{c_r}{d_r+j}.\end{equation}
Then   
\[
\sigma(I_{t_r}\dots I_1)w_0^{K_r}=w_0^{Y_r},
\]
and 
\[
l(w_0^{K_r})+\sum_{j=1}^{t_r}l(I_j)= \sum_{i=1}^{d_r-c_r+1}i + \sum_{j=1}^{t_r} (d_r+j+1-c_r) =
\sum_{i=1}^{n-c_r} i=l(w_0^{Y_r}).
\]
So $I_{t_r}\dots I_1$ is a reduced word of $w_0^{Y_r}w_0^{K_r}$ satisfying \eqref{eq:redw}.

Next, assume that we have a reduced word  $I_s\dots I_1$ for  $w_0^{Y_i}w_0^{L_i}$ as claimed. 
Consider the case of $i-1$. 
We have 
\begin{equation}\label{eq:midstep}
w_0^{K_{i-1}}w_0^{Y_{i}}
=\begin{pmatrix} 
c_{i-1} & \dots &d_{i-1}+1 &d_{i-1}+2&\dots &c_{i}-1&c_i &\dots & n \\
d_{i-1}+1 & \dots &c_{i-1} &d_{i-1}+2&\dots &c_{i}-1&n &\dots & c_i 
\end{pmatrix},
\end{equation}
which reverses $[c_{i-1}, d_{i-1}+1]$ and $[c_i, n]$ and fixes each 
number in $[d_{i-1}+2, c_i-1]$.
We construct permutations  which  move in turn  $c_i-1, \dots, c_{i-1}$ 
in the bottom row of \eqref{eq:midstep} to the end of the sequence. More precisely, 
for $1\leq j\leq t_i$,  
let 
\[
I_{s+j}=\intw{n-1}{c_i-j}.
\]
For 
$1\leq j\leq d_{i-1}+2-c_{i-1}$, let 
\[
I_{s+t_i+j}=\intw{n-j}{d_{i-1}+2-j}.
\]
We have \[
t_i+d_{i-1}+2-c_{i-1}=c_i-c_{i-1}.\] 
Note that by construction and the fact that $w_0^{K_{i-1}}$ commutes with $w_0^{Y_i}$, 
\[
w_0^{Y_{i-1}}= \sigma(I_{s+c_i-c_{i-1}}\dots I_{s+1}) w_0^{K_{i-1}}w_0^{Y_i}=
\sigma(I_{s+c_i-c_{i-1}}\dots I_{s+1}) w_0^{Y_i} w_0^{K_{i-1}}.
\]
So  by induction,
\begin{eqnarray*}
w_0^{Y_{i-1}}
&=&\sigma(I_{s+c_i-c_{i-1}}\dots I_{s+1}) \sigma(I_s\dots I_1) w_0^{L_i} w_0^{K_{i-1}} \\
&=&\sigma(I_{s+c_i-c_{i-1}}\dots I_{1}) w_0^{L_{i-1}}.  
\end{eqnarray*}
Also, by construction,  $I_{s+c_i-c_{i-1}},  \dots,  I_1$ satisfy the properties 
\eqref{eq:redw}. So it remains to show that $I_{s+c_i-c_{i-1}}\dots I_1$ is reduced. 
We have 
\begin{eqnarray*}
l(I_{s+c_i-c_{i-1}} \dots I_{s+1})& = &l(I_{s+c_i-c_{i-1}} \dots I_{s+t_i+1}) +l(I_{s+t_i} \dots I_{s+1})  \\
&=&(d_{i-1}+2-c_{i-1})(n-d_{i-1}-1) + \sum_{j=1}^{t_i} (n-c_i+j)
\end{eqnarray*}
and 
\[l(w_0^{K_{i-1}})=\frac{(d_{i-1}+1-c_{i-1})(d_{i-1}+2-c_{i-1})}{2}.
\]
Therefore, 
\[
l(I_{s+c_i-c_{i-1}} \dots I_{s+1})+l(w_0^{K_{i-1}})=
l(w_0^{Y_{i-1}})-l(w_0^{Y_{i}}),
\]
and so $I_{s+c_i-c_{i-1}}\dots I_1$ is reduced. This completes the proof.
\end{proof}

\begin{example}\label{ex:w0w0K}
Let $n=7$, $\setK=\{1,  5, 6\}$ and $\setL=\{1, 3, 4\}$. 
First note that 
\[
w_0^\setK
=\begin{pmatrix} 1&2 &3 &4&5 &6&7 \\
2 & 1 &3 & 4&7 &6&5
\end{pmatrix}.
\]
Then apply the permutation $\sigma(654)$, which is the 4-cycle $(7654)$, to move 4 to the back 
of  row 2 and obtain
\[
\sigma(654) w_0^\setK
=\begin{pmatrix} 1&2 &3 &4&5 &6&7 \\
2 & 1 &3 & 7 &6&5&4
\end{pmatrix}.
\]
Apply in turn the permutations defined by the words $\intw{6}{3}$, $\intw{6}{2}$, 
$\intw{5}{1}$ to move 3, 2, 1 to the back, and thus obtain $w_0$. 
Therefore, 
\[w_0w_0^\setK=\sigma(\intw{5}{1}\intw{6}{2}\intw{6}{3}\intw{6}{4}), \]
which is reduced because it has length $17=l(w_0)-l(w_0^\setK)$. 
Similarly, 
\[\intw{5}{1}\intw{6}{2}\intw{3}{6}\intw{3}{5}\] is a reduced word for  $w_0w_0^\setL$.
Both reduced words satisfy \eqref{eq:redw}. 
In particular, for $w_0w_0^\setL$, the subword $\intw{3}{6}\intw{3}{5}$ 
is precisely the one constructed in the foundation step of the proof of \Cref{prop:w0w0K}.
\end{example}

Let $w=w_0w_0^\setK$ with the reduced word $\mathbf{i}=I_m\dots I_1$ as in  \Cref{prop:w0w0K}.
The graph  $\supp (\mathbf{i})$ is of type $\mathbb{A}_{n-1}$. There are 
$(n-1)$ indecomposable projective-injective modules
$\FIng_{\mathbf{i}, 1}, \dots, \FIng_{\mathbf{i}, n-1}$
 in $\fac \Pi_\setJ$ with 
$\soc \FIng_{\mathbf{i}, j}=S_j$.

We have the dual version of \Cref{Lem:soclesub}.

\begin{lemma} \label{Lem:topfac}
Let $\setJ=\{j_1 < j_2 < \cdots < j_r\}$. 
\begin{itemize}
\item[(1)] If $j< j_1$ then $\top \FIng_{\bfi, n-j}=S_{j_1}$
\item[(2)] If $j\in \setJ$, then  $\top \FIng_{\bfi, n-j}=S_j$.
\item[(3)] If $j_i < j < j_{i+1}$, then $\top \FIng_{\bfi, n-j} = S_{j_i}\oplus S_{j_{i+1}}$.
\item[(4)] If $j>j_r$, then $\top \FIng_{\bfi, n-j} = S_{j_r}$.
\end{itemize}
\end{lemma}

Let $\Pi'$ be the preprojective algebra of type $\mathbb{A}_{n-2}$ with vertices 
$2, \dots, n-1$. Denote by $\mathsf{I}'_j$ the injective $\Pi'$-module at vertex $j$.
By observation, we have the following. 

\begin{lemma}\label{lem:onMk}
The module $\mathsf{I}_{1}$ and  the quotients $\mathsf{I}_j/\mathsf{I}'_j$ for all $j\geq 2$ all have  a
simple top.  
\end{lemma}

As illustrated in \Cref{ex:Mik2},  
the modules $M_{\bfi, k}$, in general, do not have a simple top.
We now show that, 
for a carefully chosen reduced word in 
the case of  $w=w_0w_0^\setK$, they do.

\begin{proposition}\label{prop:simptop} Let $w=w_0w_0^\setK$ and $\bfi=(i_r, \dots, i_1)$ 
be a reduced word  
of the form $I_{m} \dots I_{1}$ as in \eqref{eq:w0}. Then  
the module $M_{\bfi, k}$ has a simple top for any $k$. 
\end{proposition}
\begin{proof} Let 
$\bfi_t=I_t\dots I_1=(i_{r_t}, \dots, i_1)$ for any $t\leq m$, 
where $r_t$ is the length of $\bfi_t$.
%
Note that each $I_i$ is an interval, so  for any $i_a$ with $r_{t-1}<a\leq r_{t}$, 
 \[  V_{\bfi_t, a}=\mathsf{I}_{\bfi_t, i_a}, \]
 which by \Cref{thm:gls} (2) is the indecomposable projective module of
 $\fac \Pi_{\setJ_t}$ with $\setJ_t=(\cup_{j\leq t} I_t)\backslash \setK$. 
Suppose that $r_{t'-1}<a^-\leq r_{t'}$ for some $t'$. Then $V_{\bfi_t, a^-}=\FIng_{\bfi_{t'}, i_a}$ is also projective
in $\fac \FIng_{\bfi_{t'}}$. 
 Moreover, by definition,  $i_a$ does not appear in the intervals $I_x$ for $t'<x<t$. So  
 \[
 \FIng_{\bfi_{t'}, i_a}=\FIng_{\bfi_{t-1}, i_a}.
 \]
 By construction, the support of $\bfi_t$ has exactly one more vertex than that of $\bfi_{t-1}$.
 Therefore $M_{\bfi, a}=V_{\bfi_t, a}/V_{\bfi_t, a^-}$ has a simple top, by \Cref{lem:onMk}.
\end{proof}

Let $\bfi$ be a reduced word for $w_0w_0^K$ as described in \Cref{prop:w0w0K}, and let
 $N_{\bfi, k}$ be the dual of $M_{\bfi, k}$ (see \Cref{rem:swapfacsub}), 
\begin{equation}\label{eq:Nk} 
N_{\bfi, k}= \Dual M_{\bfi, k}.
\end{equation}

\begin{proposition}\label{prop:simsoc}
The modules $N_{\bfi, k}$ for all $k$ are contained in $\sub Q_\setJ$ and 
have a simple socle. 
\end{proposition}

\begin{lemma}\label{lem:ReachComp} Let
$\setJ=[a, b]\subseteq [n-1]$. Then there exists a reduced word $\bfi$ for $w_0w_0^\setK$ such that
\[ 
\Dual V_\bfi=\begin{cases}  \pi T_{\maxNC_{\setJ, b, \square}} & \text{ if } b<n-1,\\ 
\pi T_{\maxNC_{\setJ, a, \square}} & \text{ if } b=n-1.
\end{cases}\]
\end{lemma}

\begin{proof}
When $b<n-1$, a reduced word for $w_0w_0^\setK$
as described in \Cref{prop:w0w0K} is 
\begin{equation}\label{eq:GrRed}
 \intw{n-a}{1}  ...  \intw{n-1}{a} \intw{n-1}{a+1}...\intw{n-1}{b-1}\intw{n-1}{b},
\end{equation}
where the first $a$ intervals are of the form $\intw{n-a+i-1}{i}$ with $1\leq i\leq a$, 
while the other intervals are of the form $\intw{n-1}{a+j}$ with $1\leq j\leq b-a$. 

When $b=n-1$, a reduced word for $w_0w_0^\setK$
as described in \Cref{prop:w0w0K} is
\[
\intw{1}{a}\dots \intw{1}{b},
\]
where the intervals are of the form $\intw{1}{j}$ with $a\leq j\leq b$.

In either case, by observation, $\Dual V_\bfi$ is as described. For instance, 
when $\setJ=[n-1]$, 
\[
\pi {T}_{\setJ, 1, \square}=\Dual V_\bfi.
\]
\end{proof}

\begin{remark} \label{rem:ReachComp2}
For $\Dfiltered(\setJ)$,  we defined reachability in \S \ref{sec:1.4} when $\setJ$ is an interval: 
a module is reachable if it is a summand of a cluster tilting object
that is mutation equivalent to one of the form $T_{\maxNC_{\setJ, k, \square}}$. Now
for $\sub Q_\setJ$, by adapting the definition for $\cat_{\bfi}$ in \Cref{rem:ViVj},  reachable objects are summands of cluster tilting objects 
that are mutation equivalent to some $\Dual V_\bfi$.
By \Cref{rem:reach}, $T\in \Dfiltered(\setJ)$ is a reachable from $T_{\maxNC_{\setJ, k, \square}}$ if and only if $\pi T$ is reachable from $\pi T_{\maxNC_{\setJ, k, \square}}$ in $\sub Q_\setJ$.
Therefore, 
following \Cref{lem:ReachComp}, reachability in $\Dfiltered(\setJ)$ and reachability in $\sub Q_\setJ$ are equivalent. 
\end{remark}

\section{The Gabriel quiver  $Q_{V_{\bfi}}$}
We recall the construction of the Gabriel quiver $Q_{V_{\bfi}}$  of $(\End  V_{\bfi})^{\text{op}}$
for some reduced words and prove properties of $Q_{V_{\bfi}}$ \cite{GLS13}.
In particular, when
$\setK\subset\setK'$, we show that there exists a reduced word $\bfi$ for $w_0w_0^{\setK}$ 
such that $\bfi=\bfi' \bfi_{\setK'}$ with $\bfi_{\setK'}$ a reduced word for $w_0w_0^{\setK'}$
and that the mutable vertices in $Q_{V_{\bfi_{\setK'}}}$, which is a subquiver 
of $Q_{V_{\bfi}}$, are not connected to vertices outside $Q_{V_{\bfi_{\setK'}}}$.

\subsection{More on reduced words for $w_0w_0^{\setK}$.}\label{sec:redw}
\begin{lemma}\label{lem:w0w0K} Let $j\in [n-1]\backslash \setK$ and $\setK'=\setK\cup\{ j\}$.
There exists $w\in W$ such that \[w_0w_0^{\setK}=ww_0w_0^{\setK'} \text{ and }
l(w_0w_0^{\setK})=l(w)+l(w_0w_0^{\setK'}).\]
More precisely, 
\begin{equation}\label{eq:w}w=\left\{ 
 \begin{tabular}{ll} 
$\sigma(\intw{n-b-1}{n-a})$, & if $j=b+1$ and $b+2\not\in \setK$;\\
$\sigma(\intw{n-a+1}{n-b})$, & if $j=a-1$ and $a-2\not\in \setK$;\\
$\sigma(\bfi)$, & if  $j=b+1=c-1$;\\
$\sigma_{n-j}$, & otherwise.
\end{tabular}\right.\end{equation}
where $[a, b], [c, d]$ are maximal interval subsets of $\setK$ with $b<c$, and $\bfi$ is a
reduced word of $w_0^{[n-d, n-a]}w_0^{[n-d, n-a]\backslash\{n-j\}}$ 
as described in \eqref{eq:GrRed}. 
\end{lemma}
\begin{proof}
We prove the existence of $w$ such that $w_0w_0^{\setK}=ww_0w_0^{\setK'}$  only for the first case in  \eqref{eq:w}, as the remaining cases follow by
similar arguments.   We compare the images of $i\in [n]$ under the permutations 
$w_0w_0^{\setK}$ and $w_0w_0^{\setK'}$. When $i\not\in [a, b+2]$, 
\[
w_0w_0^{\setK}(i)=w_0w_0^{\setK'}(i).
\]
Next, consider the restrictions of both permutations to $[a, b+2]$, 
\[
w_0w_0^{\setK}|_{[a, b+2]}=\begin{pmatrix} 
a & \dots & b+1& b+2\\
n-b&\dots &  n-a+1 & n-b-1
\end{pmatrix},
\]
while
\[
w_0w_0^{\setK'}|_{[a, b+2]}=\begin{pmatrix} 
a & \dots &  b+1 & b+2\\
n-b-1&\dots & n-a & n-a+1 &
\end{pmatrix}.
\]
So applying the cycle permutation $w=((n-b-1)(n-b)\dots (n-a+1))$ to $w_0w_0^{\setK'}$
gives $w_0w_0^{\setK}$. Therefore, 
\[
w_0w_0^{\setK}=ww_0w_0^{\setK'},
\]
as claimed. 

Note that   for the first two cases in \eqref{eq:w}
\[l(w_0^{\setK})+(b-a+2)=l(w_0^{\setK'});\]
 for the third case, by \Cref{prop:w0w0K},
 
\[l(w_0^{\setK})+l(w)=l(w_0^{\setK'});\]
and otherwise, 
\[l(w_0^{\setK})+1=l(w_0^{\setK'}).\]
Note that $l(w_0^{\setK''})=l(w_0)-\l(w_0w_0^{\setK''})$ for any 
$\setK''\subseteq [1, n-1]$.
Therefore, we have the length equality as claimed. 
\end{proof}

The following is a direct consequence of \Cref{lem:w0w0K}.

\begin{corollary}\label{newcor}
Let $\setK\subseteq \setK'\subseteq [n-1]$. Then any reduced word  
$\bfi'$ of $w_0w_0^{\setK'}$ can be extended to 
 a reduced word $\bfi=\bfi''\bfi'$ for $w_0w_0^\setK$.
\end{corollary}

\subsection{The Gabriel quiver $Q_{V_{\bfi}}$} 
Let $\bfi=(i_r, \dots, i_1)$ be a reduced word.  Let $k^-$ be defined as in 
\eqref{eq:kminus}, and in addition, let 
\begin{equation}\label{eq:kplus}
k^+=\min \{r+1\}\cup  \{s\colon s>k,~ i_s=i_k\} . 
\end{equation}

Let $C=(c_{ij})$ be the Cartan matrix of type $\typeA_{n-1}$. 
Associated to the reduced word $\bfi$, 
Berenstein--Fomin--Zelvinsky \cite{BFZ} (see also \cite[\S 2]{GLS11}) constructed a quiver
$\Gamma_{\bfi}$ as follows.  The vertices are labelled by $1, \dots, r$. For any $1\leq s, t\leq r$, there is 
one arrow from $s$ to $t$ when $t=s^->0$; and there is
$|c_{i_si_t}|$ arrow from $s$ to $t$, when
$s<t<s^+\leq t^+$ and $i_s\not= i_t$. 
Note that the equality $s^+= t^+$ for $s\not=t$ holds only when both $s^+$ and $t^+$ are equal to $r+1$.
 Since $c_{ij}=0$ or $-1$ when $i\not=j$, it follows by construction that there is at most one arrow 
 between any two vertices in the quiver.
Geiss--Leclerc--Schr\"{o}er \cite[Thm 2.10]{GLS11} proved that $\Gamma_\bfi$ is
 the Gabriel quiver $Q_{V_{\bfi}}$ of  $(\End  V_{\bfi})^{\mathrm{op}}$.

\begin{proposition}\label{prop:GrPartial}
Let $\bfi=\bfi_1\bfi_2$ be a reduced word.  
Then $Q_{V_{\bfi_2}}$ is a (full) subquiver of $Q_{V_\bfi}$. 
Moreover the mutable vertices of $Q_{V_{\bfi_2}}$ 
are not connected to vertices outside $Q_{V_{\bfi_2}}$ 
\end{proposition}
\begin{proof} By construction, $Q_{V_{\bfi_2}}$  is a (full) subquiver of $Q_{V_{\bfi}}$ . 
Let 
\[\bfi_2=(i_r,\dots, i_1) \text{ and } \bfi=(i_s,\dots, i_r, \dots, i_1).\] 
Let $1\leq j\leq r$ be mutable, that is the corresponding 
summand of $V_{\bfi_2}$ is not projective in $\fac \FIng_{\sigma(\bfi_2)}$. 
By the description of the projective-injective objects in \Cref{thm:gls}, 
there exists $j< k\leq r$ such that $i_j=i_k$. By construction, 
 there is no arrow in $Q_{V_\bfi}$ starting from some vertex $l$ with $l>r$ and 
terminating at $j$. Moreover, 
any arrow starting from $j$ terminates at a vertex in the open interval $(j, k)$.
So $j$ is not connected to any vertex outside $Q_{V_{\bfi_2}}$.
\end{proof}

Denote $n-1-|\setK|$ by $k$. Let  $\setK_1=\setK$ and inductively 
$\setK_{s+1}=\setK_{s}\cup \{t_{s}\}$ for some $t_{s}\not\in \setK_{s}$.
We assume that $j$ is not added in any step. That is, 
$\setK_{k}=[n-1]\backslash \{j\}$. Let
\[
w_0w_0^{\setK_{i-1}}=w_{i-1}w_0w_0^{\setK_i}, 
\]
for any $2\leq i\leq k$, where $w_{i-1}=\sigma(\bfi_{i-1})$ is a permutation, as constructed 
in \Cref{lem:w0w0K}. Let
$\bfi_{k}$ be  a reduced word of $w_0w_0^{\setK_{k}}$, in which 
case the corresponding flag variety is the Grassmannian $\Gr(j, n)$.
We may choose $\bfi_k$ such that $\Dual V_{\bfi_k}=\pi T_{\maxNC_{j, \square}}$
(see \Cref{lem:ReachComp}).
Finally, let 
\[
\bfi=\bfi_1\dots \bfi_{k},
\]
which is reduced word for $w_0w_0^{\setK}$.

\begin{corollary}\label{cor:GrPartial}
The mutable vertices of $Q_{V_{\bfi_j\dots \bfi_{k}}}$ 
are not connected to vertices of $Q_{V_\bfi}$ that 
are outside $Q_{V_{\bfi_j\dots \bfi_{k}}}$, for any $1<j\leq k$.
\end{corollary}
\begin{example}
Let $n=7$, $\setJ=\{1, 2, 3, 5\}$. Then $\setK=\{4, 6\}$. 

\noindent (1) We construct the words $\bfi_j$ as in \Cref{lem:w0w0K}. 
We have $\setK_1=\setK$.  Let
 \[\setK_2=\{2, 4, 6\}, \setK_3=\{2, 4, 5, 6\} \text{ and }\setK_4=\{1, 2, 4, 5, 6\}.\]
 
 \begin{itemize}
\item[(a)] 
$\bfi_4=(4,3,2,1,5,4,3,2,6,5,4,3)$, which is a reduced word of $w_0w_0^{\setK_4}$
as described  in \Cref{prop:w0w0K}.

\item[(b)]  $\bfi_3=(6, 5)$, which is  the reduced word for the second case in \Cref{lem:w0w0K} with 
$[a, b]=[2, 2]$ and $j=1$.

\item[(c)] $\bfi_2=(2,1,3,2)$, which is the reduced word for the third case  in \Cref{lem:w0w0K} with
 $[a, b]=[4, 4]$, $j=5$ and $[c, d]=[6, 6]$. Restricting to $[4, 7]$, 
\[
w_0w_0^{\setK_3}|_{[4, 7]}=
\begin{pmatrix}
 4 & 5  & 6 & 7\\
 1 & 2 & 3 & 4
\end{pmatrix}
\text{ and }
w_0w_0^{\setK_2}|_{[4, 7]}=
\begin{pmatrix}
 4 & 5  & 6 & 7\\
 3 & 4 & 1 & 2
\end{pmatrix}.
\]
The permutation $\sigma(\bfi_2)$ maps $1, 2, 3, 4$ to $3, 4, 1, 2$, respectively. 
The word $\bfi_2$ is as described in \Cref{prop:w0w0K}
for $w^{[3]}_0w_0^{\setL}$ with $\setL=\{1, 3\}\subseteq [3]$.  

\item[(d)]  $\bfi_1=(5)$, which follows from the fourth case in \Cref{lem:w0w0K}
with $j=2$. In this situation, $j$ does not extend any (non-empty) interval subset of $\setK_1$.
\end{itemize}
We have the following reduced word of $w_0w_0^{\setK}$,
\[\bfi=\bfi_1\bfi_2\bfi_3\bfi_4= (5, 2,3,1,2, 6,5,  4,3,2,1,5,4,3,2,6,5,4,3).\] 

\noindent (2) The quiver $Q_{V_{\bfi}}$ is as follows.
The full subquiver supported at $[1, 12]$ is $Q_{V_{\bfi_4}}$, with the red vertices 
 the indecomposable projectives. The blue vertices are the mutable vertices in  $Q_{V_{\bfi_4}}$ 
 and are not connected to the black vertices, which do not belong to $Q_{V_{\bfi_4}}$. 
 
\[
\xymatrix{
 & 16\ar[r] \ar[dl]& {\bf\color{red}9}\ar[dl] & \\
 18\ar[r] & 15 \ar[r] \ar[u] \ar[dl]  &{\bf\color{red}10}\ar[dl]  \ar[r] &{\bf\color{blue}5}\ar[dl]\ar[ul]   \\
 17\ar[r]  \ar[u]  & {\bf\color{red}11} \ar[r] \ar[u]  \ar[d]   & {\bf\color{blue}6}\ar[r] \ar[u] \ar[d]& 
 {\bf\color{blue}1} \ar[d]\ar[u]  \\
              & {\bf\color{red}12} \ar[r]  \ar[ul] \ar[dl]  & {\bf\color{blue}7}\ar[d] \ar[ul] \ar[r] & 
              {\bf\color{blue}2} \ar[d] \ar[ul]  &\\
 19 \ar[r] & 13\ar[r] \ar[d]  & {\bf\color{red}8} \ar[ul] \ar[r] & {\bf\color{blue}3} \ar[dl] \ar[ul]  &\\
        & 14 \ar[ul]\ar[r] & {\bf\color{red}4}\ar[ul]
}
\]
\end{example}

\section{Quantum cluster structure on quantum flag varieties}\label{sec16}
We show that the quantum coordinate ring $\pqflag$ is a cluster algebra over 
$\Cq$.

\subsection{Quantum flag varieties}
The quantum coordinate ring $\pqflag$ of the flag variety $\flag(\setJ)$ is 
the quantum deformation of $\CC[\flag(\setJ)]$ and can be viewed as a subalgebra of
the quantum matrix algebra $\CC_q[M_{n\times n}]$ (see  \eqref{eq:matqrel}),  generated 
by the quantum minors $\qminor{I}$ (see \eqref{eq:qminor}) with $|I|\in \setJ$. 
When $\setJ=[n-1]$, we also denote the flag variety $\flag(\setJ)$ by $\flag$.

For any subsets $I, J\subseteq[n]$, 
let \[\inv(I, J)=| \{(i, j)\in I\times J\colon  i>j \}|. \]
The generating relations of $\qflag$ in  \cite{Fioresi, TT} were
reformulated by Leclerc--Zelevinsky 
\cite{LZ} as follows. 
When  $|I|< |J|$, 
\begin{equation}\label{eq:LZrel1}
\qminor{I}\qminor{J}=\sum_{L} (-q)^{\inv(I, L)-\inv(J\backslash L, L)} \qminor{I\cup L} \qminor{J\backslash L},
\end{equation}
summing over $L$ such that $L\subseteq J\backslash I$ and $|L|=|J|-|I|$. 
When $|I|-1\geq |J|+1$,
\begin{equation}\label{eq:LZrel}
\sum_{i\in I\backslash J} (-q)^{\inv(\{i\}, J)-\inv(\{i\}, I\backslash \{i\})} \qminor{I\backslash \{i\}} \qminor{Ji}=0.
\end{equation}
In particular,   
when $I=Labc$, $J=Ld$, and $a<b<c<d$ are not contained in $L$, 
\begin{equation}\label{eq:plucker}
\qminor{Lac}\qminor{Lbd}=q^{-1}\qminor{Lab}\qminor{Lcd}+ q \qminor{Lbc}\qminor{Lad}, 
\end{equation}
and when $I=Labc$, $J=L$ , and $a<b<c$ are not contained in $L$,  
\begin{equation}\label{eq:flip}
\qminor{Lac}\qminor{Lb}=q^{-1}\qminor{Lab}\qminor{Lc}+ q \qminor{Lbc}\qminor{La}.
\end{equation}
The relations \eqref{eq:plucker} are referred to as {\it (short) quantum Pl\"{u}cker relations} and 
those in \eqref{eq:flip} are called {\it (short) quantum incidence relations}.
Note the sign difference of the formulation of \eqref{eq:LZrel1} and  \eqref{eq:LZrel}
from \cite{LZ}. This is due to the choice of sign in the presentation of the quantum matrix algebra
\eqref{eq:matqrel}. 

\subsection{Cluster structure on the complete quantum flag varieties}\label{sec:16.2}

As in the classical case (see \eqref{eq:gradingE}), 
the quantum coordinate ring $\qflag$ decomposes as a graded algebra over 
the Grothendieck group $\grothE$,
$$
\qflag = \bigoplus_{w} \qflag_w
$$
into free $\CC[q^{\frac{1}{2}}, q^{-\frac{1}{2}}]$-modules of  finite rank.
The
specialisation of $\qflag_w$ at $q=1$ gives $\CC[\flag]_w$ (see for instance \cite{TT}). 

\emph{In this subsection, we assume that $\setJ$ is an interval}. Let $\mathcal{S}=\maxNC_{\setJ, k, \square}$
with $k\in \setJ$.
Note that the quantum minors in $\{\qminor{I}\colon I\in \maxNC\}$ 
quasi-commute with the commutation rules given by $c(I,J)$ for any $I, J\in \maxNC$ (see \Cref{Thm:LZ}), and 
are algebraically independent when $q=1$. 
Let 
\[T=T_{\maxNC} \text{ and }(\mat B, \mat L)=(\mat B_T, L_T).\]
We have $\mat L=(c(I, J))_{I, J\in \maxNC}$ 
(see \Cref{Thm:quasicomm} or  \Cref {rem:13.3}).
Therefore, there is an injective map
\begin{equation}\label{eq:noname}
\qclusalg\rightarrow \CC_q(\flag), ~ \Upsilon(M_I)\mapsto \qminor{I}, 
\end{equation}
where  $ \CC_q(\flag)$ is the skew-field of fractions of  $\CC_q[\flag]$. 
Note that $\qflag$  has polynomial growth, 
so it is an Ore domain (see \cite{BZ, IM}) and thus
$\CC_q(\flag)$ is well-defined. 

\begin{proposition}\label{prop:flipqrel} 
Let $L\subseteq [n]$.
\begin{itemize}
\item[(1)] 
Suppose that $\{a, b, c, d\}\subseteq [n]\backslash L$ are cyclically ordered with $|L|+2\in \setJ$.
Then the mutation relation defined by  the mutation sequences \eqref{eq:mutseq1} and  \eqref{eq:mutseq2} 
is the quantum relations \eqref{eq:plucker}, that is,
\[
\Upsilon([M_{Lac}]) \Upsilon([M_{Lbd}])=q^{-1} \Upsilon([M_{Lab}]) \Upsilon([M_{Lcd}]) + q \Upsilon([M_{Lbc}])\Upsilon([M_{Lad}]).
\]

\item[(2)] Suppose that $\{a< b< c\} \subseteq [n]\backslash L$ with $|L|+1, |L|+2\in \setJ$.
Then the mutation relation defined by the mutation sequences
\eqref{eq:flipseq1} and \eqref{eq:flipseq2}   is the quantum relation \eqref{eq:flip}, that is,
\[
\Upsilon([M_{Lac}]) \Upsilon([M_{Lb}])=q^{-1} \Upsilon([M_{Lab}]) \Upsilon([M_{Lc}]) + q \Upsilon([M_{Lbc}])\Upsilon([M_{La}]).
\]
\end{itemize}
\end{proposition}
\begin{proof}
We only prove (2), as (1) can be similarly proved (see also \cite[Prop 9.3]{JKS2}).
By \eqref{eq:qexchange} and the mutation sequences   \eqref{eq:flipseq1} and \eqref{eq:flipseq2}, we have 
\begin{eqnarray*}
\Upsilon([M_{Lbd}]) &= &\Upsilon([M_{Lab}]+[M_{Lc}]-[M_{Lb}]) + 
\Upsilon([M_{Lbc}]+[M_{La}]-[M_{Lb}]) \\
&= & q^{x} \Upsilon([M_{Lab}]) \Upsilon([M_{Lc}]) \Upsilon([M_{Lb}])^{-1} + q^{y}\Upsilon([M_{Lbc}])  
\Upsilon([M_{La}]) \Upsilon([M_{Lb}])^{-1}
\end{eqnarray*}
Now following the quasi-commutation rules (see \Cref{Thm:quasicomm} and \Cref{rem:13.3}), 
we have $x=-1$ and $y=1$. So we have the equality as required. 
\end{proof}

\begin{proposition}\label{prop:overfield}
Via the inclusion \eqref{eq:noname}, we have 
\begin{equation}\label{eq:flagemb}
\pqflag \subseteq \qclusalg 
\end{equation}
as subalgebras of $\CC_q(\flag(\setJ))$ such that  $\qminor{I}=\Upsilon([M_I])$ for all $I$.
Moreover, 
$$
\pqflag\otimes_{\CC[q^{\frac{1}{2}}, q^{-\frac{1}{2}}]} \CC(q^{\frac{1}{2}}) = \qclusalg \otimes _{\CC[q^{\frac{1}{2}}, q^{-\frac{1}{2}}]}  \CC(q^{\frac{1}{2}}).
$$
\end{proposition}

\begin{proof}
Since $\pqflag$ is generated by the quantum minors $\qminor{I}$, to prove \eqref{eq:flagemb}, it suffices to show that $\qminor{I}\in \qclusalg$ for all $I$.
By definition, it is true for the initial minors $\{\qminor{I}\colon I\in \mathcal{S}\}$.
Following  \Cref{prop:flipqrel}, mutation relations corresponding to geometric
exchanges are  quantum Pl\"{u}cker relations, 
and mutation relations corresponding to flips are incidence relations. 
Now by \Cref{Thm:reachability}, any $M_I$ is reachable by geometric exchanges and flips, 
therefore any  quantum minor $\qminor{I}$ is a cluster variable in $\qclusalg$ and so we have the inclusion as claimed. 

Note that $\pqflag$ and $\qclusalg$ are flat deformations of their classical versions at $q=1$, see for instance \cite[Thm. 3.5 (d)]{TT} and \cite[Thm. 5.1]{GLS20}. In particular, they are free $\CC[q^{\frac{1}{2}}, q^{-\frac{1}{2}}]$-modules.
\Cref{Thm:classiccat} implies that 
the inclusion \eqref{eq:flagemb} becomes an equality 
when  $q=1$, that is,    $$\CC[\flag(\setJ)]=\mathcal{A}(\mat B),$$ 
where $\mathcal{A}(\mat B)$ is the specialisation of $\mathcal{A}(\mat B, \mat L)$
at $q=1$.
For each graded component, 
 \[\CC[\flag(\setJ)]_{\omega}=\mathcal{A}(\mat B)_\omega, \]
which implies $\pqflag_{\omega}\subseteq \qclusalg_{\omega}$ have the same rank. Therefore 
$$
\pqflag\otimes_{\CC[q^{\frac{1}{2}}, q^{-\frac{1}{2}}]} \CC(q^{\frac{1}{2}}) = \qclusalg \otimes _{\CC[q^{\frac{1}{2}}, q^{-\frac{1}{2}}]}  \CC(q^{\frac{1}{2}}).
$$
This completes the proof.
\end{proof}

The following result shows that the quantum coordinate ring of the 
full flag variety is a cluster algebra. 

\begin{theorem} \label{thm:pfl1} Let $\setJ=[n-1]$ and 
$T=T_{\maxNC_{\setJ, k, \square}}$. 
Then 
\[\chi\colon \qflag \to \qcl{\mat B_T, \mat L_T}, ~\qminor{I}\mapsto \Upsilon([M_I]), 
\forall I\subseteq [n] \text{ and } I\not=[n].\]
is an isomorphism of $\Cq$-algebras.
 \end{theorem}

\begin{proof}
By \Cref{Thm:reachability}, the extended rectangle cluster tilting objects are mutation 
equivalent. We may assume that $T=T_{\maxNC_{\setJ, 1, \square}}$, in which case 
the labels of the summands of $T$ are all intervals. 
Denote $\projfun T$ by $\ol T$ and let $(\ol{\mat B}, \ol{\mat L})=(\mat B_{\ol T}, \mat L_{\ol T}).$
Observe that \[\ol T=V_{\bfi} \text{ with } \bfi=\intw{1}{n-1}\dots \intw{1}{2} \intw{1}{1}.\]
Following the isomorphism in \eqref{eq:isow0w0K} for the case $\setJ=[n-1]$ and 
\Cref{rem:swapfacsub}, we have an isomorphism, 
\[
  \qclusalgbar \to \qnil
\]
such that 
\[\upsilon(N_{\bfi, k})\mapsto E^*(\beta_k) \text{ and } q\mapsto q^{-1}.\]  
It induces an isomorphism of the twisted Laurent polynomial rings, 
\[
 \gamma\colon \qnil [K(\add \charT_\setJ)] \to  \qclusalgbar[K(\add \charT_\setJ)]. 
\]
Note that, for $j\in \setJ$ and $u\in K(\mod \Pi)$, the quasi-commutation rules between $z^{[\charT_j]}$ and any $x$ in a homogeneous component $ \qclusalgbar_{u}$  is defined by
\eqref{eq:exponents}, and $\gamma$ induces the quasi-commutation rules between 
$z^{\charT_j}$ and $\gamma^{-1}(x)$.
There is a unique map $\delta$ such that the following diagram commutes,
\begin{equation}\label{diag}
\begin{tikzpicture}[scale=0.7, baseline=(bb.base)]
\pgfmathsetmacro{\eps}{0.2}
\coordinate (bb) at (0,0);
 
\path (1,1.5) node (c1) {$\qflag$};
\path (7.5,1.5) node (c2) {$ \qclusalg$};
\path (1,-1.5) node (c3) {$\qnil[K(\add \charT_\setJ)] $};
\path (7.5,-1.5) node (c4) {$ \qclusalgbar)[K(\add \charT_\setJ)],$};

\draw[->]  (c1) edge node[auto]{$\chi$} (c2);
\draw[->]  (c1) edge node[left]{$\mapfl$}  (c3);
\draw[->]  (c2) edge node[auto]{$\qnu$} (c4);
\draw[->]  (c3) edge node[auto]{$\gamma$} (c4);
\end{tikzpicture}
\end{equation}
where $\chi$ is the embedding \eqref{eq:flagemb},  and 
$\qnu$ is defined in \eqref{eq:xi}. That is,
\begin{equation}\label{eq:mapfl}\mapfl=\gamma^{-1} \cdot \qnu \cdot \chi.\end{equation} 
Then $\mapfl$ is injective, since it is a composition of injective maps, and 
\begin{equation}\label{eq:beta2}
\mapfl(\qminor{[j]}) =z^{[\charT_j]}, \text{ for all } j.
\end{equation}

As observed in \Cref{rem:simtop}, the modules $M_{\bfi, k}$ all have a simple top. Hence
the duals $N_{\bfi, k}$ (see  \eqref{eq:Nk}) all have a 
simple socle and so their minimal lifts to $\Dfiltered$ are rank one modules of the form $M_{I_k}$.
By \Cref{prop:overfield}, 
\[\chi(\qminor{I_k})= \Upsilon([M_{I_k}]).\] 
So
\begin{equation}\label{eq:beta1}
\mapfl(\qminor{I_k})=\gamma^{-1}(\upsilon([N_{\bfi, k}])z^{[\charT_{a}]})=
E^*(\beta_k)z^{[\charT_{a}]},
\end{equation}
where $a=|I_k|$ and the last equality follows from the definition of $\gamma$.

Note that $\Upsilon([\charT_j])$ quasi-commutes with any cluster variable 
$\Upsilon([M])$ and $\qminor{[j]}$ quasi-commutes with all 
quantum minors. Therefore
we can localise $ \qclusalg$ and $\qflag$ by inverting 
$\{\Upsilon([\charT_j])\colon j\}$ and 
$\{\qminor{[j]}\colon j\}$, respectively. 

After the  localisation,  $\qnu$ becomes an isomorphism, 
since \[\Upsilon([M_{[j]}]) \mapsto z^{[\charT_j]}.\]
Moreover, 
$\mapfl$ becomes surjective, since  all the generators 
$E^*(\beta_k)$ and $z^{\pm [\charT_i]}$ are contained in the image. 
Therefore $\mapfl$ becomes an isomorphism.
Consequently, so does $\chi$.   

By \Cref{prop:overfield}, $\chi$ is also an isomorphism after extending 
the ground ring to $\CC(q^{\frac{1}{2}})$. We show that
any cluster variable $x$ in $ \qclusalg$ is contained in the image of $\gamma$
without localisation, 
and thus conclude that $\chi$ is an isomorphism of $\Cq$-algebras.
First note the following. 
\begin{itemize}
\item[(i)] By \Cref{prop:overfield}, there exist a polynomial $f\in \Cq$ and a polynomial 
 $p_1$ in the quantum minors such that 
\[ f x = \chi(p_1).\]
\item[(ii)] Since $\mapfl$ is an isomorphism after localisation, there exist a module
$\charT'=\oplus_j \charT_{i_j}$ in $ \add\charT$ and a polynomial $p_2$ in the quantum minors  such that 
\[
x z^{[\charT']} =\chi( p_2).
\]
\end{itemize}
As $\chi$ is an injective map, 
\[
p_1\prod_j\qminor{[i_j]}=fp_2\in \qflag.
\]
Writing both sides as a linear combination of standard monomials (see \cite{LR, TT}) and 
by comparing the coefficients, we see that the polynomial $f$ divides $p_1$. 
Therefore \[x=\chi(f^{-1}p_1)\in \im \chi,\] as required. This completes the proof.
\end{proof}

\subsection{Cluster structure on the partial quantum flag varieties}
In this subsection, $\setJ$ is an arbitrary subset of $[n-1]$. We show that there is a 
quantum cluster structure on the quantum coordinate ring 
$\pqflag$. We view $\pqflag$ as a subalgebra of $\qflag$ generated 
by the quantum minors $\qminor{I}$ with $|I|\in \setJ$.

\begin{proposition} \label{prop:GabrSub}
Let $\bfi=\bfi_1\bfi_2$ be a reduced word and  $M$ a reachable rigid indecomposable module 
in $\fac(\FIng_{\bfi_2})$. Then $M$ is also reachable in $\fac(\FIng_{\bfi})$ and 
the cluster variables in $\qcl{\mathcal{C}_{\bfi_2}}$ and $\qcl{\mathcal{C}_{\bfi}}$
associated to $M$ are the same, that is,  
\[ \upsilon^{V_{\bfi}}([M]) =\upsilon^{V_{\bfi_2}}([M]). \]
Consequently, 
the cluster algebra $\qcl{\mathcal{C}_{\bfi_2}}$ is a subalgebra of 
$\qcl{\mathcal{C}_{\bfi}}$
\end{proposition}
\begin{proof}
By \Cref{prop:GrPartial}, mutations of $Q_{V_{\bfi_2}}$
are mutations in $Q_{V_{\bfi}}$. 
Moreover, 
such a mutation does not introduce new arrows between vertices 
outside $Q_{V_{\bfi_2}}$ and its mutable vertices. Now, as 
the mutation of $V_\bfi $ at a mutable summand is compatible with the mutation of $Q_{V_\bfi}$ 
at the corresponding mutable vertex (see \cite[Prop. 10.1 and 10.2]{GLS13}), 
 mutations in $\qcl{\mathcal{C}_{\bfi_2}}$ are mutations 
in $\qcl{\mathcal{C}_{\bfi}}$. Therefore, the two cluster variables 
associated to $M$ are the same.
\end{proof}

We will soon see that the analogous result holds for $\Dfiltered(\setJ)$.
Let $U_\bfi \in \Dfiltered(\setJ)$ be the lifted cluster tilting object of $\Dual V_\bfi$. 
We say that an object in  $\Dfiltered(\setJ)$ is \emph{reachable} if 
it is a summand of a cluster titling object that is mutation equivalent to 
some ${U}_\bfi$. This definition  is consistent with the case where $\setJ$ is an interval 
(see in \Cref{rem:ReachComp2}).
In that setting, 
\Cref{Thm:reachability} says that any rank one module $M_I\in \Dfiltered(\setJ)$
is reachable. We  now extend this result to the present setting of an arbitray $\setJ$, as 
a special case of the following result. 

\begin{proposition} \label{cor:towersofalg2}
Let $\setJ_2\subseteq \setJ_1\subseteq [n-1]$. 
Any reachable module 
in $\Dfiltered(\setJ_2)$ is also reachable in $\Dfiltered(\setJ_1)$. 
Consequently, any rank one module in $\Dfiltered(\setJ)$ is reachable.
\end{proposition}
\begin{proof} Let $\setK_i=[n-1]\backslash \setJ_i$ for $i=1, 2$. Then 
$ \setK_1\subseteq \setK_2$.
By \Cref{newcor}, there are reduced words $\bfi_1$ of $w_0w_0^{\setK_1}$, $\bfi_2$
of $w_0w_0^{\setK_2}$ and $\bfi$ such that $\bfi_1=\bfi \bfi_2$.

Applying \Cref{Lem:liftexchange} with $T=\liftV_{\bfi_2}$,  the mutation sequences for a mutable summand $T_k\in \add T$
in \eqref{Eq:exseq1} and \eqref{Eq:exseq2} are lifts of the mutation sequences of 
$\Dual V_{\bfi_2}$ at $\pi T_k$, where  $E_k$ (resp. $F_k$) and the minimal 
lift of $\pi E_k$ (resp. $\pi F_k$), up to some summands in $\add \charT_{\setJ_2}$.
Now, as we have seen in the proof of \Cref{prop:GabrSub},   mutations in the category $\sub Q_{\setJ_2}$ are mutations in $\sub Q_{\setJ_1}$.
Therefore, mutations in $\Dfiltered(\setJ_2)$ are  mutations 
in $\Dfiltered(\setJ_1)$. Thus reachable objects in $\Dfiltered(\setJ_2)$ are 
reachable in $\Dfiltered(\setJ_1)$.

Next, for an arbitray $\setJ$, suppose $M_I\in \Dfiltered(\setJ)$. Let 
$j=|I|$ and $\setJ_2=\{j\}$. Then
$\setJ_2 \subseteq \setJ$,
$\setJ_2$ is an interval and $M_I\in \Dfiltered(\setJ_2)$.
By \Cref{Thm:reachability}, 
$M_I$ is reachable in $\Dfiltered(\setJ_2)$. 
Therefore, applying the result just established for the inclusion 
$\setJ_2\subseteq \setJ$, we can conclude that 
$M_I$ is reachable in $\Dfiltered(\setJ)$.
\end{proof}

\begin{remark}
For any two such reduced words $\bfi$ and $\bfj$, the cluster tilting objects 
$\liftV_\bfi$ and $\liftV_\bfj$ are mutation equivalent, 
because $V_\bfi$ and $V_\bfj$ are 
mutation equivalent (see \Cref{rem:ViVj}), and mutations in  $\Dfiltered(\setJ)$
are lifts of mutations in $\sub Q_\setJ$, as explained in the proof of  \Cref{cor:towersofalg2}.
Consequently, 
the definition of reachability in $\Dfiltered(\setJ)$ is independent of the choice of 
the reduced word $\bfi$ for $w_0w_0^\setK$, and  
\[
\qcl{\mat{B}_{\liftV_\bfi}, \mat{L}_{\liftV_\bfi}}=\qcl{\mat{B}_{\liftV_\bfj}, \mat{L}_{\liftV_\bfj}}. 
\]
In the spirit of the discussion in \S \ref{sec:14.2}, we  may also denote 
the quantum cluster algebra constructed from some $\liftV_\bfi$
by $\qcl{\Dfiltered(\setJ)}$.
\end{remark}

\begin{corollary}\label{cor:compqclalg}
Let $\setJ_2\subseteq \setJ_1\subseteq [n-1]$. Let $M$ be indecomposable and
reachable in $\Dfiltered(\setJ_2)$. Then
the cluster variables $\Upsilon_1(M) \in \qcl{\Dfiltered(\setJ_1)}$ and $\Upsilon_2(M)\in \qcl{\Dfiltered(\setJ_2)}$ coincide. 
Consequently, 
\[\qcl{\Dfiltered(\setJ_2)}\leq\qcl{\Dfiltered(\setJ_1)}.\]
\end{corollary}
\begin{proof} As explained in the proof of \Cref{cor:towersofalg2},
mutation sequences in $\Dfiltered(\setJ_1)$ are mutation sequences in $\Dfiltered(\setJ_2)$. Therefore, 
the cluster variables in $\qcl{\Dfiltered(\setJ_1)}$ and $\qcl{\Dfiltered(\setJ_2)}$
associated to $M$ are the same. 
As the quantum cluster algebras are generated by cluster variables, we have
the claimed inclusion of algebras.
\end{proof}

\begin{theorem} \label{thm:qmain2} Let $\setJ\subseteq [n-1]$. Then 
\begin{equation}\label{eq:pgflagJ}
\chi_{\setJ}\colon \pqflag \to \qcl{\Dfiltered(\setJ)}, ~\qminor{I}\mapsto \Upsilon^T([M_I]), 
\end{equation}
is a $\Cq$-algebra isomorphism. 
 \end{theorem}

\begin{proof} Except for working with a general set $\setJ\subseteq [n-1]$, we use the same notation as 
in the proof of  \Cref{thm:pfl1}. 

For any $I$ with $|I|\in \setJ$, $M_I$ is reachable (see \Cref{cor:towersofalg2}), and 
$\qcl{\Dfiltered(\setJ)} $
is a cluster subalgebra of $\qcl{\Dfiltered}$ (see  \Cref{cor:compqclalg}). Hence, 
$\chi (\qminor{I})=\Upsilon([M_I]) $ belongs to  $\qcl{\Dfiltered(\setJ)}$.  Consequently,
 the restriction of $\chi$ to $ \pqflag$ defines a homomorphism, 
\[
\chi_{\setJ}\colon \pqflag \to \qcl{\Dfiltered(\setJ)}.
\]
Similar to the case where $\setJ$ is an interval, 
$\chi_{\setJ}$ is an isomorphism after extending the ground ring $\Cq$ to $\CC(q)$, 
\begin{equation}\label{eq:pgfloverfield}
\pqflag\otimes_{\Cq} \CC(q) \cong 
\qcl{\Dfiltered(\setJ)} \otimes_{\Cq} \CC(q).
\end{equation}

Next, by \Cref{prop:simsoc}, there exists a reduced word $\bfi$ for $w_0w_0^\setK$ such that  
$N_{\bfi, k}$ has a simple socle, 
and so can be lifted to a rank one module $M_{I_k}$. 
Since  $\qnilw$ is generated by $E^*(\beta_k)$ (see \Cref{sec:14.2}), 
for all $k$, 
$\qnilw[K(\add \charT_{\setJ})]$ is generated by  $E^*(\beta_k)$ and $z^{\pm[\charT_j]}$ 
for all $k$, and  $j\in \setJ$.
Denote by $\delta_\setJ$ the restriction of $\delta$ to $\pqflag$.
By  \eqref{eq:beta1},
\[
 \delta_\setJ(\qminor{I_k})= E^*(\beta)z^{[\charT_{a}]},
\]
which is containd in  $\qnilw[K(\add \charT_{\setJ})]$, 
where $a=|I_k|$.
So $\delta_{\setJ}$ is a homomorphism, 
\[
\pqflag \to \qnilw[K(\add \charT_{\setJ})].
\]

Localise $\pqflag$ at the minors $\{\Delta^q_{[j]}\colon j\in \setJ\}$
and $\qcl{\Dfiltered(\setJ)}$ at the frozen cluster variables in 
$\{z^{[\charT_j]} \colon \charT_j\in \add \charT_\setJ\}$.
Then, as in the case of $\setJ=[n-1]$, the maps 
$\delta_\setJ$ and $\chi_\setJ$
become isomorphisms. 
Now, by similar arguments for proving  \Cref{thm:pfl1},  
$\chi_{\setJ}$  is an isomorphism as claimed.
\end{proof}

\begin{remark}\label{rem:finale}
(1) The map $\delta_\setJ$ is a quantum analogue of $\hat{\rho}$ in \eqref{eq:rhohat}.

(2) As in the classical case, $n$ can be included in $\setJ$. In this case, the additional generator $\qminor{[n]}\in \pqflag$ is central in
 the algebra. In parallel,   $\charT_n$ is a summand of any cluster tilting object $T$ in 
$\Dfiltered(\setJ)$ and $\Upsilon(\charT_n)=z^{[\charT_n]}$ commutes with any cluster 
variable $\qcl{\mat B_T,\mat L_T}$, and $\chi_\setJ(\qminor{[n]})=z^{[\charT_n]}$.
So Theorems \ref{thm:pfl1} and \ref{thm:qmain2} still hold.
\end{remark}

\begin{remark} \label{rem:finrem} 
Following \Cref{rem:finale} and  the proof of \Cref{thm:qmain2}, we can extend 
 the diagram \eqref{diag}  to the following commutative diagram.
\[
\begin{tikzpicture} [scale=0.7, baseline=(bb.base)] 
 \pgfmathsetmacro{\eps}{0.2}
\coordinate (bb) at (0,0);
 
\path (-2,2) node (c1) {$\qflag$};
\path (5.5,2) node (c2) {$\qcl{\Dfiltered}$};
\path (-2,-2) node (c3) {$\qnil[K(\add \charT)] $};
\path (5.5,-2) node (c4) {$\qcl{\mod \Pi}[K(\add \charT)]$};
\draw[thick, ->]  (c1) edge node[auto]{} (c2);
\draw[thick, ->]  (c1) edge node[left]{}  (c3);
\draw[thick, ->]  (c2) edge node[auto]{} (c4);
\draw[thick, ->]  (c3) edge node[auto]{} (c4);
\path[teal] (2,4) node (d1) {$\pqflag$};
\path[teal] (9.5,4) node (d2) {$\qcl{\Dfiltered(\setJ)}$};
\path[teal] (2,0) node (d3) {$\qnilw[K(\add \charT_{\setJ})]$};
\path[teal] (9.5,0) node (d4) {$\qcl{\sub {Q_{\setJ}}}[K(\add \charT_{\setJ})]$};

\draw[teal, thick, ->]  (d1) edge node[auto]{} (d2);
\draw[teal, thick, ->]  (d1) edge node[left]{}  (d3);
\draw[teal, thick, ->]  (d2) edge node[auto]{} (d4);
\draw[teal, thick, ->]  (d3) edge node[auto]{} (d4);

\draw[magenta, thick, ->]  (d1) edge node[auto]{} (c1);
\draw[magenta, thick, ->]  (d2) edge node[left]{}  (c2);
\draw[magenta, thick, ->]  (d3) edge node[auto]{} (c3);
\draw[magenta, thick, ->]  (d4) edge node[auto]{} (c4);
\end{tikzpicture}
\] 
Here, all the magenta arrows are inclusions,  and 
the algebras in the front black rectangle are constructed using
the full set $[n]$.
\end{remark}

\noindent{\bf Acknowledgements}: 
We are grateful to the algebra group at Huaqiao University, China, for their hospitality during 
our visits in 2018-2019,  when part of this work was carried out. 
We also thank Thomas Br\"{u}stle for pointing out the connection, discussed in  \Cref{rem:TBr},  between \Cref{Lem:exactcategory} and \cite[Prop 1.10]{DRSKeller}, and Matthew Pressland for explaining the relationship between \Cref{prop:main} and \cite[Thm. 6.22]{GP}. 

\end{document}